\documentclass{llncs}

\usepackage{hyperref}
\usepackage{stmaryrd}
\usepackage[leqno]{amsmath}
\usepackage{amssymb}
\usepackage{xspace}
\usepackage{graphicx} 
\usepackage{bussproofs}\EnableBpAbbreviations
\usepackage{tikz}
\usepackage{float}
\usepackage{xifthen}%
\usepackage{url}
\usepackage{cancel}

\usepackage{enumitem}
\setlist{leftmargin=*}
\setitemize[2]{label=$\bullet$}
\setenumerate[1]{label=(\arabic*), ref=(\arabic*)}

\usepackage[color=green!20,linecolor=red]{todonotes}

\usepackage[linesnumbered,commentsnumbered,figure]{algorithm2e}

\SetAlFnt{\small}
\SetAlCapFnt{\small}
\SetAlCapNameFnt{\small}
\SetAlCapHSkip{0pt}
\IncMargin{-\parindent}

\newcommand{\mathsc}[1]{{\normalfont\textsc{#1}}}

\newcommand{\emptySet}{\cdot}
\newcommand{\seqsArrow}{\vdash}
\newcommand{\seqs}[2]{
  \ifthenelse{\isempty{#1}}{\emptySet}{#1} 
  \seqsArrow
  \ifthenelse{\isempty{#2}}{\emptySet}{#2} 
}
\newcommand{\Lhs}[1]{\mathrm{Lhs}(#1)} 
\newcommand{\Rhs}[1]{\mathrm{Rhs}(#1)} 

\newcommand{\IPL}{\mathrm{IPL}\xspace}

\newcommand{\FRJof}[1]{\mathbf{FRJ}(#1)\xspace} 
\newcommand{\GBUof}[1]{\mathbf{Gbu}(#1)\xspace} 

\newcommand{\seqfrjArrow}{\Rightarrow}
\newcommand{\seqfrj}[2]{
  \ifthenelse{\isempty{#1}}{\emptySet}{#1} 
  \seqfrjArrow
  \ifthenelse{\isempty{#2}}{\emptySet}{#2} 
}
\newcommand{\tseqfrj}[2]{
  \ifthenelse{\isempty{#1}}{\emptySet}{#1} 
  &\;\seqfrjArrow\; 
  \ifthenelse{\isempty{#2}}{\emptySet}{#2} 
}

\newcommand{\seqfrjiArrow}{\rightarrow}
\newcommand{\seqfrji}[3]{
  \ifthenelse{\isempty{#1}}{\emptySet}{#1} 
  \,;\,
  \ifthenelse{\isempty{#2}}{\emptySet}{#2} 
  \seqfrjiArrow  
  \ifthenelse{\isempty{#3}}{\emptySet}{#3} 
}
\newcommand{\tseqfrji}[3]{
  \ifthenelse{\isempty{#1}}{\emptySet}{#1} 
  \:;\:
  \ifthenelse{\isempty{#2}}{\emptySet}{#2} 
  & \; \seqfrjiArrow\;  
  \ifthenelse{\isempty{#3}}{\emptySet}{#3} 
}

\newcommand{\ruleAXI}{\mathrm{Ax}_\rightarrow\xspace}
\newcommand{\ruleAXR}{\mathrm{Ax}_\Rightarrow\xspace}
\newcommand{\ruleJOIN}{\Join\xspace}
\newcommand{\ruleJOINA}{\Join^{\mathrm{At}}\xspace}
\newcommand{\ruleJOINO}{\Join^\lor\xspace}
\newcommand{\ruleIMPi}{\imp_\in\xspace}
\newcommand{\ruleIMPni}{\imp_{\not\in}\xspace}

\newcommand{\seqgbuArrow}{\Rightarrow_g}
\newcommand{\seqgbu}[2]{
\ifthenelse{\isempty{#1}}{\emptySet}{#1} 
\seqgbuArrow
\ifthenelse{\isempty{#2}}{\emptySet}{#2} 
}

\newcommand{\seqgbuiArrow}{\rightarrow_g}
\newcommand{\seqgbui}[2]{
\ifthenelse{\isempty{#1}}{\emptySet}{#1} 
\seqgbuiArrow
\ifthenelse{\isempty{#2}}{\emptySet}{#2} 
}

\newcommand{\ruleAXID}{\mathrm{Ax}\xspace}
\newcommand{\ruleIMPRi}{\imp R_\in\xspace}
\newcommand{\ruleIMPRni}{\imp R_{\not\in}\xspace}

\newcommand{\imp}{\supset}
\newcommand{\proves}[2]{\vdash_{#1} #2}
\newcommand{\nproves}[2]{{\not\vdash}_{#1} #2}

\newcommand{\K}{\Kcal}
\newcommand{\forcing}{\Vdash}
\newcommand{\forcings}{\Vdash^*}
\newcommand{\nforcing}{\nVdash}

\newcommand{\Modname}{\mathrm{Mod}}
\newcommand{\Mod}[1]{\Modname(#1)}

\newcommand{\sigmareg}[2]{\s^{\seqfrjArrow}_{#1}(#2)}
\newcommand{\sigmairr}[2]{\s^{\seqfrjiArrow}_{#1}(#2)}
\newcommand{\genericArrow}{\odot}

\newcommand{\Bcal}{\mathcal{B}}
\newcommand{\Ccal}{\mathcal{C}}
\newcommand{\Dcal}{\mathcal{D}}

\newcommand{\Kcal}{\mathcal{K}}
\newcommand{\Ical}{\mathcal{I}}
\newcommand{\Jcal}{\mathcal{J}}
\newcommand{\Lcal}{\mathcal{L}}

\newcommand{\Rcal}{\mathcal{R}}

\newcommand{\Tcal}{\mathcal{T}}

\renewcommand{\a}{\alpha}
\renewcommand{\b}{\beta}
\newcommand{\g}{\gamma}
\renewcommand{\d}{\delta}
\newcommand{\s}{\sigma}
\newcommand{\D}{\Delta}
\newcommand{\G}{\Gamma}

\newcommand{\EndEs}{\mbox{}~\hfill$\Diamond$}

\newcommand{\stru}[1]{\langle #1 \rangle} 

\newcommand{\wgname}{\mathrm{wg}}
\newcommand{\wg}[1]{\wgname(#1)}
\newcommand{\tp}[1]{\mathrm{tp}(#1)}
\newcommand{\tpm}[1]{\mathrm{tp}^-(#1)}
\newcommand{\card}[1]{||#1||}
\newcommand{\size}[1]{| #1|}

\newcommand{\PV}{\mbox{$\mathcal{V}$}\xspace} 
\newcommand{\Prime}{\PV^\bot\xspace}  

\newcommand{\Sf}[1]{\mathrm{Sf}(#1)}
\newcommand{\Sfm}[1]{\mathrm{Sf}^-(#1)}
\newcommand{\Sflname}{\mathsc{Sl}}
\newcommand{\Sfl}[1]{\Sflname(#1)}
\newcommand{\Sfrname}{\mathsc{Sr}}
\newcommand{\Sfr}[1]{\Sfrname(#1)}
\newcommand{\Sfgname}{\mathsc{Sx}}
\newcommand{\Sfg}[1]{\Sfgname(#1)}

\newcommand{\Gat}{\G^{\mathrm{At}}}
\newcommand{\Dat}{\D^{\mathrm{At}}}
\newcommand{\Gimp}{\G^\imp}
\newcommand{\bG}{{\overline\G}}
\newcommand{\bGat}{{\overline\G}^{\mathrm{At}}}
\newcommand{\bGimp}{{\overline\G}^\imp}
\newcommand{\That}{\Theta^{\mathrm{At}}}

\newcommand{\Thimp}{\Theta^\imp}

\newcommand{\Sigat}{\Sigma^{\mathrm{At}}}
\newcommand{\Sigimp}{\Sigma^\imp}

\newcommand{\Lambdas}{\Lambda^*}
\newcommand{\Lambdasi}{\Lambda^{*_\imp}}
\newcommand{\Fmimp}{\Lcal^\imp}

\newcommand{\Clo}[1]{\Ccal l(#1)}

\newcommand{\Restr}[2]{#1/#2}

\newcommand{\PS}[1]{\mathrm{P}(#1)}

\newcommand{\mapstoz}{\mapsto_0}
\newcommand{\mapstorz}[1]{\stackrel{#1}{\mapsto_0}}
\newcommand{\mapstos}{\mapsto_*}
\newcommand{\mapstois}{\mapsto^\mathtt{Ir}_*}

\newcommand{\frj}{\texttt{frj}\xspace}
\newcommand{\lsj}{\texttt{lsj}\xspace}
\newcommand{\bddint}{\texttt{BDDIntKt}\xspace}

\newcommand{\eval}{\rhd}
\newcommand{\neval}{\mbox{$\,\not\rhd\,$}}

\newcommand{\DBof}[1]{\mathtt{D}_{#1}}
\newcommand{\DBMof}[1]{\mathtt{D}^{\ast}_{#1}}

\newcommand{\Omat}{\Omega^{\mathrm{At}}}
\newcommand{\Omimp}{\Omega^\imp}

\newcommand{\wggbuname}{\mathrm{Wg}}
\newcommand{\wggbu}[1]{\wggbuname(#1)}

\newcommand{\Search}{\mbox{\sc BSearch}\xspace}

\newcommand{\GJ}{\mathbf{G3i}\xspace}
\newcommand{\Rn}[1]{\mathrm{Rn}(#1)}  
\newcommand{\derfrjg}[3]{\Dcal^{#1}_{#2}(#3)}
\newcommand{\derfrjr}[2]{\derfrjg{\seqfrjArrow}{#1}{#2}}
\newcommand{\derfrji}[2]{\derfrjg{\seqfrjiArrow}{#1}{#2}}
\newcommand{\height}[1]{\mathrm{h}(#1)}
\newcommand{\Lambdasa}{\Lambda^{*_\mathrm{At}}}

\newcommand{\ass}{\;\leftarrow\;}

\makeatletter  
\newcommand{\leqnomode}{\tagsleft@true}
\newcommand{\reqnomode}{\tagsleft@false}
\makeatother

\title{Duality between unprovability and provability in forward
  proof-search for Intuitionistic Propositional Logic}

\author{Camillo Fiorentini\inst{1}, Mauro Ferrari\inst{2}}

\institute{DI, Univ. degli Studi di Milano, Via Comelico, 39, 20135
  Milano, Italy \and DiSTA, Univ. degli Studi dell'Insubria, Via
  Mazzini, 5, 21100, Varese, Italy }

\begin{document}




\maketitle

\begin{abstract} 
  The inverse method is a saturation based theorem proving technique;
  it relies on a forward proof-search strategy and can be applied to
  cut-free calculi enjoying the subformula property.  Here we apply
  this method to derive the unprovability of a goal formula $G$ in
  Intuitionistic Propositional Logic. To this aim we design a forward
  calculus $\FRJof{G}$ for Intuitionistic unprovability which is prone
  to constructively ascertain the unprovability of a formula $G$ by
  providing a concise countermodel for it; in particular we prove that
  the generated countermodels have minimal height. Moreover, we
  clarify the role of the saturated database obtained as result of a
  failed proof-search in $\FRJof{G}$ by showing how to extract from
  such a database a derivation witnessing the Intuitionistic validity
  of the goal.
\end{abstract}



\section{Introduction}

The inverse method, introduced in the 1960s by
Maslov~\cite{Maslov:67}, is a saturation based theorem proving
technique closely related to (hyper)resolution~\cite{VoronkovHAR:01};
it relies on a forward proof-search strategy and can be applied to
cut-free calculi enjoying the subformula property.  Given a goal, a
set of instances of the rules of the calculus at hand is selected;
such specialized rules are repeatedly applied in the forward
direction, starting from the axioms (i.e., the rules without
premises).  Proof-search terminates if either the goal is obtained or
the database of proved facts saturates (no new fact can be added).  As
pointed out by Vladimir Lifschitz~\cite{Lifschitz89}, \emph{``the role of
the inverse method in the Soviet work on proof procedures for
predicate logic can be compared to the role of resolution method in
theorem proving projects in the West''}.  But, he regrets, 
\emph{``for a number of reasons, this work has not been duly appreciated outside a
small circle of Maslov's associates''}.  The method has been
popularized by Degtyarev and Voronkov~\cite{VoronkovHAR:01}, who
provide the general recipe to design forward calculi, with
applications to Classical Predicate Logic and some non-classical
logics.  Further extensions can be found
in~\cite{ChaudB:Tabl15,DonnellyGKMP:04,KovacsMV:13}.  A significant
investigation is presented in~\cite{ChaudP:CADE05,ChaudPPIJCAR:06},
where focused calculi and polarization of formulas are exploited to
reduce the search space in forward proof-search for Intuitionistic
Logic.  These techniques are at the heart of the design of the prover
Imogen~\cite{McLPfe:2008}.

In all the mentioned papers, the inverse method has been exploited to
prove the validity of a goal in a specific logic.  Here we follow the
dual approach, namely: we design a forward calculus to derive the
unprovability of a goal formula in Intuitionistic Propositional Logic
($\IPL$).  Our motivation is twofold. Firstly, we aim to define a
calculus which is prone to constructively ascertain the unprovability
of a formula by providing a concise countermodel for it. The second
motivation is to clarify the role of the saturated database obtained as result
of a failed proof-search.  In the case of the usual forward calculi
for Intuitionistic provability, if proof-search fails, a saturated
database is generated which \emph{``may be considered a kind of
  countermodel for the goal sequent''} \cite{McLPfe:2008}.  However,
as far as we know, no method has been proposed to effectively extract
it.  Actually, the main problem comes from the high level of
non-determinism involved in the construction of countermodels. Here,
assuming the dual approach, the saturated database generated by a
failed proof-search can be considered as \emph{``a kind of proof of the
  goal''}; we give evidence of this by showing how to extract from
such a database a derivation  witnessing the Intuitionistic validity
of the goal.

Our different perspectives requires a deep adjustment of the inverse
method itself.  Sequents $\seqs{\G}{C}$ of standard forward calculi
encode the fact that the right formula $C$ is provable from the set of
left formulas $\G$ in the understood logic.  In our approach, a
sequent $\seqfrj{\G}{C}$ signifies the unprovability of $C$ from $\G$
in $\IPL$.  From a semantic viewpoint, this means that, in some world
of a Kripke model, all the formulas in $\G$ are forced and $C$ is not
forced.  In standard forward reasoning, axioms have the form
$\seqs{p}{p}$, where $p$ is a propositional variable.  In our
approach, axioms have the form $\seqfrj{\Gat}{p}$ or
$\seqfrj{\Gat}{\bot}$, where $\Gat$ is a set of propositional
variables and $p$ is a propositional variable not belonging to $\Gat$.
Rules must preserve (top-down) unprovability.  Examples of sound rules
for unprovability are:
\[
\AXC{$\seqfrj{\G}{A}$}
\RightLabel{$R \land$}
\UIC{$\seqfrj{\G}{A\land B}$}
\DP
\qquad
\AXC{$\seqfrj{A,\G}{C}$}
\RightLabel{$L \lor$}
\UIC{$\seqfrj{A\lor B,\G}{C}$}
\DP
\]
The former rule states that if $A$ is not provable from $\G$, then
$A\land B$ is not provable from $\G$.  The latter corresponds to the
contrapositive of Inversion Principle for left $\lor$: if $C$ is not
provable from $\{A\}\cup\G$, then $C$ is not provable from $\{A\lor
B\}\cup\G$.  The tricky point is how to cope with rules having more
than one premise.  In direct forward calculi, left formulas must be
gathered.  For instance, the rule for right $\land$
\[
\AXC{$\seqs{\G_1}{A}$}
\AXC{$\seqs{\G_2}{B}$}
\RightLabel{$R \land$}
\BIC{$\seqs{\G_1,\G_2}{A\land B}$}
\DP
\]
encodes the property that if $A$ is provable from $\G_1$ and $B$ is
provable from $\G_2$, then $A\land B$ is provable from $\G_1\cup\G_2$.
Apparently, in the forward refutation calculus we should follow the
dual approach and intersect left formulas.  Thus, the rule $R\lor$
should be
\[
\AXC{$\seqfrj{\G_1}{A}$}
\AXC{$\seqfrj{\G_2}{B}$}
    \RightLabel{$R\lor$}
    \BIC{$\seqfrj{\G_1\cap\G_2}{A\lor B}$}
    \DP
\]
to be interpreted as ``if $A$ is not provable from $\G_1$ and
$B$ is not provable from $\G_2$, then
  $A\lor B$   is not provable from $\G_1\cap\G_2$''.
But the alleged rule  $R\lor$ does not preserve unprovability,
  as shown by this  trivial counterexample:
\[
\AXC{$\seqfrj{q_2,\,p,\, H}{q_1}$}
\AXC{$\seqfrj{q_1,\,p,\,H }{q_2}$}
\RightLabel{$R \lor$}
\BIC{$\seqfrj{p,H}{q_1\lor q_2}$}
\DP
\qquad H\;=\;p\imp q_1\lor q_2
\]
Here $q_1$ is not provable from $\G_1=\{q_2,p,H\}$ and $q_2$ is not
provable from $\G_2=\{q_1,p,H\}$, but the right formula of the
conclusion $q_1\lor q_2$ is provable from $\G_1\cap\G_2=\{p,H\}$.  The
drawback is that the conclusion cannot retain both $p$ and $H$ in
left. To get a sound rule, we have to select a suitable subset
$\G_0$ of $\G_1\cap\G_2$: the possible choices are $\G_0=\{p\}$ or
$\G_0=\{H\}$ or $\G_0=\emptyset$.  Thus, we need a more clever
strategy to join sequents and to treat left formulas in multi-premise
rules.  In addition to the sequents mentioned so far, we call
\emph{regular} sequents, we introduce \emph{irregular} sequents of the
kind $\seqfrji{\Sigma}{\Theta}{C}$, where the left formulas are
partitioned into two sets $\Sigma$ and $\Theta$; in forward
proof-search, formulas in the sets $\Sigma$ must be kept as much as
possible.

The definition of the rules of a forward calculus depends on the
formula to be proved (the goal formula).  The calculus we define is
parametrized by the goal formula $G$ (where the goal is to prove that
$G$ is not valid in $\IPL$).  We call the related calculus $\FRJof{G}$
(Forward Refutation calculus for $\IPL$ parametrized by $G$); formulas occurring
in the sequents of $\FRJof{G}$ are suitable subformulas of $G$.  The
rules of the calculus are shown in Fig.~\ref{fig:FRJ} and discussed in
Sect.~\ref{sec:FRJ}.  We point out that, differently from standard
sequent calculi, left-hand sides of sequents only host propositional
variables and implicative formulas $A\imp B$; moreover $\FRJof{G}$
only supplies right rules.  In Sec.~\ref{sec:proofsearch} we define a
forward proof-search procedure to build an $\FRJof{G}$-derivation of a
goal formula $G$, namely an $\FRJof{G}$-derivation of a regular
sequent of the form $\seqfrj{\G}{G}$.  This is a standard saturation
procedure where the provable sequents of $\FRJof{G}$ are collected
step-by-step in a database $\DBof{G}$.  Initially $\DBof{G}$ contains
all the axioms of $\FRJof{G}$; then a loop is entered where the rules
of the calculus are repeatedly applied (in the forward direction) to
the sequents in $\DBof{G}$.  The loop ends when either $G$ is proved
or no new sequent can be added to $\DBof{G}$; since the number of
sequents of $\FRJof{G}$ is bounded, the process eventually ends.  To
avoid redundancies and maintain $\DBof{G}$ compact, we introduce a
\emph{subsumption relation} between sequents; for instance, if at some
step $\s$ is proved and $\s$ is subsumed by a sequent already in
$\DBof{G}$, then $\s$ is discarded and not added to $\DBof{G}$
(\emph{forward subsumption}).

The rules of $\FRJof{G}$ are inspired by Kripke semantics.  In
Sec.~\ref{sec:models} we show that, from a derivation of $G$, we can
extract a countermodel for $G$, namely a Kripke model such that, at
its root, the formula $G$ is not forced, witnessing that $G$ is not
valid in $\IPL$~\cite{ChaZak:97}. Actually, there is a close
correspondence between a derivation and the related Kripke model.
Thus, our forward proof-search procedure can be understood as a
top-down method to build a countermodel for $G$, starting from the
final worlds down to the root. This original approach is dual to the
standard one, where countermodels are built bottom-up, mimicking the
backward application of rules
(see.e.g.,~\cite{AvelloneFM:15,FerFioFio:2010lpar,FerFioFio:2013,FerFioFio:2015tocl,GorPos:2010,Negri:14,PinDyc:95}).
This different viewpoint has a significant impact in the outcome.
Indeed, the countermodels generated by a backward procedure are always
trees, which might contain some redundancies.  Instead, forward
methods are prone to re-use sequents and to not replicate them; thus
the generated models do not contain duplications and are in general
very concise (see the compact models in Figs.~\ref{fig:countAST}
and~\ref{fig:N17}).  In Sec.~\ref{sec:FRJ} we show that
$\FRJof{G}$-derivations have height quadratic in $\size{G}$ (the size
of $G$).  Moreover, if $G$ is not valid, the countermodel extracted
from an $\FRJof{G}$-derivation of $G$ has height at most $\size{G}$.

The relationship between a non valid formula $G$ and the height of the
countermodel extracted from an $\FRJof{G}$-derivation of $G$ is deeply
investigated in Sec.~\ref{sec:minimality}.  Here we show that, given a
countermodel for $G$ of height $h$, we can build an
$\FRJof{G}$-derivation of $G$ having height at most $h$.  By this
fact, we conclude that, if $G$ is not valid in $\IPL$, we can build an
$\FRJof{G}$-derivation of $G$ such that the height $h$ of the
extracted countermodel is minimal (namely, there exists no
countermodel for $G$ having height less than $h$).  We can tweak the
proof-search procedure so that, if $G$ is non valid, it yields an
$\FRJof{G}$-derivation of $G$ such that the extracted countermodel has
minimal height.

If the formula $G$ is valid in $\IPL$,  proof-search for $G$ fails
(indeed, no $\FRJof{G}$-derivation of $G$ can be built) and we
eventually get a \emph{saturated database} $\DBof{G}$ for $G$.  This
means that for every sequent $\s$ provable in $\FRJof{G}$, $\DBof{G}$
contains a sequent $\s'$ which subsumes $\s$; thus $\DBof{G}$ is in
some sense representative of all the sequents provable in $\FRJof{G}$.
We can exploit $\DBof{G}$ to build a derivation of $G$ in a sequent
calculus for $\IPL$.  To this aim, in Sec.~\ref{sec:gbu} we introduce
the sequent calculus $\GBUof{G}$ (see Fig.~\ref{fig:GBU}), a variant
of the well-known sequent calculus $\GJ$~\cite{TroSch:00}.  From a
$\GBUof{G}$-derivation of $G$, we can immediately obtain a
$\GJ$-derivation of $G$.  Differently from $\GJ$, backward
proof-search in $\GBUof{G}$ always terminate; indeed, we can define a
weight function on sequents such that, after the backward application
of a rule of $\GBUof{G}$ to a sequent, the weight of the sequents
decreases.  Nonetheless, backward-proof search in $\GBUof{G}$ might
present several backtrack points, in correspondence of the
applications of rules for left implication and right disjunction.  The
crucial point is that we can remove backtracking by querying the
database $\DBof{G}$: in presence of multiple non-deterministic
choices, we exploit $\DBof{G}$ to select the right way so to
successfully continue proof-search.  Thus, we can consider $\DBof{G}$
as a \emph{proof-certificate} of the validity of $G$, in the sense
that it contain enough information to reconstruct a  derivation
of $G$ in the sequent calculus $\GJ$.  If we eliminate from $\DBof{G}$ all the redundancies (if $\s$
belongs to $\DBof{G}$, then remove from $\DBof{G}$ all the sequents
subsumed by $\s$), then we get a saturated database $\DBMof{G}$ which
is the \emph{minimum} among the saturated databases of $G$, hence we
can consider $\DBMof{G}$ as the canonical proof-certificate of the
validity of $G$.  To get the minimum saturated database, we have to
enhance the proof-search procedure by implementing \emph{backward
  subsumption}.  

To evaluate the potential of our approach we have implemented $\frj$,
a Java prototype of our proof-search procedure based on the JTabWb
framework~\cite{FerFioFio:2017}\footnote{\frj is available at
  \url{http://github.com/ferram/jtabwb_provers/}.}. \frj implements
term-indexing, forward and backward subsumption and it allows the user
to generate the rendering of proofs and of the extracted
countermodels.

\section{Preliminaries}
\label{sec:prel}

We consider the propositional language $\Lcal$ based on a denumerable
set of propositional variables $\PV$, the connectives $\land$, $\lor$,
$\imp$ (as usual, $\land$ and $\lor$ bind stronger than $\imp$) and
the logical constant $\bot$; $\neg A$ is a shorthand for $A\imp \bot$.
If $\odot$ is a logical connective, we call  $\odot$-formula  a formula
with top-level connective $\odot$.
By $\Prime$ we denote the set $\PV\cup\{\bot\}$ and by $\Fmimp$ the
set of $\imp$-formulas  of $\Lcal$.  Capital Greek
letters $\G$, $\Sigma$, \dots denote sets of formulas; we use
notations like $\Gat$ and $\Gimp$ to mean that $\Gat\subseteq\PV$ and
$\Gimp\subseteq\Fmimp$.  Given a formula $G$, $\Sf{G}$ is the set of
all subformulas of $G$ (including $G$ itself).
By $\Sfl{G}$ and $\Sfr{G}$ we denote
the subsets of \emph{left} and \emph{right}
subformulas of $G$ (a.k.a.~negative/positive subformulas of $G$~\cite{TroSch:00}).
Formally, 
$\Sfl{G}$ and $\Sfr{G}$ are 
the smallest subsets of $\Sf{G}$ such that: 

\begin{itemize}
\item  $G\in\Sfr{G}$; 

\item $A\odot B\in \Sfg{G}$ implies $\{A,B\}\subseteq \Sfg{G}$, 
where   $\odot\in\{\land,\lor\}$ and 
$\Sfgname\in\{\Sflname,\Sfrname\}$;

\item $A\imp B\in \Sfl{G}$ implies $B\in \Sfl{G}$ and $A\in\Sfr{G}$;

\item $A\imp B\in \Sfr{G}$ implies $B\in \Sfr{G}$ and $A\in\Sfl{G}$.
\end{itemize}

\noindent
By $\size{A}$ we denote the \emph{size} of $A$, namely the number of
symbols in $A$.  A {\em Kripke model} is a structure
$\K=\stru{P,\leq,\rho, V}$, where $\stru{P,\leq}$ is a finite poset
with minimum $\rho$ (the \emph{root} of $\K$) and
$V:P\to 2^{\mathcal{V}}$ is a function such that $\alpha\leq \beta$
implies $V(\alpha)\subseteq V(\beta)$.  The \emph{forcing relation}
$\forcing\subseteq P\times\Lcal$ 
is defined as follows:
\begin{itemize}[leftmargin=*]
\item $\K,\a\nforcing \bot$;
\item for every $p\in\PV$, $\K,\a\forcing p$ iff $p\in V(\alpha)$;
\item $\K,\a\forcing A\land B$ iff $\K,\a\forcing A$ and $\K,\a\forcing B$;
\item $\K,\a\forcing A\lor B$ iff $\K,\a\forcing A$ or $\K,\a\forcing B$;
\item $\K,\a\forcing A\imp B$ iff, for every $\b\in P$ such that $\a\leq\b$,
  $\K,\b\nforcing A$ or $\K,\b\forcing B$.
\end{itemize}
\emph{Monotonicity property} holds for arbitrary formulas, i.e.:
$\K,\a\forcing A$ and $\a\leq \b$ imply $\K,\b\forcing A$.  A formula
$A$ is \emph{valid} in $\K$ iff $\K,\rho\forcing A$; we say that $A$
is \emph{valid} iff $A$ is valid in all the Kripke models;
Intuitionistic Propositional Logic $\IPL$ coincides with the set of
valid formulas~\cite{ChaZak:97}.  If $\K,\rho\nforcing A$, we say that
$\K$ is a \emph{countermodel} for $A$.  A \emph{final} world $\g$ of
$\K$ is a maximal world in $\stru{P,\leq}$; for every classically
valid formula $A$, we have $\K,\g\forcing A$.  Let $\G$ be a set of
formulas, by $\K,\a\forcing \G$ we mean that $\K,\a\forcing A$ for
every $A\in\G$.  Using the above notation we avoid to mention the
model $\K$ whenever it is understood (e.g., we write $\a\forcing A$
instead of $\K,\a\forcing A$); moreover, by ``model'' we mean ``Kripke
model''.

The \emph{closure} of $\G$, denoted by $\Clo{\G}$, is the smallest set
containing the formulas $X$ defined by the following grammar:
\[
X\;::=\;C~|~X\land X~|~A\lor X~|X\lor A~|~A\imp X
\qquad\mbox{$C\in\G$, $A$ any formula}
\]
 
\noindent
The following properties of closures can be easily proved:
\begin{enumerate}[label=($\Ccal l$\arabic*),ref=($\Ccal l$\arabic*)]
\item\label{propClo:1} $\K,\a\forcing \G$ implies $\K,\a\forcing
  \Clo{\G}$.

\item\label{propClo:2} $A\in\Clo{\G}$ implies
  $A\in\Clo{\,\G\cap\Sf{A}\,}$.

\item\label{propClo:3} $\G\subseteq\Clo{\G}$ and
  $\Clo{\,\Clo{\G}\,}=\Clo{\G}$.

\item\label{propClo:4} $\G_1\subseteq\G_2$ implies
  $\Clo{\G_1}\subseteq\Clo{\G_2}$.

\item\label{propClo:5} $\Clo{\G}\cap\PV\;=\;\G\cap\PV$.

\item\label{propClo:6} $\G_1\subseteq\Clo{\G_2}$ implies
  $\Clo{\G_1}\subseteq\Clo{\G_2}$ (this follows from~\ref{propClo:3}
  and~\ref{propClo:4}).
\end{enumerate}

\section{The calculus $\FRJof{G}$}
\label{sec:FRJ}

The Forward Refutation calculus $\FRJof{G}$ is a calculus to infer the
unprovability of a formula $G$ (the \emph{goal formula}) in $\IPL$ and
it is designed to support forward proof-search.  The calculus acts on
$\FRJof{G}$-sequents that depend on the subformulas of $G$. Throughout
the paper, we use the following notation:
\[
\bGat\;=\;\Sfl{G}\cap\PV
\qquad
\bGimp=\Sfl{G}\cap \Fmimp
\qquad
\bG=\bGat\cup\bGimp
\]
There are two types of $\FRJof{G}$-sequents, we call \emph{regular}
(arrow $\seqfrjArrow$) and \emph{irregular} (arrow $\seqfrjiArrow$),
defined as follows:
\begin{itemize}
\item \emph{regular sequents} have the form $\seqfrj{\Gamma}{C}$,
  where $\Gamma\subseteq\bG$ and $C\in\Sfr{G}$;

\item \emph{irregular sequents} have the form
  $\seqfrji{\Sigma}{\Theta}{C}$, where $\Sigma\cup\Theta\subseteq\bG$
  and $C\in\Sfr{G}$.
\end{itemize}

\noindent
Given a sequent $\s$, the set $\Lhs{\s}$ of \emph{left} formulas of
$\s$ and the \emph{right} formula of $\sigma$ are defined as follows:
\[
\Lhs{\s}\;=\;
\begin{cases}
  \G &\mbox{if $\s$ is regular}  
  \\
  \Sigma\cup\Theta &\mbox{if $\s$ is irregular}  
\end{cases}
\hspace{5em}
\Rhs{\s}\;=\;C
\]
We remark that the left formulas of $\s$ are left subformulas of $G$
and the right formula of $\s$ is a right subformula of $G$.
Accordingly, the number of $\FRJof{G}$-sequents is finite, hence
$\FRJof{G}$ satisfies the Finite Rule Property~\cite{VoronkovHAR:01}.
Left formulas of irregular sequents $\s$ are partitioned into the sets
$\Sigma$, the \emph{stable} set of $\s$, and $\Theta$.  In forward
proof-search, formulas in $\Sigma$ are preserved as much as possible,
while some of the formulas in $\Theta$ can be lost.  Provability in
$\FRJof{G}$ is defined as follows:
\begin{itemize}
\item $\Dcal$ is an \emph{$\FRJof{G}$-derivation of $\s$} iff $\Dcal$
  is a derivation in the calculus $\FRJof{G}$ having $\s$ as root
  sequent;

\item $\Dcal$ is an \emph{$\FRJof{G}$-derivation of $G$} iff there
  exists a (possibly empty) set of formulas $\G$ such that $\Dcal$ is
  an $\FRJof{G}$-derivation of $\seqfrj{\Gamma}{G}$;

\item $\s$  is \emph{provable} in $\FRJof{G}$,
denoted by $\proves{\FRJof{G}} \s$,
 iff there
  exists an $\FRJof{G}$-derivation of $\s$;

\item $G$  is \emph{provable} in $\FRJof{G}$,
denoted by $\proves{\FRJof{G}}G$,
 iff there
  exists an $\FRJof{G}$-derivation of $G$.

\end{itemize}

\noindent
Soundness of $\FRJof{G}$ is stated as follows:

\begin{theorem}[Soundness of $\FRJof{G}$]\label{theo:FRJsound}
  $\proves{\FRJof{G}}G$ implies $G\not\in\IPL$.
\qed
\end{theorem}

\noindent
Soundness of  $\FRJof{G}$ relies on the fact that rules of 
$\FRJof{G}$ satisfy the following soundness properties:

\begin{enumerate}[label=(S\arabic*),ref=(S\arabic*)]
\item\label{prop:soundr} if $\seqfrj{\G}{C}$ is provable in
  $\FRJof{G}$, then there exists a world $\a$ of a model $\K$ such
  that $\a\forcing\G$ and $\a\nforcing C$;

\item\label{prop:soundi} if $\s=\seqfrji{\Sigma}{\Theta}{C}$ is
  provable in $\FRJof{G}$ and $\s$ can be used to (directly or
  indirectly) prove a regular sequent in $\FRJof{G}$, then there exist
  a world $\a$ of a model $\K$ and a set $\G$ such that
  $\Sigma\subseteq \Gamma\subseteq \Sigma\cup\Theta$ and
  $\a\forcing\G$ and $\a\nforcing C$.
\end{enumerate}

\noindent
Properties~\ref{prop:soundr} and~\ref{prop:soundi}  follow by
Lemma~\ref{lemma:soundFRJ} of Sec.~\ref{sec:wg}. 
  Note that soundness  hides subsumption, which is typical of forward reasoning.  Indeed,
 let $G$ be provable in $\FRJof{G}$.  Then, there exists an
 $\FRJof{G}$-derivation of a sequent $\seqfrj{\G}{G}$.
 Property~\ref{prop:soundr} implies that the formula $(\land \G) \imp
 G$ is not valid (equivalently, $G$ is not provable from assumptions  $\G$).

The calculus $\FRJof{G}$ consists of two axiom rules, some right
introduction rules for the connectives $\land$, $\lor$, $\imp$ and the
rules $\ruleJOINA$ and $\ruleJOINO$ to join sequents; there are no
left rules.  Rules of $\FRJof{G}$ are collected in Fig.~\ref{fig:FRJ},
below we provide an in-depth presentation.

\paragraph{Axiom rules}
We have two axiom rules:
\[
\AXC{}
\RightLabel{$\ruleAXR$}
\UIC{$\seqfrj{\bGat\setminus\{F\}}{F}$}
\DP
\hspace{4em}
\AXC{}
\RightLabel{$\ruleAXI$}
\UIC{$\seqfrji{}{\bGat\setminus\{F\},\bGimp}{F}$}
\DP
\qquad F\in\Prime
\]
Rule $\ruleAXR$ introduces a regular axiom, rule $\ruleAXI$ an
irregular one.  Both axiom sequents have in the right a formula
$F\in\Prime$; in irregular axioms the set $\Sigma$ is empty (denoted
by~$\cdot$).  
We recall that forward
proof-search starts from axiom sequents.

\paragraph{Rules for $\land$}
We have two rules to introduce $\land$ in the right, one concerning
regular sequents and one the irregular ones.  In both rules the type
of the sequents (regular or irregular) does not change:
\[
\AXC{$\seqfrj{\G}{A_k}{}$}
\RightLabel{$\land$}
\UIC{$\seqfrj{\G}{A_1\land A_2}$}
\DP
\hspace{4em}
\AXC{$\seqfrji{\Sigma}{\Theta}{A_k}{}$}
\RightLabel{$\land$}
\UIC{$\seqfrji{\Sigma}{\Theta}{A_1\land A_2}$}
\DP
\quad k\in\{1,2\}
\]

\paragraph{Rules for $\imp$}
In standard refutation calculi, the rule for right implication is
\[
\AXC{$\seqfrj{\G}{B}$}
\RightLabel{$R \imp$}
\UIC{$\seqfrj{\G}{ A\imp B}$}
\DP
\qquad A\in \G
\]
The side condition $A\in \G$ is needed to guarantee that, assuming
that $B$ is not provable from $\G$, then $A\imp B$ is not provable
from $\G$ as well.  With such a rule alone, the calculus $\FRJof{G}$
would be incomplete.  For instance, using the rules presented so far,
we are not able to derive the goal $G=p_1\land p_2 \imp q$.  Indeed,
let us start from the axiom $\seqfrj{ p_1,p_2}{q}$.  We can build the
derivations
\[
\begin{array}{lll}
  \AXC{}
  \RightLabel{$\ruleAXR$}
  \UIC{$\seqfrj{{ p_1},p_2}{q}$}
  \RightLabel{$R \imp$}
  \UIC{$\seqfrj{{ p_1},{ p_2}}{{ p_1}\imp q}$}
  \RightLabel{$R \imp $}
  \UIC{$\seqfrj{p_1,{ p_2}}{{ p_2}\imp (p_1\imp q})$}
  \DP
  &\qquad&
  \AXC{}
  \RightLabel{$\ruleAXR$}
  \UIC{$\seqfrj{p_1,{ p_2}}{q}$}
  \RightLabel{$R \imp$}
  \UIC{$\seqfrj{p_1,{ p_2}}{{ p_2}\imp q}$}
  \RightLabel{$R \imp$}
  \UIC{$\seqfrj{{ p_1},p_2}{{ p_1}\imp (p_2\imp q})$}
  \DP
\end{array}
\]
but there is no means to obtain $G$ in the right, since the formula
$p_1\land p_2$ is not allowed in the left of sequents.  To compensate
for this, we have to relax the side condition.  This leads to the rule
$\ruleIMPi$ of $\FRJof{G}$
\[
\AXC{$\seqfrj{\G}{B}$}
    \RightLabel{$\ruleIMPi$}
    \UIC{$\seqfrj{\G}{A\imp B}$}
\DP
\quad A\in\Clo{\G} 
\]
where we introduce the more general side condition $A\in\Clo{\G}$
(whence the $\in$ in the rule name).  Using this rule, we can prove
the  goal  $G=p_1\land p_2 \imp q$ as follows:
\[
\AXC{}
\RightLabel{$\ruleAXR$}
\UIC{$\seqfrj{p_1,p_2}{q}$}
\RightLabel{$\ruleIMPi$}
\UIC{$\seqfrj{p_1,p_2}{{p_1\land p_2}\imp q}$}
\DP
\qquad { p_1\land p_2}\,\in\,\Clo{\,\{p_1,p_2\}\,}
\]
We can also apply rule $\ruleIMPi$ to an irregular sequent
$\s_1=\seqfrji{\Sigma}{\Theta'}{B}$ to get an irregular sequent
$\s=\seqfrji{\Sigma'}{\Theta}{A\imp B}$.  In this case the antecedent
$A$ must belong to $\Clo{\Sigma'}$ and we are allowed to transfer
formulas from $\Theta'$ to $\Sigma'$ so to satisfy such a condition.
We can formalize rule $\ruleIMPi$ for irregular sequents as follows:
\[
\AXC{$\s_1\;=\;\seqfrji{\Sigma}{\overbrace{\Theta,\Lambda}^{\Theta'}}{B}$}
\RightLabel{$\ruleIMPi$}
\UIC{$\s\;=\;\seqfrji{\underbrace{\Sigma,\Lambda}_{\Sigma'}}{\Theta}{A\imp B}$}
\DP
\quad\begin{array}{l}
       \Theta\cap\Lambda=\emptyset\\[.5ex]
       A\in\Clo{\Sigma\cup\Lambda}
  \end{array}
\]  
The set $\Theta'$ has been partitioned as $\Theta\cup\Lambda$, where
the (possibly empty) set $\Lambda$ is shifted to the left of
semicolon; note that $\Lhs{\s}=\Lhs{\s_1}$.  We remark that, 
to get the formula $A\imp B$ in the conclusion, 
in general many choices of $\Lambda$ are possible, as illustrated in the
next example.

\begin{example}\label{ex:impPS1}
  Let $\s_1=\seqfrji{}{p,q}{B}$ be an $\FRJof{G}$-sequent such that the
  formula $p\lor q\imp B$ belongs to $\Sfr{G}$.  To apply rule
  $\ruleIMPi$ to $\s_1$ so to get a sequent with $p\lor q\imp B$ in the right, we
  have to select a subset $\Lambda$ of $\{p,q\}$ satisfying the side
  condition $p\lor q\in\Clo{\Lambda}$. The following three choices are
  possible:
  \[
  \begin{array}{lll}
    \Lambda_1\;=\;\{p\}
    &
    \Lambda_2\;=\;\{q\}
    &
    \Lambda_3\;=\;\{p,q\}
    \\[1.5ex]
    \AXC{$\seqfrji{}{p,q}{B}$}
    \RightLabel{$\ruleIMPi$}
    \UIC{$\seqfrji{p}{q}{p\lor q\imp B}$}
    \DP
    \hspace{2em}
    &
    \AXC{$\seqfrji{}{p,q}{B}$}
    \RightLabel{$\ruleIMPi$}
    \UIC{$\seqfrji{q}{p}{p\lor q\imp B}$}
    \DP
    \hspace{2em}
    &
    \AXC{$\seqfrji{}{p,q}{B}$}
    \RightLabel{$\ruleIMPi$}
    \UIC{$\seqfrji{p,q}{}{p\lor q\imp B}$}
    \DP
  \end{array}
  \]
  We point out that $\Lambda_1$ and $\Lambda_2$ are minimal sets
  satisfying the side condition, while $\Lambda_3$ is not.  Indeed,
  the empty set is the only proper subset of $\Lambda_1$ and
  $\Lambda_2$ and $p\lor q\not\in\Clo{\emptyset}$.  On the other hand,
  $\Lambda_3$ is not minimal, since both $\Lambda_1$ and $\Lambda_2$
  are proper subsets of $\Lambda_3$.
  \EndEs
\end{example}

We need a further rule to introduce $A\imp B$ in the right, having
premise $\s_1=\seqfrj{\G}{B}$ and conclusion $\s=\seqfrji{}{\Theta}{A
  \imp B}$; this is the only rule of $\FRJof{G}$ turning a regular
sequent into an irregular one.  We require that $A\in\Clo{\G}$ and
$\Theta$ is any subset of $\Clo{\G}\cap\bG$ such that
$A\not\in\Clo{\Theta}$.  By the latter condition, we call the rule
$\ruleIMPni$:
\[
\AXC{$\s_1\;=\;\seqfrj{\G}{B}$}
    \RightLabel{$\ruleIMPni$}
    \UIC{$\s\;=\;\seqfrji{}{\Theta}{A\imp B}$}
\DP 
\qquad 
    \begin{minipage}{15em}
$A\in \Clo{\G}$
\\[.5ex]
$\Theta\;\subseteq\; \Clo{\Gamma} \,\cap\,\bG$
\\[.5ex]
$A\not\in\Clo{\Theta}$      
\end{minipage}
\]
The condition $A\not\in\Clo{\Theta}$ is essentially needed to
introduce a decreasing weight function on sequents (see
Sec.~\ref{sec:wg}).  Also in this case, many choices of $\Theta$ are
possible in general, as shown in the next example.

\begin{example}\label{ex:impPS2}
  Let $\s_1=\seqfrj{p,q}{B}$ be an $\FRJof{G}$-sequent and let assume that
  \[
  \bG\;=\;\Sfl{G}\cap(\PV\cup\Fmimp)\;=\;\{p,\,q,\,r,\,r\imp p,\,p\imp r\}
  \hspace{4em}
  C\,=\,p\land q\imp B\,\in\,\Sfr{G}
  \]
  We show all the possible applications of $\ruleIMPni$ to $\s_1$
  yielding an irregular sequent of the form $\seqfrji{}{\Theta}{C}$.
  By the side conditions, $\Theta$ must be a subset of
  $\Clo{\{p,q\}}\cap\bG=\{p,q,r\imp p\}$, which gives rise
  to eight possible choices:
  \[
  \begin{array}{llll}
    \Theta_1\,=\, \emptyset
    &\Theta_2\,=\,\{ p\}
    & \Theta_3\,=\,\{q\}
    &\Theta_4\,=\,\{r\imp p\}
    \\[.5ex]
    \Theta_5\,=\,\{p,r\imp p\}
    \qquad
    &
    \Theta_6\,=\,\{q,r\imp p\}
    \qquad
    &
    \Theta_7\,=\,\{p,q\}
    \qquad
    &
    \Theta_8\,=\,\{p,q,r\imp p\}
  \end{array}
  \]
  By the side conditions, we need  $p\land q\not\in\Clo{\Theta}$, hence
  $\Theta_7$ and $\Theta_8$ must be ruled out. This leads to six
  possible applications of rule $\ruleIMPni$:
  \[
  \begin{array}{ll}
    \Theta_1\,=\,\emptyset\quad
    \AXC{$\seqfrj{p,q}{B}$}
    \RightLabel{$\ruleIMPni$}
    \UIC{$\seqfrji{}{\Theta_1}{C}$}
    \DP
    &
    \Theta_2\,=\,\{ p\}\quad
    \AXC{$\seqfrj{p,q}{B}$}
    \RightLabel{$\ruleIMPni$}
    \UIC{$\seqfrji{}{\Theta_2}{C}$}
    \DP
    \\[5ex]
    \Theta_3\,=\,\{q\}\quad
    \AXC{$\seqfrj{p,q}{B}$}
    \RightLabel{$\ruleIMPni$}
    \UIC{$\seqfrji{}{\Theta_3}{C}$}
    \DP
    &
    \Theta_4\,=\,\{r\imp p\}\quad
    \AXC{$\seqfrj{p,q}{B}$}
    \RightLabel{$\ruleIMPni$}
    \UIC{$\seqfrji{}{\Theta_4}{C}$}
    \DP
    \\[5ex]
    \Theta_5\,=\,\{p,r\imp p\}\quad
    \AXC{$\seqfrj{p,q}{B}$}
    \RightLabel{$\ruleIMPni$}
    \UIC{$\seqfrji{}{\Theta_5}{C}$}
    \DP
    \quad
    &
    \Theta_6\,=\,\{q,r\imp p\}\quad
    \AXC{$\seqfrj{p,q}{B}$}
    \RightLabel{$\ruleIMPni$}
    \UIC{$\seqfrji{}{\Theta_6}{C}$}
    \DP
  \end{array}
  \]
  We point out that $\Theta_5$ is a maximal set satisfying the side
  condition.  Indeed, the only proper superset $\Theta'$ of $\Theta_5$
  such that $\Theta'\,\subseteq\, \Clo{\{p,q\}}\cap\bG$ is $\Theta_8$
  and $p\land q\in\Clo{\Theta_8}$.  Similarly, $\Theta_6$ is maximal as well.  On
  the other hand, sets $\Theta_1,\dots,\Theta_4$ are not maximal,
  since each of them is a proper subset of $\Theta_5$ or $\Theta_6$.
 \EndEs
\end{example}

\paragraph{Rule $\lor$}
The rule $\lor$ has premises $\s_1= \seqfrji{\Sigma_1}{\Theta_1}{C_1}$
and $\s_2= \seqfrji{\Sigma_2}{\Theta_2}{C_2}$ and conclusion
$\s=\seqfrji{\Sigma}{\Theta}{C_1\lor C_2}$ introducing a disjunction
in the right.  As discussed in the Introduction, some care is required
in defining the left formulas of $\s$.  The leading idea is that the
$\Sigma$-sets of premises must be preserved in the conclusion, while
the $\Theta$-sets are intersected.  We have to guarantee that
$\Lhs{\s}\subseteq\Lhs{\s_1}\cap\Lhs{\s_2}$, hence some side
conditions on the $\Sigma$-sets are needed.  We define:
\[
\AXC{$\s_1\;=\;\seqfrji{\Sigma_1}{\Theta_1}{C_1}$}
\AXC{$\s_2\;=\;\seqfrji{\Sigma_2}{\Theta_2}{C_2}$}
\RightLabel{$\lor$}
\BIC{$\s\;=\;
  \seqfrji{\underbrace{\Sigma_1,\Sigma_2}_{\Sigma}\;}
          {\;\underbrace{\Theta_1\cap\Theta_2}_{\Theta}}
          {C_1\lor C_2}$}
\DP
\qquad 
\begin{minipage}{10em}
  $\Sigma_1\;\subseteq\;\Sigma_2\cup\Theta_2$\\[.5ex]
  $\Sigma_2\;\subseteq\;\Sigma_1\cup\Theta_1$     
\end{minipage}
\]

\paragraph{Join rules}
The \emph{join} rules $\ruleJOINA$ and $\ruleJOINO$ apply to $n\geq 1$
irregular sequents
$\s_1=\seqfrji{\Sigma_1}{\Theta_1}{A_1},\dots,\s_n=\seqfrji{\Sigma_n}{\Theta_n}{A_n}$
and yield a regular sequent $\s=\seqfrj{\G}{C}$; this is the only way
to obtain a regular sequent from irregular ones.  These two rules have
a similar structure and only differ in the form of $C$: in rule
$\ruleJOINA$, $C\in\Prime$ while in $\ruleJOINO$, $C$ is an
$\lor$-formula.  For every $1 \leq j\leq n$, we write the premise
$\s_j$ as follows:
\[
\s_j\;=\;
\seqfrji{\underbrace{\Sigat_j,\Sigimp_j}_{\Sigma_j}\;}
{\;\underbrace{\That_j,\Thimp_j}_{\Theta_j}}{A_j}
\qquad
\begin{minipage}{16em}
  $\Sigat_j\cup\That_j\,\subseteq \,\PV$
  and
  $\Sigimp_j\cup\Thimp_j\,\subseteq\,\Fmimp$
\end{minipage}
\]
Similarly to rule $\lor$, formulas in $\Sigma_j$ must be preserved in
the conclusion and we need $\Lhs{\s}\subseteq
\Lhs{\s_1}\cap\cdots\cap\Lhs{\s_n}$.  Thus,
we require the side condition 
of rule $\lor$ for every pair of distinct premises, namely:

\begin{enumerate}[label=(J\arabic*),ref=(J\arabic*)]
\item\label{sc:J1}
  $\Sigma_i\subseteq \Sigma_j\cup\Theta_j$ for every $i\neq j$.
\end{enumerate}

\noindent
From a semantic perspective, the role of join rules is to downward
expand a countermodel under construction; the conclusion
$\s=\seqfrj{\G}{C}$ directly corresponds to a new world $\a$ of the
countermodel such that all the formulas in $\G$ are forced in $\a$ and
$C$ is not forced in $\a$ (see the soundness property~\ref{prop:soundr}).  
To perform this, we need a further side condition~\ref{sc:J2} on the sets $\Sigma_j$.
Let
\[
\Sigat\;=\;\bigcup_{1\leq j\leq n} \Sigat_j
\hspace{5em}
\Sigimp\;=\;\bigcup_{1\leq j\leq n} \Sigimp_j
\]  
Since both $\Sigat$ and $\Sigimp$ must be kept in the conclusion
$\s=\seqfrj{\G}{C}$ (namely, $\Sigat\cup\Sigimp\subseteq\Gamma$), we
need that all the formulas in $\Sigat\cup\Sigimp$ are forced in the
new world $\a$.  A formula $Y\imp Z$ is \emph{supported} if there
exists a premise $\s_k$ ($1\leq k\leq n$) such that $Y=A_k$ ($A_k$ is
the right formula of $\s_k$).  Intuitively, if a formula
$Y\imp Z\in\Sigimp$ is supported, then $Y$ is not forced in $\a$, and
this allows us to conclude that $Y\imp Z$ is forced in $\a$.  We
require that all the $\imp$-formulas in the sets $\Sigma_j$ are
supported, and this is formalized by the following side condition:

\begin{enumerate}[label=(J\arabic*),ref=(J\arabic*),start=2]
\item\label{sc:J2}
  $Y\imp Z\in\Sigimp$ implies $Y\in\{A_1,\dots,A_n\}$.
\end{enumerate}

\begin{example}
  Let $\s_1=\seqfrji{\Sigma_1 }{\Theta_1}{B}$, where
  $\Sigma_1=\{B\imp X_1, B\imp X_2, C\imp X_3 \}$ and $C\neq B$.  We
  cannot apply a join rule having $\s_1$ as only premise since the
  formula $C\imp X_3$ would not be supported (while both $B\imp X_1$
  and $B\imp X_2$ are, being $B$ the right formula of $\s_1$).  Thus,
  to apply a join rule, we need a further premise
  $\s_2=\seqfrji{\Sigma_2}{\Theta_2}{C}$ such that
  $\Sigma_1\subseteq\Sigma_2\cup\Theta_2$ and
  $\Sigma_2\subseteq\Sigma_1\cup\Theta_1$ (see~\ref{sc:J1}).  In turn,
  $\Sigma_2$ might contain some non-supported formulas; e.g., a
  formula $D\imp X_4$ such that $D\neq B$ and $D\neq C$; in this case,
  we need a further premise $\s_3$ s.t. $\Rhs{\s_3}=D$ to support it. \EndEs
\end{example}

Side conditions do not concern the sets $\Theta_j$ and in the
conclusion we keep some of the formulas in the intersection of all the
$\Theta_j$'s.  More precisely, given a set of $\imp$-formulas $\Gimp$
and a set of formulas $\Upsilon$, let
\[
\Restr{\Gimp}{\Upsilon}\;=\;
\{~Y\imp Z\in \Gimp~|~Y\in\Upsilon~\}
\]  
We call $\Restr{\Gimp}{\Upsilon}$ the \emph{restriction} of $\Gimp$ to
$\Upsilon$.  Then, the formulas $\That$ and $\Thimp$ to be kept in the
conclusion are defined as follows:
\[
 \That\;=\;\bigcap_{1\leq j\leq n} \That_j
 \hspace{5em}
\Thimp\;=\;\Restr{  \left(\,\bigcap_{1\leq j\leq n} \Thimp_j\, \right) }{\;\{A_1,\dots,A_n\} }
\]
We can now define the conclusion $\s$ of a join rule having premises
$\s_1,\dots,\s_n$ matching the side conditions~\ref{sc:J1}
and~\ref{sc:J2}.  In rule $\ruleJOINA$, we can choose as right formula
of $\s$ any formula $F\in\Prime\cap\Rhs{G}$ such that
$F\not\in\Sigat$.  We set:
\[
\s\;=\; \seqfrj{\Sigat,\,\That\setminus\{F\},\,\Sigimp,\,\Thimp}{F}
\]
In rule $\ruleJOINO$, we can choose as right formula of
$\s$ any $\lor$-formula $C_1\lor C_2\in\Rhs{G}$ such that both $C_1$ and $C_2$ are
among $A_1,\dots , A_n$.  We set:
\[
\s\;=\; \seqfrj{\Sigat,\,\That,\,\Sigimp,\,\Thimp}{C_1\lor C_2}
\qquad\{C_1,C_2\}\;\subseteq\;\{A_1,\dots,A_n\}
\]
We note that rule $\ruleJOINO$ is similar to the rule $r_n$ presented
in~\cite{Skura:89}.

We point out the following special case of application of $\ruleJOINA$
having only one premise ($\Sigimp$ is possibly empty):

\[
\begin{array}{c}
  \AXC{$\seqfrji{\Sigat,\,\Sigimp\,}{\,\That,\,\Thimp}{A}$}
  \RightLabel{$\ruleJOINA$}
  \UIC{$\seqfrj{\Sigat,\,\Sigimp,\,\That\setminus\{F\},\,\Restr{\Thimp}{\{A\}}}{F}$}
  \DP
  \qquad 
  \begin{minipage}{20em}
    $\Sigimp\;=\;\{A\imp B_1,\dots, A\imp B_m\}$
    \\[.5ex]
    $F\in\Prime\setminus\Sigat$
  \end{minipage}
\end{array}
\]

\medskip

Now, we begin to study the properties of $\FRJof{G}$-derivations. Let
us introduce the following relations between sequents of $\FRJof{G}$:
\begin{itemize}
\item $\s_1\mapstorz{\Rcal} \s_2$ iff $\Rcal$ is a rule of $\FRJof{G}$
  such that $\s_2$ is the conclusion of $\Rcal$ and $\s_1$ is one of
  its premises;
  
\item $\s_1\mapstoz \s_2$ iff there exists a rule $\Rcal$ of
  $\FRJof{G}$ such that $\s_1\mapstorz{\Rcal} \s_2$;

\item $\mapsto$ is the transitive closure of $\mapstoz$;

\item $\mapstos$ is the reflexive closure of $\mapsto$.
\end{itemize}

\noindent
By inspecting the rules of $\FRJof{G}$ and exploiting
properties~\ref{propClo:3} and~\ref{propClo:6} of closures, the
following facts can be easily proved:

\begin{lemma}\label{lemma:lhs}
\mbox{}
\begin{enumerate}[label=(\roman*),ref=(\roman*)]
\item\label{lemma:lhs:1} $\s_1\mapstorz{\Rcal} \s_2$ and
  $\Rcal\neq\,\ruleIMPni$ imply $\Lhs{\s_2}\,\subseteq\,\Lhs{\s_1}$.

\item\label{lemma:lhs:2} $\s_1\mapstoz \s_2$ implies
  $\Lhs{\s_2}\,\subseteq\,\Clo{\Lhs{\s_1}}$.

\item\label{lemma:lhs:3} $\s_1\mapstos \s_2$ implies
  $\Lhs{\s_2}\,\subseteq\,\Clo{\Lhs{\s_1}}$.  \qed
\end{enumerate}
\end{lemma}

%
\begin{figure}[t]
  \[
  \begin{array}{c}
    \begin{minipage}{1.0\linewidth}\small
      $\bGat\,=\,\Sfl{G}\cap\PV$,\, $\bGimp\,=\,\Sfl{G}\cap \Fmimp$, 
      $\bG\,=\,\bGat\cup\bGimp$. \\[.5ex]
      In the conclusion $\s$ of each rule,
      $\Rhs{\s}\in\Sfr{G}$
    \end{minipage}
    \\[4ex]
    \AXC{}
    \RightLabel{$\ruleAXR$}
    \UIC{$\seqfrj{\bGat\setminus\{F\}}{F}$}
    \DP
    \hspace{6em}
    \AXC{}
    \RightLabel{$\ruleAXI$}
    \UIC{$\seqfrji{}{\bGat\setminus\{F\},\bGimp}{F}$}
    \DP
    \qquad F\in\Prime
    \\[4ex]
    \AXC{$\seqfrj{\G}{A_k}{}$}
    \RightLabel{$\land$}
    \UIC{$\seqfrj{\G}{A_1\land A_2}$}
    \DP
    \hspace{8em}
    \AXC{$\seqfrji{\Sigma}{\Theta}{A_k}{}$}
    \RightLabel{$\land$}
    \UIC{$\seqfrji{\Sigma}{\Theta}{A_1\land A_2}$}
    \DP
    \quad k\in\{1,2\}
    \\[4ex]
    \AXC{$\seqfrji{\Sigma_1}{\Theta_1}{C_1}$}
    \AXC{$\seqfrji{\Sigma_2}{\Theta_2}{C_2}$}
    \RightLabel{$\lor$}
    \BIC{$\seqfrji{\Sigma_1,\Sigma_2}{\Theta_1\cap\Theta_2}{C_1\lor C_2}$}
    \DP
    \qquad 
    \begin{minipage}{6em}
      $\Sigma_1\subseteq\Sigma_2\cup\Theta_2$\\
      $\Sigma_2\subseteq\Sigma_1\cup\Theta_1$     
    \end{minipage}
    \\[4ex]
    \AXC{$\seqfrj{\G}{B}$}
    \RightLabel{$\ruleIMPi$}
    \UIC{$\seqfrj{\G}{A\imp B}$}
    \DP
    \quad A\in\Clo{\G} 
    \hspace{6em}
    \AXC{$\seqfrji{\Sigma}{\Theta,\Lambda}{B}$}
    \RightLabel{$\ruleIMPi$}
    \UIC{$\seqfrji{\Sigma,\Lambda}{\Theta}{A\imp B}$}
    \DP
    \quad\begin{array}{l}
           \Theta\cap\Lambda=\emptyset\\[.5ex]
           A\in\Clo{\Sigma\cup\Lambda}
         \end{array}     
    \\[4ex]
    \AXC{$\seqfrj{\G}{B}$}
    \RightLabel{$\ruleIMPni$}
    \UIC{$\seqfrji{}{\Theta}{A\imp B}$}
    \DP 
    \qquad 
    \begin{minipage}{15em}
      $\Theta\;\subseteq\; \Clo{\Gamma} \,\cap\,\bG$
      \\[1ex]
      $A\;\in\; \Clo{\G}\,\setminus\,\Clo{\Theta}$      
    \end{minipage}
    \\[6ex]
    \begin{minipage}{1.0\linewidth}
      Let, for $1\leq j\leq n$, $\s_j\;=\;
      \seqfrji{\underbrace{\Sigat_j,\Sigimp_j}_{\Sigma_j}}
      {\underbrace{\That_j,\Thimp_j}_{\Theta_j}}{A_j}$ and
      $\Upsilon=\{A_1,\dots,A_n\}$
    \end{minipage}
    \\[4ex]
    \Sigat=\bigcup_{1\leq j\leq n} \Sigat_j
    \qquad
    \That=\bigcap_{1\leq j\leq n} \That_j
    \qquad
    \Sigimp=\bigcup_{1\leq j\leq n} \Sigimp_j
    \qquad
    \Thimp=\Restr{(\bigcap_{1\leq j\leq n} \Thimp_j)}{ \Upsilon }
    \\[4ex]
    \AXC{$\s_1\qquad\cdots\qquad \s_n$}
    \RightLabel{$\ruleJOINA$}
    \UIC{$\seqfrj{\Sigat,\,\That\setminus\{F\},\,\Sigimp,\,\Thimp}{F}$}
    \DP 
    \qquad
    \begin{minipage}{20em}
      $\Sigma_i\subseteq \Sigma_j\cup\Theta_j$,
      for every $i\neq j$
      \\[.5ex]
      $Y\imp Z\in\Sigimp$ implies $Y\in\Upsilon$
      \\[.5ex]
      $F\in\Prime\setminus\Sigat$
    \end{minipage}
    \\[6ex]
    \AXC{$\s_1\qquad\cdots\qquad \s_n$}
    \RightLabel{$\ruleJOINO$}
    \UIC{$\seqfrj{\Sigat,\,\That,\,\Sigimp,\,\Thimp}{C_1\lor C_2}$}
    \DP 
    \qquad
    \begin{minipage}{20em}
      $\Sigma_i\subseteq \Sigma_j\cup\Theta_j$,
      for every $i\neq j$
      \\[.5ex]
      $Y\imp Z\in\Sigimp$ implies $Y\in\Upsilon$
      \\[.5ex]
      $\{C_1,C_2\}\,\subseteq\,\Upsilon$
    \end{minipage}
  \end{array}
  \]
  \caption{The calculus $\FRJof{G}$.}
  \label{fig:FRJ}
\end{figure}

In the next example we show that the unsound derivation discussed in
the Introduction cannot be simulated in $\FRJof{G}$.

\begin{example}\label{ex:Ror}
  In the Introduction, by applying the unsound rule $R\lor$, we got a
  wrong derivation of the sequent $\s=\seqfrj{p,H}{q_1\lor q_2}$,
  where $H= p\imp q_1\lor q_2$.  By the soundness
  property~\ref{prop:soundr}, if $\s$ could be proved in $\FRJof{G}$,
  there should be a world $\a$ of a model such that both $p$ and $H$
  are forced in $\a$ and $q_1\lor q_2$ is not forced in $\a$, a
  contradiction.  We show that it is not possible to build an
  $\FRJof{G}$-derivation $\Dcal$ of $\s$.  Indeed, the root rule of
  $\Dcal$ should be $\ruleJOINO$, the only rule of $\FRJof{G}$ having
  as conclusion a regular sequent with an $\lor$-formula in the right.
  Since the formula $H$ must be supported (see~\ref{sc:J2}), rule
  $\ruleJOINO$ should have a premise
  $\s'=\seqfrji{\Sigma'}{\Theta'}{p}$.  Since the only irregular
  sequents having a formula in $\Prime$ in the right are the irregular
  axioms, $\Dcal$ should have the following shape:
  \[
  \AXC{$\cdots$}
  \AXC{}
  \RightLabel{$\ruleAXI$}
  \UIC{$\s'\;=\;\seqfrji{}{\Theta'}{p}$}
  \AXC{$\cdots$}
  \LeftLabel{$\Dcal\;=\quad$}
  \RightLabel{$\ruleJOINO$}
  \TIC{$\s\;=\;\seqfrj{p,H}{q_1\lor q_2}$}
  \DP
  \qquad p\not\in\Theta'
  \]  
  Hence, $\s'\mapstorz{\ruleJOINO} \s$ which, by
  Lemma~\ref{lemma:lhs}\ref{lemma:lhs:1}, implies $\{p,H\}\subseteq
  \Theta'$.  Thus, both $p\in\Theta'$ and $p\not\in\Theta'$, a
  contradiction.  We conclude that $\s$ is not provable.
  
  However, it is possible to have applications of rule $\ruleJOINO$
  having conclusion $\seqfrj{p}{q_1\lor q_2}$ or $\seqfrj{H}{q_1\lor
    q_2}$.  For instance, let
  \[
  G\;=\; (p\land H)\,\imp \, (q_1\lor q_2)
  \qquad  H\;= \;p\imp q_1\lor q_2
  \]
  We have:
  \[
  \Sfl{G}\:=\:\{\, p\land H,\, H,\,q_1\lor q_2,\,p,\,\,q_1,\,q_2\, \}
  \qquad
  \Sfr{G}\:=\:\{\, G,\, \,q_1\lor q_2,\,p,\,\,q_1,\,q_2\, \}
  \]  
  We can build the following $\FRJof{G}$-derivation:
  \[
  \begin{array}{c}
    \AXC{}
    \RightLabel{$\ruleAXI$}
    \UIC{$\seqfrji{}{p,q_2,H}{q_1}$}
    \AXC{}
    \RightLabel{$\ruleAXI$}
    \UIC{$\seqfrji{}{p,q_1,H}{q_2}$}
    \RightLabel{$\ruleJOINO$}
    \BIC{$\seqfrj{p}{q_1\lor q_2}$}
    \DP
  \end{array}
  \]  
  We point out that the side conditions~\ref{sc:J1} and~\ref{sc:J2}
  trivially hold, since the $\Sigma$-sets of the premises are empty.
  In the conclusion, $H$ is left out since it is not supported (no
  premise has $p$ in the right).  Similarly, we can build the
  following $\FRJof{G}$-derivation:
  \[
  \AXC{}
  \RightLabel{$\ruleAXI$}
  \UIC{$\seqfrji{}{p,q_2,H}{q_1}$}
  \AXC{}
  \RightLabel{$\ruleAXI$}
  \UIC{$\seqfrji{}{p,q_1,H}{q_2}$}
  \AXC{}
  \RightLabel{$\ruleAXI$}
  \UIC{$\seqfrji{}{q_1,q_2,H}{p}$}
  \RightLabel{$\ruleJOINO$}
  \TIC{$\seqfrj{H}{q_1\lor q_2}$}
  \DP
  \]
  In the conclusion, $p$ is omitted since it does not occur as left
  formula in the right-most premise.  We can also build the following
  $\FRJof{G}$-derivation of an irregular sequent $\s_1$ having the
  goal formula $G$ in the right:
  \[
  \AXC{}
  \RightLabel{$\ruleAXI$}
  \UIC{$\seqfrji{}{p,q_2,H}{q_1}$}
  \AXC{}
  \RightLabel{$\ruleAXI$}
  \UIC{$\seqfrji{}{p,q_1,H}{q_2}$}
  \RightLabel{$\lor$}
  \BIC{$\seqfrji{}{p,H}{q_1\lor q_2}$}
  \RightLabel{$\ruleIMPi$}
  \UIC{$\s_1\;=\;\seqfrji{p,H}{}{(p\land H)\imp (q_1\lor q_2)}$}
  \DP
  \]
  Apparently, this contradicts the soundness of the calculus, since
  the formula $G$ is valid.  Actually, we conclude that the
  hypothesis of the soundness property~\ref{prop:soundi} does not
  hold, namely, $\s_1$ cannot be used to derive in $\FRJof{G}$ a
  regular sequent.
  \EndEs
\end{example}

Hereafter, to reduce the proof-search space, we assume that
$\FRJof{G}$-derivations comply with the following
restrictions~\ref{PS1}--\ref{PS4}:

\begin{enumerate}[label=(PS\arabic*),ref=(PS\arabic*)]
\item\label{PS1} In rule $\ruleIMPi$, we require that $\Lambda$ is a \emph{minimal}
  set satisfying the side condition, namely:
  $\Lambda'\subsetneq\Lambda$ implies
  $A\not\in\Clo{\Sigma\cup\Lambda'}$.

\item\label{PS2} In rule $\ruleIMPni$, we require that $\Theta$ is a \emph{maximal}
  set satisfying the side condition, namely:
  $\Theta\,\subsetneq\,\Theta'\,\subseteq\, \Clo{\Gamma}\cap\bG$ implies
  $A\in\Clo{\Theta'}$.

\item\label{PS3} In $\ruleJOINA$, for every $Y\in\Upsilon$ there is
  $Y\imp Z\in\Sfl{G}$.

\item\label{PS4} In $\ruleJOINO$, for every $Y\in\Upsilon$, either
  there is $Y\imp Z\in\Sfl{G}$ or there is $Y\lor Z\in\Sfr{G}$ or
  there is $Z\lor Y\in\Sfr{G}$.

\end{enumerate}

\noindent
In Ex.~\ref{ex:impPS1}, the applications of $\ruleIMPi$ complying
with~\ref{PS1} are the ones concerning $\Lambda_1$ and $\Lambda_2$; In
Ex.~\ref{ex:impPS2}, the applications of $\ruleIMPni$
satisfying~\ref{PS2} are the ones related to $\Theta_5$ and
$\Theta_6$.

\begin{figure}[t]\small
  \centering
  \[
  \begin{array}{l}
    S \;= \; H\,\imp\, \neg\neg p \lor \neg p
    \qquad
    H\;=\; (\neg\neg p\imp p)\, \imp\, \neg p\lor  p
    \\[1ex]
    \Sfl{S} \,=\, 
    \{\, H,\,\neg p\lor p,\,\neg\neg p,\,\neg p,\,p\,\} 
    \quad
    \Sfr{S} \,=\, 
    \{\, S,\,\neg\neg p\lor \neg p,\,\neg\neg p \imp p,\,
    \neg\neg p,\,\neg p,\, p,\, \bot   \,\}    
  \end{array}
  \]
  \leqnomode 
  \setcounter{equation}{0} 
  \begin{align}
    \label{exST:1}
    \tseqfrji{}{p,\:H,\:\neg\neg p,\:\neg p}{\bot}  && \ruleAXI 
    && \mbox{// Start}
    \\
    \label{exST:2}
    \tseqfrji{}{H,\:\neg\neg p,\:\neg p}{p}  && \ruleAXI
    \\[1ex]
    \hline 
    \label{exST:3}
    \tseqfrji{p}{H,\:\neg\neg p,\:\neg p}{\neg p} &&
    \ruleIMPi\;\eqref{exST:1}                                          
    && \mbox{// Iteration 1} 
    \\
    \label{exST:4}
    \tseqfrji{\neg\neg p}{H,\:\neg p}{\neg\neg p\imp p} &&
    \ruleIMPi\;\eqref{exST:2}                                             
    \\
    \label{exST:5}
    \tseqfrj{\neg p}{\bot} &&
    \ruleJOINA\;\eqref{exST:2} 
    \\[1ex]
    \hline  
    \label{exST:6}
    \tseqfrj{p,\:\neg\neg p}{\bot} &&
    \ruleJOINA\;\eqref{exST:3}
    && \mbox{// Iteration 2} 
    \\
    \label{exST:7}
    \tseqfrji{}{H}{\neg\neg p}  &&
    \ruleIMPni\;\eqref{exST:5}  
    \\[1ex]
    \hline 
    \label{exST:8}
    \tseqfrji{}{H,\neg\neg p}{\neg p} &&
    \ruleIMPni\;\eqref{exST:6}
    && \mbox{// Iteration 3}                                   
    \\[1ex]
    \hline  
    \label{exST:9}
    \tseqfrj{H,\:\neg\neg p}{p} &&
    \ruleJOINA\;\eqref{exST:4}\;\eqref{exST:8} 
    && \mbox{// Iteration 4}                        
    \\[1ex]
    \hline 
    \label{exST:10}
    \tseqfrji{}{H}{\neg\neg p\imp p} &&
    \ruleIMPni\;\eqref{exST:9}                              
    && \mbox{// Iteration 5}                        
    \\[1ex]
    \hline 
    \label{exST:11}
    \tseqfrj{H}{\neg\neg p\lor \neg p} &&
    \ruleJOINO\;\eqref{exST:7}\;\eqref{exST:8}\;\eqref{exST:10}
    && \mbox{// Iteration 6} 
    \\[1ex]
    \hline  
    \label{exST:12}
    \tseqfrj{H}{S} &&
    \ruleIMPi\;\eqref{exST:11}
    && \mbox{// Iteration 7}                                    
  \end{align}

  \vspace{4ex}
  \caption{The $\FRJof{S}$-derivation $\Dcal_S$ of $S$ and the model $\Mod{\Dcal_S}$.}
  \label{fig:frjST}
\end{figure}

We now provide some significant examples of derivations.

\begin{example}\label{ex:frjNishimura}
  Let us consider the following instances $S$ and $T$ of \emph{Scott}
  and \emph{Anti-Scott} principles, which are equivalent to Nishimura
  formulas $N_{10}$ and $N_{9}$ respectively~\cite{ChaZak:97} (the
  schema generating $N_i$ is given in Sec.~\ref{sec:rel}):
  \[
  S\;=\;\left((\neg\neg p\imp p) \imp \neg p\lor  p\right)\,\imp\,\neg\neg p \lor \neg p
  \qquad T\;=\; S \,\imp\,(\neg\neg p\imp p)\lor \neg\neg p 
  \]
  Both formulas are valid in Classical Logic but not in $\IPL$.
  Figs.~\ref{fig:frjST} and~\ref{fig:frjAST} show an
  $\FRJof{S}$-derivation $\Dcal_S$ of $S$ and an
  $\FRJof{T}$-derivation $\Dcal_T$ of $T$ respectively, in linear
  representation.  We populate the database of proved sequents
  according with the naive recipe of~\cite{VoronkovHAR:01}: we start
  by inserting the axioms; then we enter a loop where, at each
  iteration, we apply the rules to the sequents collected in previous
  steps.  For the sake of conciseness, we only show the sequents
  needed to get the goal.  We denote with $\s_{(k)}$ the sequent
  derived at line~$(k)$. The tree-like structure of derivations
  $\Dcal_S$ and $\Dcal_T$ are displayed in Figs.~\ref{fig:countST}
  and~\ref{fig:countAST} respectively.  We point out that $\Dcal_T$
  contains an application of $\ruleJOINO$ having four premises,
  namely:
  \[
  \begin{array}{c}
    \AXC{$\sigma_{\eqref{exAST:2}} \qquad \sigma_{\eqref{exAST:7}} \qquad\sigma_{\eqref{exAST:8}}\qquad\sigma_{\eqref{exAST:11}}$}
    \RightLabel{$\ruleJOINO$} 
    \UIC{$\sigma_{\eqref{exAST:13}}$}
    \DP
    \\[3ex]
    \begin{array}{ll}
      \sigma_{\eqref{exAST:2}}\;=\; \seqfrji{}{S,\:\neg \neg p \imp p,\:\neg\neg p,\:\neg p}{p} 
      \qquad
      &
      \sigma_{\eqref{exAST:7}}\;=\;\seqfrji{}{S,\:\neg \neg p \imp p}{\neg\neg p}
      \\
      \sigma_{\eqref{exAST:8}}\;=\;  \seqfrji{}{S,\:\neg \neg p \imp p,\:\neg\neg p}{\neg p} 
      &
      \sigma_{\eqref{exAST:11}}\;=\; \seqfrji{\neg \neg p \imp p}{S,\:\neg\neg p}{H}
      \\
      \sigma_{\eqref{exAST:13}}\;=\; \seqfrj{\neg \neg p \imp p,\:S}{ \neg p\lor p } 
      &      
      S\;=\;   H\,\imp\, \neg\neg p \lor \neg p 
    \end{array}
  \end{array}                                              
  \]   
  The sequent $\sigma_{\eqref{exAST:13}}$ is essential to build the
  countermodel since it corresponds to a world where both
  $\neg \neg p \imp p$ and $S$ are forced, while $\neg p\lor p$ is
  not; to get $\G=\{ \neg \neg p \imp p,S\}$ in the left of
  $\sigma_{\eqref{exAST:13}}$, we have to use premises $\s'$ such that
  $\G\subseteq\Lhs{\s'}$ (see Lemma~\ref{lemma:lhs}\ref{lemma:lhs:1}).
  Sequents $\sigma_{\eqref{exAST:8}}$ and $\sigma_{\eqref{exAST:2}}$
  are needed to obtain $\neg p\lor p$ in the right of
  $\sigma_{\eqref{exAST:13}}$, while sequents
  $\sigma_{\eqref{exAST:7}}$ and $\sigma_{\eqref{exAST:11}}$ are
  needed to support $\neg \neg p \imp p$ and $S$ respectively.  One
  can easily check that the side conditions~\ref{sc:J1}
  and~\ref{sc:J2} hold, thus the displayed application of rule
  $\ruleJOINO$ is sound.  Finally, we point out that it is not
  possible to obtain $\sigma_{\eqref{exAST:13}}$ using less than four
  premises.
 \EndEs   
\end{example}

   
\begin{figure}[t]\small
  \centering
  \[
  \begin{array}{lcl}
    \multicolumn{3}{c}{
      T\;=\; S \,\imp\, (\neg \neg p \imp p)\lor \neg\neg p
      \qquad
      S \;=\; H\,\imp\, \neg\neg p \lor \neg p
      \qquad
      H \;=\;(\neg \neg p \imp p)\,\imp\, \neg p \lor p}
    \\[1ex]
    \Sfl{T} & \;=\; &
    \{\: S,\: \neg \neg p \imp p,\: \neg \neg p\lor \neg p,\:\neg\neg p,\:\neg p,\:p\:\} 
    \\[.5ex]
    \Sfr{T} & \;=\; & 
    \{\: T,\:H,\: (\neg\neg p \imp p)\lor \neg \neg p,\:\neg \neg p \imp p,\:\neg\neg p,\:
    \neg p\lor p ,\:\neg p,\:p,\: \bot   \:\}    
  \end{array}
  \]
  \leqnomode 
  \setcounter{equation}{0} 
  \begin{align}
    \label{exAST:1}
    \tseqfrji{}{p,\:S,\:\neg \neg p \imp p,\:\neg\neg p,\:\neg p}{\bot}  && \ruleAXI
    \\
    \label{exAST:2}
    \tseqfrji{}{S,\:\neg \neg p \imp p,\:\neg\neg p,\:\neg p}{p}  && \ruleAXI
    \\[1ex]
    \hline 
    \label{exAST:3}
    \tseqfrji{p}{S,\:\neg \neg p \imp p,\:\neg\neg p,\:\neg p}{\neg p}  &&               
    \ruleIMPi\;\eqref{exAST:1}
    \\
    \label{exAST:4}
    \tseqfrji{\neg p}{p,\:S,\:\neg\neg p \imp p,\:\neg\neg p}{\neg\neg p}   &&                 
    \ruleIMPi\;\eqref{exAST:1}
    \\[1ex]
    \hline 
    \label{exAST:5}
    \tseqfrj{\neg p,\:\neg \neg p \imp p}{\bot}  &&  
\ruleJOINA\;\eqref{exAST:2} \;\eqref{exAST:4}
    \\
    \label{exAST:6}
    \tseqfrj{p,\:\neg\neg p}{\bot}  &&  
    \ruleJOINA\;\eqref{exAST:3}                                     
    \\[1ex]
    \hline 
\label{exAST:7}
    \tseqfrji{}{S,\:\neg \neg p \imp p}{\neg\neg p}  &&  
    \ruleIMPni\;\eqref{exAST:5} 
\\
    \label{exAST:8}
    \tseqfrji{}{S,\:\neg \neg p \imp p,\:\neg\neg p}{\neg p}  &&  
    \ruleIMPni\;\eqref{exAST:6} 
    \\[1ex]
    \hline 
        \label{exAST:9}
    \tseqfrji{}{S,\:\neg \neg p \imp p,\:\neg\neg p}{ \neg p\lor p }  &&  
    \lor \;\eqref{exAST:2}  \;\eqref{exAST:8}
     \\
    \label{exAST:10}
    \tseqfrj{\neg\neg p}{p} &&
     \ruleJOINA\;\eqref{exAST:8}
    \\[1ex]
    \hline 
    \label{exAST:11}
    \tseqfrji{\neg \neg p \imp p}{S,\:\neg\neg p}{H}  &&  
    \ruleIMPi\;\eqref{exAST:9}  
    \\ 
    \label{exAST:12}
    \tseqfrji{}{S}{\neg \neg p \imp p} &&
    \ruleIMPni\;\eqref{exAST:10}                                                  
    \\[1ex] 
    \hline 
    \label{exAST:13}
    \tseqfrj{\neg \neg p \imp p,\:S}{ \neg p\lor p }  &&  
    \ruleJOINO\;\eqref{exAST:2}\;\eqref{exAST:7}\;\eqref{exAST:8}\;\eqref{exAST:11}
    \\[1ex]
    \hline 
    \label{exAST:14} 
    \tseqfrji{}{S}{H}  &&  
    \ruleIMPni\;\eqref{exAST:13}
    \\[1ex]
    \hline 
    \label{exAST:15}
    \tseqfrj{S}{(\neg \neg p \imp p)\lor \neg\neg p}  &&  
    \ruleJOINO\;\eqref{exAST:7}\;\eqref{exAST:12}\;\eqref{exAST:14}     
    \\[1ex]
    \hline 
    \label{exAST:16}
    \tseqfrj{S}{T}  &&  
    \ruleIMPi\;\eqref{exAST:15}
  \end{align}

  \vspace{4ex}
  \caption{The $\FRJof{T}$-derivation $\Dcal_T$ of $T$ and the model $\Mod{\Dcal_T}$.}
  \label{fig:frjAST}
\end{figure}

\begin{example}\label{ex:frjKP}
  Another significant example is the $\FRJof{K}$-derivation $\Dcal_K$
  of the instance
  $K=(\neg a \imp b\lor c)\imp (\neg a \imp b) \lor (\neg a \imp c)$
  of Kreisel-Putnam principle~\cite{ChaZak:97} shown in
  Fig.~\ref{fig:frjKP}.  Differently from the formulas $S$ and $T$ in
  the previous example, the formula $K$ contains three propositional
  variables, which give rise to eight axioms.  In the figure, we only
  consider the sequents needed to prove the goal.
 \EndEs
\end{example}


\begin{figure}[t]
  \centering
  \[
  \begin{array}{l}
    K \;=\; K_0\imp  K_1
    \qquad K_0\;=\;  \neg a \imp b\lor c
    \qquad K_1 \;=\; (\neg a \imp b)\lor (\neg a \imp c)
    \\[.5ex]
    \Sfl{K} \;=\; 
    \{\, K_0,\, \neg a,\, a,\, b,\, c\,\}                     
    \qquad
    \Sfr{K} \;=\; 
    \{\,K,\, K_1,\,  \neg a \imp b,\, \neg a \imp c,\,
    \,  a,\, b,\, c ,\,\bot\,\}    
  \end{array}
  \]  
  \leqnomode 
  \setcounter{equation}{0} 
  \begin{align}\small
    \label{exKP:1}
    \tseqfrji{}{a,\,b,\,c,\,K_0,\,\neg a  }{\bot}  && \ruleAXI
    \\
    \label{exKP:2}
    \tseqfrji{}{b,\,c,\,K_0,\,\neg a  }{a}  && \ruleAXI
    \\[1ex]
    \hline
    \label{exKP:3}
    \tseqfrji{a}{b,\,c,\,K_0,\,\neg a  }{\neg a}  &&
    \ruleIMPi\;\eqref{exKP:1}                                                                 
    \\
    \label{exKP:4}
    \tseqfrj{c,\,\neg a }{b}  &&
    \ruleJOINA\;\eqref{exKP:2}
    \\
    \label{exKP:5}
    \tseqfrj{b,\,\neg a }{c}  &&
    \ruleJOINA\;\eqref{exKP:2}
    \\[1ex]
    \hline
    \label{exKP:6}
    \tseqfrj{a,\,b,\,c,\,K_0 }{\bot}  &&
    \ruleJOINA\;\eqref{exKP:3}
    \\
    \label{exKP:7}
    \tseqfrji{}{c,\,K_0}{\neg a\imp b}  &&
    \ruleIMPni\;\eqref{exKP:4}
    \\
    \label{exKP:8}
    \tseqfrji{}{b,\,K_0}{\neg a\imp c}  &&
    \ruleIMPni\;\eqref{exKP:5}
    \\[1ex]
    \hline
    \label{exKP:9}
    \tseqfrji{}{b,\,c,\,K_0 }{\neg a}  &&
    \ruleIMPni\;\eqref{exKP:6}
    \\[1ex]
    \hline
    \label{exKP:10}
    \tseqfrj{K_0}{K_1}  &&
    \ruleJOINO\;\eqref{exKP:7}\;\eqref{exKP:8}\;\eqref{exKP:9}
    \\[1ex]
    \hline
    \label{exKP:11}
    \tseqfrj{ K_0  }{K}   &&
    \ruleIMPi\;\eqref{exKP:10}
  \end{align}

  \vspace{4ex}
  \caption{The $\FRJof{K}$-derivation $\Dcal_K$ of $K$ and the model $\Mod{\Dcal_K}$.}
  \label{fig:frjKP}
\end{figure}


\subsection{Countermodels and soundness of $\FRJof{G}$}\label{sec:models}

We show that we can extract from an $\FRJof{G}$-derivation $\Dcal$ of
$G$ a countermodel for $G$.  We call \emph{p-sequent} (\emph{prime
  sequent}) of $\Dcal$ any regular sequent occurring in $\Dcal$ which
is either an axiom or the conclusion of a join rule.  Given an
$\FRJof{G}$-derivation $\Dcal$ of $G$ let
$\Mod{\Dcal}\,=\,\stru{\,\PS{\Dcal},\,\leq,\,\rho,V\,}$ be the
structure where:
\begin{itemize}
\item $\PS{\Dcal}$ is the set of all p-sequents occurring in $\Dcal$;
\item for every $\s_1,\s_2\in\PS{\Dcal}$, $\s_1\leq\s_2$ iff
  $\s_2\mapstos \s_1$;
\item $\rho$ is the minimum of $\PS{\Dcal}$ w.r.t. $\leq$;
\item $V$ maps $\s\in \PS{\Dcal}$ to the set
  $V(\s)\,=\,\Lhs{\s}\cap\PV$.
\end{itemize}
One can check that $\Mod{\Dcal}$ is a model.  In particular, since the
root sequent of $\Dcal$ is regular, there exists $\rho\in\PS{\Dcal}$
such that $\s_p\mapstos \rho$, for every $\s_p\in\PS{\Dcal}$, hence
$\rho$ is the minimum of $\PS{\Dcal}$ w.r.t. $\leq$.  Moreover, by
Lemma~\ref{lemma:lhs}\ref{lemma:lhs:3} and~\ref{propClo:5},
$\s_1\leq\s_2$ implies $V(\s_1)\subseteq V(\s_2)$, hence the
definition of $V$ is sound.  
We call $\Mod{\Dcal}$ the \emph{model extracted from $\Dcal$}.

For every regular sequent $\s$ occurring
in $\Dcal$, let $\phi(\s)$ be the p-sequent {immediately above $\s$},
namely:
\[
\begin{array}{lcl}
  \phi(\s) \;=\; \s_p  
  &\quad\mbox{iff}\quad&
  \s_p\in\PS{\Dcal}\;\mbox{and}\; \s_p\mapstos \s\;\mbox{and}
  \\
  && 
  \mbox{for every $\s'_p\in\PS{\Dcal}$, $\s_p\mapstos
    \s'_p \mapstos \s$ implies $\s'_p=\s_p$}.
\end{array}
\]
It is easy to check that:

\begin{itemize}
\item p-sequents are fixed points of $\phi$, i.e., $\s_p\in\PS{\Dcal}$
  implies $\phi(\s_p)=\s_p$;

\item $\phi$ is a surjective map from the set of regular sequents of
$\Dcal$ onto $\Mod{\Dcal}$; 

\item
 if $\s_1$ and $\s_2$ are regular
and $\s_1\mapstos \s_2$, then $\phi(\s_2) \leq \phi(\s_1)$.
\end{itemize}

\noindent
We call $\phi$ the \emph{map associated with $\Dcal$}.
The following lemma is the key to prove the soundness
properties~\ref{prop:soundr} and \ref{prop:soundi} of $\FRJof{G}$
stated at the beginning of Sec.~\ref{sec:FRJ};
by $\Sfm{C}$ we denote the set $\Sf{C}\setminus\{C\}$.

\begin{lemma}\label{lemma:soundFRJ}
  Let $\Dcal$ be an $\FRJof{G}$-derivation of $G$, let
  $\Mod{\Dcal}$ be the model extracted from $\Dcal$ and $\phi$ the map associated with $\Dcal$. 
For every sequent
  $\s$ occurring in $\Dcal$:

  \begin{enumerate}[label=(\roman*),ref=(\roman*)]
  \item\label{lemma:soundFRJ:1} if $\s= \seqfrj{\G}{C}$, then $\phi(\s)\forcing\G$ and
    $\phi(\s) \nforcing C$;

  \item\label{lemma:soundFRJ:2} if $\s= \seqfrji{\Sigma}{\Theta}{C}$, let $\s_p\in\PS{\Dcal}$
    such that~$\s\mapsto\s_p$ and $\s_p\forcing\Sigma\cap\Sfm{C}$;
    then $\s_p\nforcing C$.\qed
  \end{enumerate}  
\end{lemma}

The proof of the lemma
is outlined in~\cite{FerFio:2017} and is presented in full
details in Appendix~\ref{sec:soundFRJ}.
As an immediate consequence, we get:

\begin{theorem}\label{theo:soundFRJ}
  Let $\Dcal$ be an $\FRJof{G}$-derivation of $G$. Then, $\Mod{\Dcal}$
  is a countermodel for $G$.  
\end{theorem}

\begin{proof}
By definition, $\Dcal$ is an   $\FRJof{G}$-derivation of a regular sequent
$\s=\seqfrj{\G}{G}$.
Let $\phi$ be the map associated with $\Dcal$.
Then, $\phi(\s)$ is the root of $\Mod{\Dcal}$ and,
by Lemma~\ref{lemma:soundFRJ}\ref{lemma:soundFRJ:1},
we have
 $\phi(\s) \nforcing G$.
We conclude that $\Mod{\Dcal}$
  is a countermodel for $G$. \qed
\end{proof}

Accordingly, if $G$ is provable in $\FRJof{G}$, then $G$ is not valid,
and this proves the Soundness of $\FRJof{G}$ as stated in
Theorem~\ref{theo:FRJsound}.

\begin{figure}[t]
  \centering
  \hspace{2em}
  \begin{minipage}{25em}
    \center
    \[
    \AXC{}
    \UIC{$\s_\eqref{exST:2}$}
    \RightLabel{$\ruleJOINA$}
    \UIC{$\s_\eqref{exST:5}*$}
    \UIC{$\s_\eqref{exST:7}$}
    \AXC{}
    \UIC{$\s_\eqref{exST:1}$}
    \UIC{$\s_\eqref{exST:3}$}
    \RightLabel{$\ruleJOINA$}
    \UIC{$\s_\eqref{exST:6}*$}
    \UIC{$\s_\eqref{exST:8}$}
    \AXC{}
    \UIC{$\s_\eqref{exST:2}$}
    \UIC{$\s_\eqref{exST:4}$} 
    \AXC{}
    \UIC{$\s_\eqref{exST:1}$}
    \UIC{$\s_\eqref{exST:3}$}
    \RightLabel{$\ruleJOINA$}
    \UIC{$\s_\eqref{exST:6}*$}
    \UIC{$\s_\eqref{exST:8}$}
    \RightLabel{$\ruleJOINA$}
    \BIC{$\s_\eqref{exST:9}*$}
    \UIC{$\s_\eqref{exST:10}$}
    \RightLabel{$\ruleJOINO$}
    \TIC{$\s_\eqref{exST:11}*$}
    \RightLabel{\hspace{10em}\mbox{$*$: p-sequents}}
    \UIC{$\s_\eqref{exST:12}$}
    \DP
    \]     
  \end{minipage}
  \\[4ex]
  \begin{minipage}{12em}
    \center
    \begin{tikzpicture}
      \path[scale=.50,grow'=up,every node/.style={fill=gray!30,rounded corners},
      level 1/.style = {sibling distance = 30mm}, 
      edge from parent/.style={black,draw}]
      [level distance=20mm]  
      node{$\s_{\eqref{exST:11}}$: }
      child{
        node{$\s_{\eqref{exST:5}}$: }
      }
      child{
        node{$\s_{\eqref{exST:9}}$: }
        child{
          node{$\s_{\eqref{exST:6}}$: $p$}
        }
      }
      ;
    \end{tikzpicture}
  \end{minipage}
  \qquad
  \begin{minipage}{20em}
    $\s_{(k)}$ refers to the sequent at line~$(k)$ in Fig.~\ref{fig:frjST}
    \\[2ex]
    $\phi(\s)\;=\;\s$, for every p-sequent $\s$
    \\
    $\phi(\s_\eqref{exST:12})\;=\; \s_\eqref{exST:11}$
  \end{minipage}
  \vspace{4ex}
  \caption{The model $\Mod{\Dcal_S}$ (see Fig.~\ref{fig:frjST}).}
  \label{fig:countST}
\end{figure}

\begin{example}\label{ex:count}
  Let us consider the formulas $S$, $T$ and $K$ in
  examples~\ref{ex:frjNishimura} and~\ref{ex:frjKP}.  The models
  $\Mod{\Dcal_S}$, $\Mod{\Dcal_T}$, $\Mod{\Dcal_K}$ and the related
  maps $\phi$ are shown in Figs.~\ref{fig:countST}, \ref{fig:countAST}
  and~\ref{fig:countKP} respectively.  The bottom world is the root
  and $\s<\s'$ iff the world $\s$ is drawn below $\s'$.  For each
  $\s$, we display the set $V(\s)$. As an example, in
  Fig.~\ref{fig:countST} we have $V(\s_\eqref{exST:6})=\{p\}$, while
  $V(\s_\eqref{exST:5})=V(\s_\eqref{exST:9})=V(\s_\eqref{exST:11})=\emptyset$.
  \EndEs
\end{example}

\subsection{Termination}\label{sec:wg}

To conclude the presentation of $\FRJof{G}$,
we exhibit a \emph{weight} function $\wgname$ on sequents of
$\FRJof{G}$ such that, after the application of a rule, the weight of
sequents decreases; accordingly, the naive proof-search procedure
always terminates, even if we do not implement any redundancy check.
Let $\s_1$ and $\s_2$ be two $\FRJof{G}$-sequents and, for
$k\in\{1,2\}$, let  $\G_k=\Lhs{\s_k}$ and   $C_k=\Rhs{\s_k}$;
by  $\card{\_}$ we denote the cardinality function.
We show that the following properties holds:

\begin{enumerate}[label=(\arabic*), ref=(\arabic*)]
  
\item\label{propTermFRJ:1} $\s_1\mapstorz{\Rcal} \s_2$ implies
  $\card{\Clo{\G_2}\cap\Sfl{G}}\;\leq\; \card{\Clo{\G_1}\cap\Sfl{G}}$;

\item\label{propTermFRJ:2} $\s_1\mapstorz{\ruleIMPni} \s_2$ implies
  $\card{\Clo{\G_2}\cap\Sfl{G}}\;<\; \card{\Clo{\G_1}\cap\Sfl{G}}$;

\item\label{propTermFRJ:3} $\s_1\mapstorz{\Rcal} \s_2$ and $\Rcal$ is
  not a join rule imply $\size{G}-\size{C_2} < \size{G}-\size{C_1}$.

\end{enumerate}

\noindent
Let $\s_1\mapstorz{\Rcal} \s_2$.  By
Lemma~\ref{lemma:lhs}\ref{lemma:lhs:2} and~\ref{propClo:6}, we have
$\Clo{\G_2}\subseteq\Clo{\G_1}$, hence Point~\ref{propTermFRJ:1}
holds.  If $\Rcal=\,\ruleIMPni$, then $C_2= A\imp B$, where
$A\in\Sfl{G}$ and $A\in \Clo{\G_1}$ and $A\not\in\Clo{\G_2}$; this
proves Point~\ref{propTermFRJ:2}.  If $\Rcal$ is not a join rule, then
$\size{C_2} > \size{C_1}$, hence Point~\ref{propTermFRJ:3} holds.
Properties~\ref{propTermFRJ:1}--\ref{propTermFRJ:3} suggest that we
can define $\wg{\s}$ as the triple of non-negative integers:
\[
\begin{array}{l}
  \wg{\s}\,=\,\stru{\;\card{\,\Clo{\G}\,\cap\,\Sfl{G}\,}\,,\,\tp{\s}\,,\,\size{G}-\size{C}\;    }
  \qquad 
  \tp{\s} \;= \;
  \begin{cases}
    0\;\mbox{if $\s$ is  regular }
    \\
    1\;\mbox{otherwise}
  \end{cases}
  \\
  \G\;=\;\Lhs{\s}
  \qquad
  C\;=\;\Rhs{\s}
\end{array}
\]
Let $\prec$ be the standard lexicographic order on triples of
integers; we get:

\begin{lemma}\label{lemma:wg}
  $\s_1\mapsto \s_2$ implies $\stru{0,0,0}\preceq \wg{\s_2} \prec
  \wg{\s_1}$.  \qed
\end{lemma}
Note that the component $\tp{\s}$ accommodates the case where $\s_2$
is the conclusion of a join rule. Indeed, in this case $\s_2$ is
regular and  $\s_1$ is irregular, hence $\tp{\s_2}<\tp{\s_1}$.



\begin{figure}[t]
  \centering
  \hspace{2em}
  \begin{minipage}{25em}
    \center
  \[
  \AXC{}
  \UIC{$\s_\eqref{exAST:2}$}
  \AXC{}
  \UIC{$\s_\eqref{exAST:1}$}
  \UIC{$\s_\eqref{exAST:4}$}
  \RightLabel{$\ruleJOINA$} 
  \BIC{$\s_\eqref{exAST:5}*$}
  \UIC{$\s_\eqref{exAST:7}$}
  \AXC{}
  \UIC{$\s_\eqref{exAST:1}$}
  \UIC{$\s_\eqref{exAST:3}$}
  \RightLabel{$\ruleJOINA$} 
  \UIC{$\s_\eqref{exAST:6}*$}
  \UIC{$\s_\eqref{exAST:8}$}
  \RightLabel{$\ruleJOINA$} 
  \UIC{$\s_\eqref{exAST:10}*$}
  \UIC{$\s_\eqref{exAST:12}$}
  \AXC{}
  \UIC{$\s_\eqref{exAST:2}$}
  \AXC{$\vdots$}
  \noLine
  \UIC{$\s_\eqref{exAST:7}$ }
  \AXC{$\vdots$}
  \noLine
  \UIC{$\s_\eqref{exAST:8}$}
  \AXC{}
  \UIC{$\s_\eqref{exAST:2}$}
  \AXC{$\vdots$}
  \noLine
  \UIC{$\s_\eqref{exAST:8}$}
  \insertBetweenHyps{\hskip -3pt}
  \BIC{$\s_\eqref{exAST:9}$}
  \UIC{$\s_\eqref{exAST:11}$}
  \insertBetweenHyps{\hskip -3pt}
  \RightLabel{$\ruleJOINO$} 
  \QuaternaryInfC{$\s\eqref{exAST:13}*$}
  \UIC{$\s_\eqref{exAST:14}$}
  \RightLabel{$\ruleJOINO$}            
  \TIC{$\s_\eqref{exAST:15}*$}
  \RightLabel{\hspace{10em}\mbox{$*$: p-sequents}}       
  \UIC{$\s_\eqref{exAST:16}$}
  \DP
  \]
  \end{minipage}
  \\[4ex]
  \begin{minipage}{12em}
    \center
    \begin{tikzpicture}
      \path[scale=.50,grow'=up,every node/.style={fill=gray!30,rounded corners},
      level 1/.style = {sibling distance = 60mm},  
      edge from parent/.style={black,draw}]
      node{$\s_{\eqref{exAST:15}}$: } 
      [level distance=20mm]  
      child{
        node{$\s_{\eqref{exAST:10}}$:}
        child{
          node(a){$\s_{\eqref{exAST:6}}$: $p$} 
        }  
      }
      child{
        node(b){$\s_{\eqref{exAST:13}}$:}  
        child{
          node{$\s_{\eqref{exAST:5}}$:}
        }    
      }
      ;
      \draw (a) -- (b);   
    \end{tikzpicture}
  \end{minipage}
  \qquad
  \begin{minipage}{20em}
    $\s_{(k)}$ refers to the sequent at line~$(k)$ in Fig.~\ref{fig:frjAST}
    \\[2ex]
    $\phi(\s)\;=\;\s$, for every p-sequent $\s$
    \\
    $\phi(\s_\eqref{exAST:16})\;=\; \s_\eqref{exAST:15}$
  \end{minipage}
  
  \vspace{4ex}
  \caption{The model $\Mod{\Dcal_T}$ (see Fig.~\ref{fig:frjAST}).}
  \label{fig:countAST}
\end{figure}

We can exploit $\wgname$ to set a bound on the height of derivations
and of the extracted models.  As usual a \emph{branch} of an
$\FRJof{G}$-derivation is a sequence of sequents $\s_1,\dots, \s_m$
such that $\s_1\mapstoz \s_2\mapstoz\cdots\mapstoz \s_m$.

\begin{lemma}\label{lemma:branch}
  Let $\Dcal$ be an $\FRJof{G}$-derivation, let $\Bcal$ be a branch of
  $\Dcal$ and $N=\size{G}$.  Then:
  \begin{enumerate}[label=(\roman*), ref=(\roman*)]
  \item\label{lemma:branch:1} the length of the branch $\Bcal$ is
    $O(N^2)$;
    
  \item\label{lemma:branch:2} $\Bcal$ contains $N$ p-sequents at most.
  \end{enumerate}
\end{lemma}

\begin{proof}
  Let $\Bcal=\s_1,\dots, \s_m$
  where $\s_1\mapstoz \s_2\mapstoz\cdots\mapstoz \s_m$.  By
  Lemma~\ref{lemma:wg}, we have
  $\wg{\s_{m}}\prec\wg{\s_{m-1}}\prec\cdots\prec \wg{\s_1}$.  For
  $j\in\{1,\dots,m\}$, let $\wg{\s_{j}}=\stru{k_j,t_j,c_j}$.  Since
  $\card{\Sfl{G}}\leq N-1$, it holds that $0\leq k_j\leq N-1$,
  $t_j\in\{0,1\}$ and $0\leq c_j\leq N$, and this
  implies~\ref{lemma:branch:1}.

  Let $\s_h=\seqfrj{\Gamma_h}{C_h}$ and $\s_l=\seqfrj{\Gamma_l}{C_l}$
  be two distinct p-sequents of $\Bcal$, where $h<l$.  For
  $j\in\{h,l\}$, we have $\wg{\s_j}=\stru{k_j,0,c_j}$, where
  $k_j=\card{\Clo{\G_j}\cap\Sfl{G}}$.  Since $\s_l$ is the conclusion
  of a join rule, the sequent $\s_{l-1}$ is irregular.  Accordingly,
  there exists $u\in\{h,\dots,l-1\}$ such that $\s_u$ is regular and
  $\s_{u+1}$ is irregular. This means that $\s_{u+1}$ is obtained by
  an application of rule $\ruleIMPni$, hence the branch has the
  following form:
  \[
  \begin{array}{c}
    \s_h\,=\,\seqfrj{\Gamma_h}{C_h}
    \quad\mapstos\quad
    \s_u\,=\,\seqfrj{\G}{B}
    \quad\mapstorz{\ruleIMPni}\quad
    \s_{u+1}\,=\,   \seqfrji{}{\Theta}{A\imp B}
    \quad\mapsto\quad
    \s_l\,=\,\seqfrj{\Gamma_l}{C_l}
    \\[1ex]
    A\in\Sfl{G}\quad\mbox{and}\quad A\in\Clo{\G}\quad\mbox{and}\quad A\not\in\Clo{\Theta}
  \end{array}
  \]  
  By Lemma~\ref{lemma:lhs}\ref{lemma:lhs:3} and
  property~\ref{propClo:6} of closures, we get
  $\Clo{\G_l}\subseteq\Clo{\G_h}$, $A\in\Clo{\G_h}$ and
  $A\not\in\Clo{\G_l}$.  This implies
  $\card{\Clo{\G_l}\cap\Sfl{G}} < \card{\Clo{\G_h}\cap\Sfl{G}}$,
  namely $k_l < k_h\leq N$.  We conclude that the branch $\Bcal$
  cannot contain more than $N$ distinct p-sequents.  
\end{proof}

Let $\Dcal$ be an $\FRJof{G}$-derivation and let $\s$ be a sequent
occurring in $\Dcal$. The \emph{height of $\s$} (in $\Dcal$), denoted
by $\height{\s}$, is the maximum distance from $\s$ to an axiom
sequent of $\Dcal$, namely:
\[
\height{\s}\;=\;
\begin{cases}
  0\quad  \mbox{if $\s$  is an axiom sequent}  
  \\
  1\;+\;\max\{\,\height{\s'}~|~\mbox{$\s'$ occurs in $\Dcal$ and $\s'\mapsto \s$}\,\}\quad
  \mbox{otherwise}
\end{cases}
\]
The \emph{height of $\Dcal$}, denoted by $\height{\Dcal}$, is the
height of the root sequent of $\Dcal$.  Let $\K=\stru{P,\leq,\rho, V}$
be a model and let $\a\in P$; the \emph{height of $\a$} (in $\K$),
denoted by $\height{\a}$, is the maximum distance from $\a$ to a final
world of $\K$, namely:
\[
\height{\a}\;=\;
\begin{cases}
  0\quad  \mbox{if $\a$  is final}  
  \\
  1\;+\;\max\{\,\height{\b}~|~\mbox{$\b\in P$ and $\a < \b$}\,\}\quad
  \mbox{otherwise}
\end{cases}
\]
The \emph{height of $\K$}, denoted by $\height{\K}$, is the height of
$\rho$.  As an immediate consequence of Lemma~\ref{lemma:branch} and
of the definition of $\Mod{\Dcal}$ we get an upper bound on the height
of $\FRJof{G}$-derivations and on the height of extracted
countermodels:

\begin{theorem}\label{theo:height}
  Let $\Dcal$ be an $\FRJof{G}$-derivation and $N=\size{G}$.  Then:
  \begin{enumerate}[label=(\roman*), ref=(\roman*)]
  \item\label{prop:height:1} $\height{\Dcal}\,=\,O(N^2)$.

  \item\label{prop:height:2} $\height{\Mod{\Dcal}}\,\leq\, N$.  \qed
  \end{enumerate}
\end{theorem}

In Sec.\ref{sec:minimality} we present further properties
on the depth of the extracted countermodels.

\begin{figure}[t]
  \centering
  \begin{tabular}{p{40ex}l}
    \hspace{-5ex}\begin{minipage}{40ex}
      \[
      \AXC{}
      \UIC{$\s_\eqref{exKP:2} $}
      \RightLabel{$\ruleJOINA$}
      \UIC{$\s_\eqref{exKP:4}*$}
      \UIC{$\s_\eqref{exKP:7}$}
      \AXC{}
      \UIC{$\s_\eqref{exKP:2}$}
      \RightLabel{$\ruleJOINA$}
      \UIC{$\s_\eqref{exKP:5}*$}
      \UIC{$\s_\eqref{exKP:8}$}
      \AXC{}
      \UIC{$\s_\eqref{exKP:1}$}
      \UIC{$\s_\eqref{exKP:3}$}  
      \RightLabel{$\ruleJOINA$}
      \RightLabel{$\ruleJOINA$}
      \UIC{$\s_\eqref{exKP:6}*$} 
      \UIC{$\s_\eqref{exKP:9} $} 
      \RightLabel{$\ruleJOINA$}
      \TIC{$\s_\eqref{exKP:10}* $}
      \RightLabel{\hspace{3em}\mbox{$*$: p-sequents}}       
      \UIC{$\s_\eqref{exKP:11}$}
      \DP
      \] 
    \end{minipage}
    &
    \begin{minipage}{10em}
      \begin{tikzpicture}
        \path[scale=.42,grow'=up,every node/.style={fill=gray!30,rounded corners},
        level 1/.style = {sibling distance = 45mm},  
        edge from parent/.style={black,draw}]
        node{$\s_{\eqref{exKP:10}}$: } 
       [level distance=25mm]  
        child{
          node{$\s_{\eqref{exKP:4}}$: $c$ }
        }
        child{
          node{$\s_{\eqref{exKP:5}}$: $b$ }
        }
        child{
          node{$\s_{\eqref{exKP:6}}$: $a,\,b,\,c$}
        }
        ;
      \end{tikzpicture}
    \end{minipage}
  \end{tabular}
  \vspace{4ex}
  \caption{The model $\Mod{\Dcal_K}$ (see Fig.~\ref{fig:frjKP}).}
  \label{fig:countKP}  
\end{figure}


\section{The  proof-search procedure and saturated Databases}
\label{sec:proofsearch}

The plain proof-search procedure outlined in Ex.~\ref{ex:frjNishimura}
suffers from the plethora of redundant sequents generated at each
iteration. To reduce the size of the database of proved sequents, we
introduce the notion of subsumption.

Given two sequents $\s_1$ and $\s_2$, we say that $\s_2$ \emph{subsumes}
$\s_1$, and we write $\s_1\sqsubseteq \s_2$ (or $\s_2\sqsupseteq
\s_1$) iff one of the following conditions hold:

\begin{enumerate}
\item $\s_1=\seqfrj{\G_1}{C}$ and  $\s_2=\seqfrj{\G_2}{C}$ and $\G_1\subseteq \G_2$;
\item $\s_1=\seqfrji{\Sigma}{\Theta_1}{C}$ and
  $\s_2=\seqfrji{\Sigma}{\Theta_2}{C}$ and $\Theta_1\subseteq
  \Theta_2$.
\end{enumerate}

\noindent
One can easily check that the relation $\sqsubseteq$ is a partial
order on the sets of $\FRJof{G}$-sequents, since it satisfies
reflexivity, antisymmetry and transitivity.  By $\s_1\sqsubset \s_2$
(or $\s_2\sqsupset \s_1$) we mean that $\s_1\sqsubseteq \s_2$ and
$\s_1\neq \s_2$.

\begin{lemma}\label{lemma:subsRules}
  Let 
  \[
  \AXC{$\s_1 \;\cdots\; \s_n$}
  \RightLabel{$\Rcal$}
  \UIC{$\s$}
  \DP
  \]
  be an instance of a rule of $\FRJof{G}$ and let $\s_1\sqsubseteq
  \s'_1$, \dots, $\s_n\sqsubseteq \s'_n$.  Then, 
  \[
  \AXC{$\s'_1 \;\cdots\; \s'_n$}
  \RightLabel{$\Rcal$}
  \UIC{$\s'$}
  \DP
  \]
  is an instance of $\Rcal$ in $\FRJof{G}$ where $\s\sqsubseteq \s'$. 
\end{lemma}

\begin{proof}
 The assertion  can be proved  by inspecting the rules of
  $\FRJof{G}$. As an example, let us consider the case of the rule $\lor$:
  \[
  \AXC{$\s_1=\seqfrji{\Sigma_1}{\Theta_1}{C_1}$}
  \AXC{$\s_2=\seqfrji{\Sigma_2}{\Theta_2}{C_2}$}
  \RightLabel{$\lor$}
  \BIC{$\s=\seqfrji{\Sigma_1,\Sigma_2}{\Theta_1\cap\Theta_2}{C_1\lor C_2}$}
  \DP
  \qquad 
  \begin{minipage}{10em}
    $\Sigma_1\,\subseteq\,\Sigma_2\,\cup\,\Theta_2$
\\[0.5ex]
    $\Sigma_2\,\subseteq\,\Sigma_1\,\cup\,\Theta_1$     
  \end{minipage}
  \]
  Let $\s'_1$ and $\s'_2$ be such that $\s_1\sqsubseteq\s'_1$ and
  $\s_2\sqsubseteq\s'_2$.  This means that
  $\s'_1=\seqfrji{\Sigma_1}{\Theta'_1}{C_1}$ and
  $\s'_2=\seqfrji{\Sigma_2}{\Theta'_2}{C_2}$ with
  $\Theta_1\subseteq\Theta'_1$ and $\Theta_2\subseteq\Theta'_2$. Hence
  \[
  \AXC{$\s'_1=\seqfrji{\Sigma_1}{\Theta'_1}{C_1}$}
  \AXC{$\s'_2=\seqfrji{\Sigma_2}{\Theta'_2}{C_2}$}
  \RightLabel{$\lor$}
  \BIC{$\s'=\seqfrji{\Sigma_1,\Sigma_2}{\Theta'_1\cap\Theta'_2}{C_1\lor C_2}$}
  \DP
  \qquad 
  \begin{minipage}{10em}
    $\Sigma_1\,\subseteq\,\Sigma_2\,\cup\,\Theta'_2$
\\[.5ex]
    $\Sigma_2\,\subseteq\,\Sigma_1\,\cup\,\Theta'_1$     
  \end{minipage}
  \]
  is an instance of the rule $\lor$ and, since
  $\Theta_1\cap\Theta_2\subseteq\Theta'_1\cap\Theta'_2$, we get 
  $\s\sqsubseteq\s'$.
\end{proof}


\SetKwProg{Fn}{Function}{}{endFun}  
\SetKwInOut{Input}{input}
\SetKwInOut{Output}{output}
\SetKwInOut{Assumption}{assumption}
\SetKwFunction{FSearch}{FSearch}

\begin{function}[t]
  \DontPrintSemicolon
  \Fn{\FSearch{$G$}}{
    \Input{The goal formula $G$} 
    \Output{An $\FRJof{G}$-derivation of $G$ or the db of proved sequents $\DBof{G}$}
    $\DBof{G}\ass \emptyset$\hspace{16.8ex}\tcp{db of proved sequents}
    $\Ical\ass$all axiom sequents\hspace{1ex}\tcp{db of sequents proved in the last iteration}
    $\Pi\ass$all axiom rules\hspace{4.5ex}\tcp{db of generated derivations}
    \While{$\Ical\not=\emptyset$ and $\DBof{G}$ does not contain a sequent of
      the kind $\seqfrj{\G}{G}$\label{fun:mainLoopBegin}}{ 
      $\DBof{G}\ass \DBof{G}\cup\Ical$~~\tcp{update $\DBof{G}$}\label{fun:updateBDG}
      $\Ical\ass \emptyset$\\
      \For{every instance $\Rcal$ of a rule of $\FRJof{G}$ active in $\DBof{G}$}{ %
        let $\s$ be the conclusion of $\Rcal$\\
        \uIf{$\DBof{G}$ does not contain $\s'$ such that $\s\sqsubseteq\s'$\label{fun:mainLoopUpdate}} {%
          \tcp{update $\Ical$ and $\Pi$}
          $\Ical\ass \Ical\cup\{\s\}$\label{fun:updateBegin}\\
          let $\Dcal_1,\dots\Dcal_1$ be the $\FRJof{G}$-derivations of $\s_1,\dots,\s_n$ in $\Pi$\\ 
          let $\Dcal$ be the derivation of $\s$ built applying rule $\Rcal$ to $\Dcal_1,\dots,\Dcal_n$\\
          $\Pi\ass\Pi\cup\{\Dcal\}$\label{fun:updateEnd}
        } 
      } 
    }\label{fun:mainLoopEnd} 
    \tcp{$\Ical=\emptyset$ or $\DBof{G}$ contains a sequent of the kind $\seqfrj{\G}{G}$}
    \uIf{$\DBof{G}$ contains a sequent of the kind $\seqfrj{\G}{G}$}
    {return the derivation of $\seqfrj{\G}{G}$ in $\Pi$} 
    \uElse{return   $\DBof{G}$}
  }

  \caption{The forward proof-search procedure FSearch}\label{fun:FSearch}
\end{function}

In Fig.~\ref{fun:FSearch} we give a high-level description of the
proof-search procedure $\FSearch$ for $\FRJof{G}$ exploiting
subsumption.  We denote with $\DBof{G}$ the database storing the
sequents proved by the proof-search procedure.

Let us consider the following instance $\Rcal$   of a rule of $\FRJof{G}$
\[
\AXC{$\s_1 \;\cdots\; \s_n$}
\RightLabel{$\Rcal$}
\UIC{$\s$}
\DP
\]
We say that $\Rcal$ is \emph{active} in $\DBof{G}$ if
$\{\s_1,\dots,\s_n\}\subseteq \DBof{G}$.  $\FSearch{G}$ maintains the
database of proved sequents $\DBof{G}$, the database $\Pi$ of
generated derivations and the database $\Ical$ of sequents proved in
the last performed iteration. $\Ical$ and $\Pi$ are initialized
applying the axiom rules of $\FRJof{G}$. The main loop
(lines~\ref{fun:mainLoopBegin}-\ref{fun:mainLoopEnd}) performs the
iterations of the proof-search procedure.  At each iteration, every
rule instance of $\FRJof{G}$ active in $\DBof{G}$ is applied; the
databases $\Ical$ and $\Pi$ are updated if, as stated by condition at
line~\ref{fun:mainLoopUpdate}, such an application generates a sequent
not subsumed by elements in $\DBof{G}$ (\emph{forward
  subsumption}). We recall that applying rules $\ruleIMPi$,
$\ruleIMPni$ and join rules we assume the
restrictions~\ref{PS1}--\ref{PS4} stated in Sec.~\ref{sec:FRJ}.  The
main loop terminates when either a derivation of $G$ has been built,
and in this case such a derivation is returned, or no new sequent has
been generated in the last iteration, and in this case the database
$\DBof{G}$ is returned.  Note that, since the set of instances of the
rules of $\FRJof{G}$ is finite, $\FSearch{G}$ terminates.

Now, we introduce the key notion of saturated database.  Let $G$ be a
formula and let $\DBof{G}$ be a set of $\FRJof{G}$-sequents:
\begin{enumerate}[label=(DB\arabic*),ref=(DB\arabic*)]
\item\label{DB1} $\DBof{G}$ is a \emph{database for $G$} iff, for
  every $\s\in\DBof{G}$, $\proves{\FRJof{G}} \s$.

\item\label{DB2} A database $\DBof{G}$ for $G$ is \emph{saturated}
  iff, for every $\FRJof{G}$-sequent $\s$ such that
  $\proves{\FRJof{G}} \s$, there exists $\s'\in \DBof{G}$ such that
  $\s\sqsubseteq \s'$.
\end{enumerate}

\noindent
We say that a proof-search procedure is \emph{adequate} iff, for
every formula $G$, the proof-search for $G$ terminates yielding either an
$\FRJof{G}$-derivation of $G$ or a saturated databases for $G$.

\begin{theorem}\label{theo:fsearchAdequate}
  $\FSearch$ is an adequate proof-search procedure.
\end{theorem}

\begin{proof}
  Let $G$ be a formula.  It is immediate to check that, if sequent
  $\s$ is added to $\DBof{G}$ at iteration $n\geq 0$, then $\Pi$
  contains a $\FRJof{G}$-derivation of $\s$ of height $n$; hence
  $\DBof{G}$ is a database for $G$. As a consequence, if $G$ is not
  provable in $\FRJof{G}$, $\FSearch{G}$ terminates returning the set
  $\DBof{G}$. To prove that $\DBof{G}$ is a saturated database, we
  show that, if $\Dcal$ is an $\FRJof{G}$-derivation of $\s$ and
  $\height{\Dcal}=n$, then, at some iteration $k\in\{1,\dots,n + 1\}$
  of the main loop, a sequent $\s'$ such that $\s\sqsubseteq\s'$ is
  added to $\DBof{G}$. If $\height{\Dcal}=0$ then $\s$ is an axiom
  sequent, hence $\s$ is added to $\DBof{G}$ at the first
  iteration. Let us assume that $\height{\Dcal}=h$ ($h>0$) and $\s$ is
  obtained by applying rule $\Rcal$ to premises $\s_1,\dots,\s_m$. By
  induction hypothesis there exist sequents $\s'_1,\dots,\s'_m$ and
  integers $k_1,\dots,k_m$ such that, for every $i\in\{0,\dots,m\}$,
  $k_i\leq h+1$, $\s'_i$ is added to $\DBof{G}$ at iteration $k_i$ and
  $\s_i\sqsubseteq\s'_i$. Let $K=\max\{k_1,\dots,k_n\}$ and let $\s'$
  be the conclusion of the application of rule $\Rcal$ having
  $\s'_1\dots\s'_n$ as premises. By Lemma~\ref{lemma:subsRules},
  $\s\sqsubseteq\s'$. Since $\Rcal$ is active at iteration $K$, we get
  that, at iteration $K+1$, either $\s'$ is added to $\DBof{G}$ or
  $\DBof{G}$ contains a sequent $\s''$ such that
  $\s\sqsubseteq\s'\sqsubseteq\s''$. In both cases we get the
  assertion.
\end{proof}

The proof-search procedure $\FSearch$ can be modified so to maintain
the database of proved sequents concise according with the following
definition:
\begin{enumerate}[label=(DB\arabic*),ref=(DB\arabic*),start=3]  
\item\label{DB3} a database $\DBof{G}$ for $G$ is \emph{compact} iff,
  for every $\s\in\DBof{G}$ there is no $\s'\in\DBof{G}$ s.t.~$\s'\sqsubset \s$.
\end{enumerate}

\noindent
Note that the databases in
Figs.~\ref{fig:frjST},~\ref{fig:frjAST} and~\ref{fig:frjKP} are compact.  
We show that a compact saturated database $\DBMof{G}$   for $G$ is the \emph{minimum}
saturated database for $G$, namely:
for every saturated  database $\DBof{G}$   for $G$, it holds that
$\DBMof{G}\subseteq \DBof{G}$.

\begin{lemma}\label{lemma:compactSatDB}
  Let $G$ be a formula and let $\DBof{G}^1$ be a compact saturated
  database for $G$.  For every saturated database $\DBof{G}^2$ for
  $G$, $\DBof{G}^1\subseteq \DBof{G}^2$.
\end{lemma}

\begin{proof}
  Let $\s_1\in \DBof{G}^1$ and let $\DBof{G}^2$ be a saturated
  database for $G$.  By~\ref{DB1}, we have $\proves{\FRJof{G}}\s_1$
  hence, by~\ref{DB2}, there exists $\s_2\in \DBof{G}^2$ such that
  $\s_1\sqsubseteq \s_2$.  By~\ref{DB1}, we have
  $\proves{\FRJof{G}}\s_2$ hence, by~\ref{DB2}, there exists
  $\s'_1\in \DBof{G}^1$ such that $\s_2\sqsubseteq \s'_1$.  By
  transitivity of $\sqsubseteq$, we get $\s_1\sqsubseteq \s'_1$.
  Since $\DBof{G}^1$ is compact, it is not the case that
  $\s_1\sqsubset \s'_1$ (see~\ref{DB3}), hence $\s_1=\s'_1$, which
  implies $\s_1\sqsubseteq \s_2\sqsubseteq \s_1$.  By antisymmetry of
  $\sqsubseteq$ we get $\s_1=\s_2$, hence $\s_1\in \DBof{G}^2$.  This
  proves $\DBof{G}^1\subseteq \DBof{G}^2$.  
\end{proof}

As an immediate consequence we get:

\begin{theorem}\label{theo:uniquetSatDB}
  For every formula $G$, there exists a unique compact saturated
  database $\DBMof{G}$ for $G$, which is the minimum saturated database for $G$.
\end{theorem}

\begin{proof}
  A compact saturated databases $\DBMof{G}$ for $G$ can be constructed
  by taking any saturated database for $G$ and removing the redundant
  sequents.  By Lemma~\ref{lemma:compactSatDB}, $\DBMof{G}$ is the
  minimum saturated database for $G$ and, by definition, it is
  unique. 
\end{proof}

\noindent 
In the proof-search procedure of Fig.~\ref{fun:FSearch}, we can keep
compact the database of proved sequents $\DBof{G}$ by tweaking the
update instruction at line~\ref{fun:updateBDG} as follows:
\begin{itemize}
\item if $\s$ is added, we delete from $\DBof{G}$ all the sequents
  $\s'$ such that $\s'\sqsubset\s$ and all the sequents $\s''$ such
  that $\s'\mapsto\s''$ (\emph{backward subsumption}).
\end{itemize}

\noindent
Clearly, the database built using this version of $\FSearch$ is
compact and, using Lemma~\ref{lemma:subsRules}, we can prove that the
corresponding version of $\FSearch$ is adequate.

\section{The calculus $\GBUof{G}$ and completeness of $\FRJof{G}$}
\label{sec:gbu}

In this section we present the sequent  calculus $\GBUof{G}$ to derive the
validity of a formula $G$ .  We show that,
 whenever $G$ is not provable in $\FRJof{G}$,
 we can exploit a saturated database  for $G$
to build  a $\GBUof{G}$-derivation of $G$, witnessing that $G$ is
valid.
This implies the completeness of  $\FRJof{G}$.

%
\begin{figure}[t]\small
  \[\small
  \begin{array}{c}
    \begin{minipage}{1.0\linewidth}
      For each sequent $\s$,   $\Lhs{\s}\subseteq\Sfl{G}$ and  $\Rhs{\s}\in\Sfr{G}$.
    \end{minipage}
    \\[4ex]
    \AXC{}
    \RightLabel{$\ruleAXID$}
    \UIC{$\seqgbu{A,\Psi}{A}$}
    \DP
    \hspace{4em}
    \AXC{}
    \RightLabel{$L \bot$}
    \UIC{$\seqgbu{\bot,\Psi }{C}$}
    \DP
    \\[4ex]
    \AXC{$\seqgbu{A,B,\Psi}{C}$}
    \RightLabel{$L \land$}
    \UIC{$\seqgbu{A\land B,\Psi}{C}$}
    \DP
\hspace{4em}
\AXC{$\seqgbu{\Psi}{A}$}
 \AXC{$\seqgbu{\Psi}{B}$}
\RightLabel{$R \land$}
\BIC{$\seqgbu{\Psi}{A\land B}$}
 \DP
\\[4ex]
\AXC{$\seqgbu{A,\Psi}{C}$}
 \AXC{$\seqgbu{B,\Psi}{C}$}
\RightLabel{$L \lor $}
\BIC{$\seqgbu{A\lor B,\Psi}{C}$}
 \DP
\hspace{4em}
\AXC{$\seqgbui{\Psi}{C_k}$}
  \RightLabel{$R \lor_k$}
  \UIC{$\seqgbu{\Psi}{C_1\lor C_2}$}
  \DP
\\[4ex]
\AXC{$\seqgbui{A\imp B,\Psi}{A}$}
\AXC{$\seqgbu{B,\Psi}{C}$}
\RightLabel{$L \imp $}
\BIC{$\seqgbu{A\imp B,\Psi}{C}$}
 \DP
\\[4ex]
\AXC{$\seqgbu{\Psi}{B}$}
  \RightLabel{$\ruleIMPRi$}
  \UIC{$\seqgbu{\Psi}{A\imp B}$}
  \DP
\quad A\in \Clo{\Psi}
\hspace{4em}
\AXC{$\seqgbu{A,\Psi}{B}$}
  \RightLabel{$\ruleIMPRni$}
  \UIC{$\seqgbu{\Psi}{A\imp B}$}
  \DP
\quad A\not\in \Clo{\Psi}
\\[4ex]
  \AXC{} 
     \RightLabel{$\ruleAXID$}
     \UIC{$\seqgbui{A,\Psi }{A}$}
     \DP
 \hspace{2em}
\AXC{$\seqgbui{\Psi}{A}$}
 \AXC{$\seqgbui{\Psi}{B}$}
\RightLabel{$R \land$}
\BIC{$\seqgbui{\Psi}{A\land B}$}
 \DP
\hspace{4em}
\AXC{$\seqgbui{\Psi}{C_k}$}
  \RightLabel{$R \lor_k$}
  \UIC{$\seqgbui{\Psi}{C_1\lor C_2}$}
  \DP
\\[3ex]
\AXC{$\seqgbui{\Psi}{B}$}
  \RightLabel{$\ruleIMPRi$}
  \UIC{$\seqgbui{\Psi}{A\imp B}$}
  \DP
\quad A\in \Clo{\Psi}
\hspace{4em}
\AXC{$\seqgbu{A,\Psi}{B}$}
  \RightLabel{$\ruleIMPRni$}
  \UIC{$\seqgbui{\Psi}{A\imp B}$}
  \DP
\quad A\not\in \Clo{\Psi}
\end{array}
\]
\caption{The calculus $\GBUof{G}$ ($k\in\{1,2\}$).}
\label{fig:GBU}
\end{figure}


Let $G$ be a formula.  The rules of the calculus $\GBUof{G}$ are
presented in Fig.~\ref{fig:GBU}.  In $\GBUof{G}$ we have two kinds of
sequents, where $\Psi\subseteq \Sfl{G}$ and $A\in\Sfr{G}$:
\begin{itemize}
\item \emph{regular sequents} of the form $\seqgbu{\Psi}{A}$;
\item \emph{irregular sequents}  of the form  $\seqgbui{\Psi}{A}$.
\end{itemize}

\noindent
The subscript $g$ on arrows is used to avoid confusion with
$\FRJof{G}$-sequents; for $\tau=\seqgbu{\Psi}{A}$
or $\tau=\seqgbui{\Psi}{A}$, we set $\Lhs{\tau}=\Psi$
and  $\Rhs{\tau}=A$.
The calculus  $\GBUof{G}$ consists of two axiom rules, namely $\ruleAXID$ and $L\bot$,
and left and right rules for each connective.
There are two rules to introduce an implication in the right,
that is $\ruleIMPRi$ and  $\ruleIMPRni$,
depending on the condition $A\in\Clo{\Psi}$.

 A $\GBUof{G}$-sequent $\tau$ is \emph{valid (in
  $\IPL$)} iff the formula $(\bigwedge \Lhs{\tau})\imp \Rhs{\tau}$ is valid
 (as usual, we set $\bigwedge\emptyset =\bot \imp \bot$). 
By $\proves{\GBUof{G}} \tau$ we mean that $\tau$ is provable in
$\GBUof{G}$.
We say that  $G$ is \emph{provable} in  $\GBUof{G}$,
denoted by $\proves{\GBUof{G}} G$,
iff $\proves{\GBUof{G}}\seqgbu{}{G}$.
A $\GBUof{G}$-derivation can be trivially
mapped to a derivation of the calculus $\GJ$ for
$\IPL$~\cite{TroSch:00}.
Basically, one has to erase the distinction between regular and irregular
sequents; in the translation of $\ruleIMPRi$, one has to add the antecedent
$A$ to the left of the premise.
Thus:

\begin{lemma}\label{lemma:GBUsound}
  If  $\proves{\GBUof{G}} \tau$, then  $\tau$ is valid.\qed
\end{lemma}

Accordingly, we get the 
Soundness of $\GBUof{G}$:

\begin{theorem}[Soundness of $\GBUof{G}$]\label{theo:GBUsound}
$\proves{\GBUof{G}} G$ implies  $G\in\IPL$.
\qed
\end{theorem}

Differently from $\GJ$, backward proof-search in $\GBUof{G}$ is
terminating.  To show this, we introduce a \emph{weight} function
$\wggbuname$ on $\GBUof{G}$-sequents such that, after the backward
application of a rule, the weight of sequents decreases.  
Given a  $\GBUof{G}$-sequent $\tau$, the size of $\tau$,
denoted by $\size{\tau}$,
is the number of logical symbols occurring in $\tau$.
Let us consider the instance
\[
\AXC{$\cdots$}
\AXC{$\tau'$}
\AXC{$\cdots$}
 \RightLabel{$\Rcal$}
   \TIC{$\tau$}
 \DP
\qquad
\begin{minipage}{10em}
$\Lhs{\tau'}\;=\;\Psi'$ 
\\[1ex]   
$\Lhs{\tau}\;=\;\Psi$ 
\end{minipage}
\]
of an application of a rule of $\GBUof{G}$,
where $\tau$ is the conclusion and   $\tau'$  any of the premises.
We prove the following properties:

\begin{enumerate}[label=(\arabic*), ref=(\arabic*)]
  
\item\label{propTermGBU:1}
$\card{\,\Sfl{G}\setminus \Clo{\Psi'}\,}\;\leq\;\card{\,\Sfl{G}\setminus \Clo{\Psi}\,}$.

\item\label{propTermGBU:2}
$\Rcal=\ \ruleIMPRni$
implies 
$\card{\,\Sfl{G}\setminus   \Clo{\Psi'}\,}\;<\;\card{\,\Sfl{G}\setminus \Clo{\Psi}\,}$.

\item\label{propTermGBU:3}
 If $\tau'$ is not the leftmost premise of $L \imp$, then
  $\size{\tau'} < \size{\tau}$.

\end{enumerate}

\noindent
Point~\ref{propTermGBU:1} follows by the fact
that $\Clo{\Psi}\subseteq\Clo{\Psi'}$,
which implies $\Sfl{G}\setminus \Clo{\Psi'}\,\subseteq \,\Sfl{G}\setminus \Clo{\Psi}$.
If $\Rcal$ is the rule $\ruleIMPRni$, we have
$\Psi'=\Psi\cup\{A\}$, with $A\not\in\Clo{\Psi}$.  Thus, $A\in
\Sfl{G}\setminus \Clo{\Psi}$ and, since $A\in\Clo{\Psi'}$, we have
$A\not\in \Sfl{G}\setminus \Clo{\Psi'}$. This proves Point~\ref{propTermGBU:2}.
Point~\ref{propTermGBU:3} can be easily checked.
By  points~\ref{propTermGBU:1}--\ref{propTermGBU:3},
we can define $\wggbu{\tau}$ as the triple of
non-negative integers:
\[
\begin{array}{c}
\wggbu{\tau}\:=\:\stru{\;\card{\,\Sfl{G}\,\setminus\,\Clo{\Psi\,}\,}\,,\,\tp{\tau}\,,\,\size{\tau}\;    }
\qquad
\tp{\tau} \;= \;
\begin{cases}
1\;\mbox{if $\tau$ is  regular }
\\
0\;\mbox{otherwise.}
\end{cases}
\\
\Psi\;=\; \Lhs{\tau}\
\end{array}
\]

We get (where $\prec$ is lexicographic order on triples
of integers):
\begin{lemma}\label{lemma:wggbu}
  Let $\Rcal$ be a rule of $\GBUof{G}$, let $\tau$ be the conclusion
  of $\Rcal$ and $\tau'$ any of the premises of $\Rcal$.  Then
  $\stru{0,0,0}\preceq \wggbu{\tau'} \prec \wggbu{\tau}$.  \qed
\end{lemma}

\noindent
Note that the component $\tp{\tau}$ of $\wggbu{\tau}$ accommodates the
case where $\tau$ is the conclusion of rule $L \imp$ and $\tau'$ the
leftmost premise; in this case $\tau$ is regular and $\tau'$ is irregular,
hence $\tp{\tau'}<\tp{\tau}$. As a consequence of Lemma~\ref{lemma:wggbu},
and reasoning as in Lemma~\ref{lemma:branch}~\ref{lemma:branch:1},
we get a bound on the height of $\GBUof{G}$-trees:

\begin{theorem}\label{theo:gbufin}
Let $\Tcal$ be  a $\GBUof{G}$-tree
and $\tau$ the root sequent of $\Tcal$.
Then, the height of $\Tcal$ is $O(|\tau|^2)$.
  \qed
\end{theorem}

Accordingly, given a $\GBUof{G}$-sequent $\tau$, the number of
$\GBUof{G}$-trees having root sequent $\tau$ is finite, hence backward
proof-search in $\GBUof{G}$ always terminates.  In contrast, we cannot
set a bound on the depth of $\GJ$-trees, since in $\GJ$ we can apply
$\imp L$ an unbounded number of times.

We present the procedure $\Search$ (\emph{Backward Search}) to search
for a $\GBUof{G}$-derivation of a goal formula $G$, namely a
$\GBUof{G}$-derivation of $\seqgbu{}{G}$.  $\Search$ resembles a
standard backward proof-search procedure: starting from the sequent
$\seqgbu{}{G}$, rules of $\GBUof{G}$ are backward applied, giving
precedence to the invertible rules.  We recall that invertible rules
can be applied in any order without affecting completeness, while the
application of a non-invertible rule introduces a backtrack point: if
proof-search fails, one has to backtrack and try another way.  In the
design of $\Search$, we avoid backtracking by exploiting a saturated
database $\DBof{G}$ for $G$.  Let:
\[
\bGat\;=\;\Sfl{G}\cap\PV
\qquad
\bGimp\;=\;\Sfl{G}\cap \Fmimp
\qquad
\bG\;=\;\bGat\cup\bGimp
\]
A $\GBUof{G}$-sequent $\tau$ is \emph{critical} iff one of the following
conditions holds: 
\begin{itemize}
\item $\tau=\seqgbu{\Omega}{D}$ and $\Omega\subseteq\bG$ and either
  $D\in\Prime$ or $D=C_1\lor C_2$;
\item $\tau=\seqgbui{\Omega}{C_1\lor C_2}$ and $\Omega\subseteq\bG$.
\end{itemize}

\noindent
In backward proof-search, to a non-critical sequent
we can always apply one of the rules 
$L\land$, $R\land $, $L\lor$, $\ruleIMPRi$, $\ruleIMPRni$,
which are the invertible rules of $\GBUof{G}$.
On the other hand,
 to a critical sequent
 we can only apply rule $L\imp$ or $R\lor_k$,
 which are not invertible.
In this case, to choose
 the right route and avoid backtracking, we  query the  database $\DBof{G}$.
To extract from
$\DBof{G}$ the relevant information, we introduce the following \emph{evaluation
  relation} $\eval$:

\begin{itemize}
\item $\DBof{G}\eval\seqgbu{\Psi}{C}$ iff there exists
  $\seqfrj{\G}{C}\in\DBof{G}$ such that $\Psi\subseteq\Clo{\G}$;

\item $\DBof{G}\eval\seqgbui{\Omega}{C}$ iff there exists
  $\seqfrji{\Sigma}{\Theta}{C}\in\DBof{G}$ such that
  $\Sigma\,\subseteq\, \Omega\,\subseteq\,\Sigma\cup \Theta$.
\end{itemize}

\noindent
By $\DBof{G}\neval\s$, we mean that $\DBof{G}\eval\s$ does not hold.

The recursive function \Search   
satisfies the following specification:

\begin{itemize}

\item
  Let $\tau$ be a $\GBUof{G}$-sequent and $\DBof{G}$ a
  saturated database for $G$ such that:
  
  \begin{enumerate}[label=(BSr\arabic*),ref=(BSr\arabic*)]
  \item \label{search:ass1}
    $\DBof{G}\neval\tau$;
    
  \item \label{search:ass2}
if $\tau =\seqgbui{\Omega}{C} $, then $\Omega\subseteq\bG$.
  \end{enumerate}
  
  Then, $\Search(\s,\DBof{G})$
  builds a $\GBUof{G}$-derivation $\Tcal$ of $\tau$.
\end{itemize}

\noindent
The $\GBUof{G}$-derivation $\Tcal$ built by $\Search(\tau,\DBof{G})$
is defined by cases on $\tau$.

\begin{enumerate}[label=(B\arabic*),ref=(B\arabic*)]
\item\label{search:1}
  If $\tau$ is a $\GBUof{G}$-axiom.
  
  \smallskip
  Then, $\Tcal$ is the  $\GBUof{G}$-derivation
  only containing $\tau$.

\item\label{search:2}
  If $\tau$ is a non-critical sequent.

  \smallskip
Let
\[
\AXC{$\tau_1$}
\AXC{$\cdots$}
\AXC{$\tau_n$}
 \RightLabel{$\Rcal$}
   \TIC{$\tau$}
   \DP
\qquad 
\begin{minipage}{15em}
$n\in\{1,2\}$
\\[1ex]
$\Rcal\in\{\, 
L\land,\,R\land,\, L\lor,\, \ruleIMPRi,\,\ruleIMPRni
\,\}$
\end{minipage}
\]
be any instance of  a rule of $\GBUof{G}$  having conclusion $\tau$.

For every $1\leq j\leq n$,
  let $\Tcal_j = \Search(\tau_j,\DBof{G})$.
  Then, $\Tcal$ is the  $\GBUof{G}$-derivation
  \[
 \AXC{$\Tcal_1$}
 \noLine
\UIC{$\tau_1$}
\AXC{\dots}
\AXC{$\Tcal_n$}
 \noLine
\UIC{$\tau_n$}
 \RightLabel{$\Rcal$}
   \TIC{$\tau$}
   \DP
\]

\item\label{search:3}
  If $\tau\,=\,\seqgbui{\Omega}{C_1\lor C_2}$   (by~\ref{search:ass2}, $\Omega \subseteq\bG$).

\smallskip
 Choose any $C_k\in\{C_1,C_2\}$ such that $\DBof{G}\neval \seqgbui{\Omega}{C_k}$.

 Let $\Tcal_k\,=\,\Search(\,\seqgbui{\Omega}{C_k}\,,\,\DBof{G}\,)$.
Then, $\Tcal$ is  the  $\GBUof{G}$-derivation
 \[
 \AXC{$\Tcal_k$}
 \noLine
\UIC{$\seqgbui{\Omega}{C_k}$}
 \RightLabel{$R \lor_k $}
   \UIC{$\seqgbui{\Omega}{C_1\lor C_2}$}
   \DP
   \]

 \item\label{search:4} If $\tau\,=\,\seqgbu{\Omega}{F}$, with
   $F\in\Prime$ (note that here $\Omega\subseteq\bG$).

  \smallskip
  Choose any $A\imp B\in\Omega$ such that $\DBof{G}\neval \seqgbui{\Omega}{A}$.

  Let $\Omega_B \,=\, (\Omega\setminus\{A\imp B\})\cup\{B\}$ and
  \[
\Tcal_1\;=\;\Search(\,\seqgbui{\Omega}{A}\,,\,\DBof{G}\,)
\qquad
\Tcal_2\;=\;\Search(\,\seqgbu{\Omega_B}{F}\,,\,\DBof{G}\,)
 \]   
 Then, $\Tcal$ is the  $\GBUof{G}$-derivation
  \[
 \AXC{$\Tcal_1$}
 \noLine
\UIC{$\seqgbui{\Omega}{A}$}
\AXC{$\Tcal_2$}
 \noLine
\UIC{$\seqgbu{\Omega_B}{F}$}
 \RightLabel{$L \imp $}
   \BIC{$\seqgbu{\Omega}{F}$}
   \DP
 \]
 
\item\label{search:5} Otherwise $\tau\,=\,\seqgbu{\Omega}{C_1\lor C_2}$
  with $\Omega\subseteq\bG$.
 
  \smallskip
  Let $\Upsilon =\{\,A~|~A\imp B\in\Omega\,\}$.

  Choose any $Z\,\in\,\Upsilon\cup\{\,C_1,C_2\,\}$ 
 such that $\DBof{G}\neval \seqgbui{\Omega}{Z}$. 

The derivation $\Tcal$ is defined by cases on $Z$.

 \begin{itemize}
 \item $Z\in\Upsilon$.

   \smallskip
 Let $A\imp B\in \Omega$ such that
 $Z=A$, let   $\Omega_B\;=\;(\Omega\setminus\{A\imp B\})\cup\{B\}$ and
\[
\Tcal_1\;=\;\Search(\,\seqgbui{\Omega}{A}\,,\,\DBof{G}\,)
\quad
\Tcal_2\;=\;\Search(\,\seqgbu{\Omega_B}{C_1\lor C_2}\,,\,\DBof{G}\,)
 \] 
   Then, $\Tcal$ is the  $\GBUof{G}$-derivation
  \[
 \AXC{$\Tcal_1$}
 \noLine
\UIC{$\seqgbui{\Omega}{A}$}
\AXC{$\Tcal_2$}
 \noLine
\UIC{$\seqgbu{\Omega_B}{C_1\lor C_2}$}
 \RightLabel{$L \imp $}
   \BIC{$\seqgbu{\Omega}{C_1\lor C_2}$}
   \DP
  \]
 
\item $Z=C_k$ ($k\in\{1,2\}$).

\smallskip
Let $\Tcal_k\,=\,\Search(\,\seqgbui{\Omega}{C_k}\,,\,\DBof{G}\,)$.
Then, $\Tcal$ is the  $\GBUof{G}$-derivation
 \[
 \AXC{$\Tcal_k$}
 \noLine
\UIC{$\seqgbui{\Omega}{C_k}$}
 \RightLabel{$R \lor_k $}
   \UIC{$\seqgbu{\Omega}{C_1\lor C_2}$}
   \DP
  \]
\end{itemize}
\end{enumerate}

\noindent
An example of computation is discussed below.  The correctness of
$\Search$ is proved in Theorem~\ref{theo:search}.  The key issue is to
show that, if cases~\ref{search:1} and~\ref{search:2} do not hold,
then at least one of the statements described
in~\ref{search:3}--~\ref{search:5} can be performed.  We point out
that in \Search there are no backtrack points.
We also remark that only the irregular sequents of  $\DBof{G}$
are used by $\Search$.

\begin{example}\label{ex:gbu1}
  Let  $E$ be the following goal formula:
  \[
  \begin{array}{l}
    E\;=\; (p\land A \land B \land C) \,\imp\, D
    \\[1ex]
    A\;=\; p\imp (q_1\lor q_2 ) \quad
    B\;=\; q_1\imp D \quad
    C\;=\; q_2\imp D \quad
    D\;=\; r_1\lor r_2
    \\[1ex]
    \begin{array}{rcl}
      \Lhs{E}\,\cap\,(\PV\cup\Fmimp) &\;=\;&
      \{\,p,\, q_1,\, q_2,\, r_1,\, r_2,\,A,\,B,\,C\,\}
      \\[1ex]
      \Rhs{E} &=&
      \{\,p,\, q_1,\, q_2,\, r_1,\, r_2,\,D,\,E\,\}
    \end{array}
  \end{array}
  \]

  \noindent
Since $E$ is valid  and the proof-search procedure \FSearch of Fig.~\ref{fun:FSearch}
is adequate (see Theorem~\ref{theo:fsearchAdequate}),
the call $\FSearch{E}$   
returns a saturated database for $E$. 
We can exploit it to run the procedure \Search
and build a $\GBUof{E}$-derivation of $E$.
We consider the compact saturated 
database  $\DBMof{E}$ for $E$ containing 
the following irregular sequents
$\s_{\eqref{exE:1}},\dots,\s_{\eqref{exE:6}}$
(we omit the regular sequents since they are not used by $\Search$):

  \leqnomode 
  \setcounter{equation}{0} 
  \begin{align}
    \label{exE:1}
    \tseqfrji{}{q_1,\, q_2,\, r_1,\, r_2,\,A,\, B,\, C}{p}
    &&  \ruleAXI
    \\
    \label{exE:2}
    \tseqfrji{}{p,\, q_2,\, r_1,\, r_2,\,A,\, B,\, C}{q_1}
    &&  \ruleAXI
    \\
    \label{exE:3}
    \tseqfrji{}{p,\, q_1,\, r_1,\, r_2,\,A,\, B,\, C}{q_2}
    &&  \ruleAXI
    \\
    \label{exE:4}
    \tseqfrji{}{p,\, q_1,\, q_2,\, r_2,\,A,\, B,\, C}{r_1}
    &&  \ruleAXI
    \\
    \label{exE:5}
    \tseqfrji{}{p,\, q_1,\, q_2,\, r_1,\,A,\, B,\, C}{r_2}
    &&  \ruleAXI
    \\  
    \label{exE:6}
    \tseqfrji{}{p,\, q_1,\, q_2,\, A,\, B,\, C}{D}
    && \lor\;\eqref{exE:4}   \eqref{exE:5}         
  \end{align}
  
  \noindent 
    Proof-search starts from the   sequent
    $\seqgbu{}{E}$.
    Rules of $\GBUof{E}$ are backward applied
according with~\ref{search:2},
  until the critical sequent $\tau_1$ is obtained:
  \[
  \AXC{$\tau_1\;=\;\seqgbu{\Psi_1}{D}$}
  \doubleLine
  \RightLabel{$L\land$ (three times)}
  \UIC{$\seqgbu{p\land A\land B\land C}{D}$}
  \RightLabel{$\ruleIMPRni$}
  \UIC{$\seqgbu{}{E}$}
  \DP
  \qquad \Psi_1\;=\; p,\,A,\,B,\,C
  \]
  To continue the construction of the derivation, we have to bottom-up
  apply a rule of $\GBUof{E}$ to the critical sequent $\tau_1$, and we
  query $\DBMof{E}$ to choose the right way.  We have 5 possible
  choices:
  \begin{enumerate}[label=(c\arabic*),ref=(c\arabic*)]
  \item\label{exE:choice1} Apply rule $L\imp$ with main formula $A$;
    by~\ref{search:5}, this requires $\DBMof{E}\neval\seqgbui{\Psi_1}{p}$.
    
  \item\label{exE:choice2} Apply rule $L\imp$ with main formula $B$;
    by~\ref{search:5}, this requires
    $\DBMof{E}\neval\seqgbui{\Psi_1}{q_1}$.

  \item\label{exE:choice3} Apply rule $L\imp$ with main formula $C$;
    by~\ref{search:5}, this requires
    $\DBMof{E}\neval\seqgbui{\Psi_1}{q_2}$.

  \item\label{exE:choice4} Apply rule $R\lor_1$; 
by~\ref{search:5},  this requires $\DBMof{E}\neval\seqgbui{\Psi_1}{r_1}$.
  
  \item\label{exE:choice5} Apply rule $R\lor_2$; 
by~\ref{search:5}, this requires $\DBMof{E}\neval\seqgbui{\Psi_1}{r_2}$.
  \end{enumerate}
  
  \noindent
  We have:
  \[
  \begin{array}{l}
    \DBMof{E}\eval\seqgbui{\Psi_1}{q_1}\quad\mbox{(see $\s_{\eqref{exE:2}}$)}
    \qquad
    \DBMof{E}\eval\seqgbui{\Psi_1}{q_2}\quad\mbox{(see $\s_{\eqref{exE:3}}$)}
    \\[1ex]
    \DBMof{E}\eval\seqgbui{\Psi_1}{r_1}\quad\mbox{(see $\s_{\eqref{exE:4}}$)}
    \qquad
    \DBMof{E}\eval\seqgbui{\Psi_1}{r_2}\quad\mbox{(see $\s_{\eqref{exE:5}}$)}
    \\[1ex]
    \DBMof{E}\neval\seqgbui{\Psi_1}{p}
  \end{array}
  \]
  hence only choice~\ref{exE:choice1} can be performed.  Note that,
  selecting any of the other
  options~\ref{exE:choice2}--\ref{exE:choice5}, where the
  corresponding condition is not matched, proof-search fails.
  Complying with~\ref{exE:choice1}, we apply rule $L\imp$ with main
  formula $A$ and we continue until we get the critical sequents
  $\tau_2$ and $\tau_3$.
  
  \[
  \AXC{}
  \RightLabel{$\ruleAXID$}
  \UIC{$\seqgbu{p,\,A,\,B,\,C}{p}$}
  \AXC{$\tau_2\;=\;\seqgbu{\Psi_2}{D}$}
  \AXC{$\tau_3\;=\;\seqgbu{\Psi_3}{D}$}
  \RightLabel{$L\lor$}
  \BIC{$\seqgbu{p,\,\underline{\,q_1\lor q_2},\, B,\,C}{D}$}
  \RightLabel{$L\imp$}
  \BIC{$\tau_1\;=\;\seqgbu{p,\,\underline{A},\,B,\,C}{D}$}
  \DP
  \qquad
  \begin{minipage}{10em}
    $\Psi_2\,=\; p,\,q_1,\,B,\,C$
    \\[1ex] 
    $\Psi_3\,=\; p,\,q_2,\,B,\,C$
  \end{minipage}
  \]
  Let us consider the critical sequent $\tau_2$. 
  We have 4 possible choices:
  
  \begin{enumerate}[label=(c\arabic*),ref=(c\arabic*),start=6]
  \item\label{exE:choice6} Apply rule $L\imp$ with main formula $B$;
    by~\ref{search:5}, this requires
    $\DBMof{E}\neval\seqgbui{\Psi_2}{q_1}$.

  \item\label{exE:choice7} Apply rule $L\imp$ with main formula $C$;
    by~\ref{search:5}, this requires
    $\DBMof{E}\neval\seqgbui{\Psi_2}{q_2}$.

  \item\label{exE:choice8} Apply rule $R\lor_1$; 
by~\ref{search:5}, this requires $\DBMof{E}\neval\seqgbui{\Psi_2}{r_1}$.

  \item\label{exE:choice9} Apply rule $R\lor_2$; 
by~\ref{search:5}, this requires $\DBMof{E}\neval\seqgbui{\Psi_2}{r_2}$.
  \end{enumerate}

  \noindent
  We note that:
  \[
  \begin{array}{l}
    \DBMof{E}\eval\seqgbui{\Psi_2}{q_2}\quad\mbox{(see $\s_{\eqref{exE:3}}$)}
    \qquad
    \DBMof{E}\eval\seqgbui{\Psi_2}{r_1}\quad\mbox{(see $\s_{\eqref{exE:4}}$)}
    \\[1ex]
    \DBMof{E}\eval\seqgbui{\Psi_2}{r_2}\quad\mbox{(see $\s_{\eqref{exE:5}}$)}
    \qquad
    \DBMof{E}\neval\seqgbui{\Psi_2}{q_1}
  \end{array}
  \]
  Thus, we choose~\ref{exE:choice6} and we
  continue until we get the critical sequents $\tau_4$ and $\tau_5$.  
  \[
  \AXC{}
  \RightLabel{$\ruleAXID$}
  \UIC{$\seqgbu{p,\,q_1,\,B,\,C}{q_1}$}
  \AXC{$\tau_4\;=\;\seqgbu{\Psi_4}{D}$}
  \AXC{$\tau_5\;=\;\seqgbu{\Psi_5}{D}$}
  \RightLabel{$L\lor$}
  \BIC{$\seqgbu{p,\,q_1,\, \underline{D},\,C}{D}$}
  \RightLabel{$L\imp$}
  \BIC{$\tau_2\;=\;\seqgbu{p,\,q_1,\,\underline{B},\,C}{D}$}
  \DP
  \qquad
  \begin{minipage}{10em}
    $\Psi_4\,=\; p,\,q_1,\,r_1,\,C$
    \\[1ex] 
    $\Psi_5\,=\; p,\,q_1,\,r_2,\,C$
  \end{minipage}
  \]
For the critical sequent $\tau_4$,  we have 3 possible options:

  \begin{enumerate}[label=(c\arabic*),ref=(c\arabic*),start=10]
  \item\label{exE:choice10} Apply rule $L\imp$ with main formula $C$;
    by~\ref{search:5}, this requires
    $\DBMof{E}\neval\seqgbui{\Psi_4}{q_2}$.

  \item\label{exE:choice11} Apply rule $R\lor_1$; by~\ref{search:5},
    this requires $\DBMof{E}\neval\seqgbui{\Psi_4}{r_1}$.

  \item\label{exE:choice12} Apply rule $R\lor_2$; by~\ref{search:5},
    this requires $\DBMof{E}\neval\seqgbui{\Psi_4}{r_2}$.
  \end{enumerate}

  \noindent
  We have
  \[
  \begin{array}{l}
    \DBMof{E}\eval\seqgbui{\Psi_4}{q_2}\quad\mbox{(see $\s_{\eqref{exE:3}}$)}
    \qquad
    \DBMof{E}\eval\seqgbui{\Psi_4}{r_2}\quad\mbox{(see $\s_{\eqref{exE:5}}$)}
    \qquad
    \DBMof{E}\neval\seqgbui{\Psi_4}{r_1}
  \end{array}
  \]
  Thus, we   select~\ref{exE:choice11} and
  we get
  \[
  \AXC{}
  \RightLabel{$\ruleAXID$}
  \UIC{$\seqgbu{p,\,q_1,\,r_1,\,C}{r_1}$}
  \RightLabel{$R\lor_1$}
  \UIC{$\tau_4\;=\;\seqgbu{p,\,q_1,\,r_1,\,C }{D}$}
  \DP
  \]
  The $\GBUof{E}$-derivations of the sequents $\tau_5$ and $\tau_3$
  have a similar construction:
  \[
  \begin{array}{c}
    \AXC{}
    \RightLabel{$\ruleAXID$}
    \UIC{$\seqgbu{p,\,q_1,\,r_2,\,C}{r_2}$}
    \RightLabel{$R\lor_2$}
    \UIC{$\tau_5\;=\;\seqgbu{p,\,q_1,\,r_2,\,C }{D}$}
    \DP
    \\[6ex]
    \AXC{}
    \RightLabel{$\ruleAXID$}
    \UIC{$\seqgbu{p,\,q_2,\,B,\,C}{q_2}$}
    \AXC{}
    \RightLabel{$\ruleAXID$}
    \UIC{$\seqgbu{p,\,q_2,\,B,\,r_1}{r_1}$}
    \RightLabel{$R\lor_1$}
    \UIC{$\seqgbu{p,\,q_2,\,B,\,r_1}{D}$}
    \AXC{}
    \RightLabel{$\ruleAXID$}
    \UIC{$\seqgbu{p,\,q_2,\,B,\,r_2}{r_2}$}
    \RightLabel{$R\lor_2$}
    \UIC{$\seqgbu{p,\,q_2,\,B,\,r_2 }{D}$}
    \RightLabel{$L\lor$}
    \BIC{$\seqgbu{p,\,q_2,\, B,\,\underline{D}}{D}$}
    \RightLabel{$L\imp$}
    \BIC{$\tau_3\;=\;\seqgbu{p,\,q_2,\,B,\,\underline{C}}{D}$}
    \DP
  \end{array}
  \]
and this completes the definition of the $\GBUof{E}$-derivation of $E$.
  \EndEs
\end{example}


Now we prove the correctness of \Search. We start by showing some
properties of the evaluation relation.

\begin{lemma}\label{lemma:gbuInv}
  Let $G$ be a formula, let $\DBof{G}$ be a saturated database for $G$,
  $\Psi\subseteq \Sfl{G}$ and $\Omega\subseteq \bG$.
  \begin{enumerate}[label=(\roman*),ref=(\roman*)]
    
  \item\label{lemma:gbuInv:1} If $\DBof{G}\eval\seqgbu{A,B,\Psi}{C}$,
    then $\DBof{G}\eval\seqgbu{A\land B,\Psi}{C}$.
    
  \item\label{lemma:gbuInv:2} If $C_1\land C_2\in\Sfr{G}$
    and $\DBof{G}\eval\seqgbu{\Psi}{C_k}$ with $k\in\{1,2\}$ then 
    $\DBof{G}\eval\seqgbu{\Psi}{C_1\land C_2}$.
    
  \item\label{lemma:gbuInv:3} If $A_1\lor A_2\in\Sfl{G}$ and
    $\DBof{G}\eval\seqgbu{A_k,\Psi}{C}$ with $k\in\{1,2\}$, then
    $\DBof{G}\eval\seqgbu{A_1\lor A_2,\Psi}{C}$.
    
  \item\label{lemma:gbuInv:4} If $A\imp B\in\Sfl{G}$ and
    $\DBof{G}\eval\seqgbu{B,\Psi}{C}$, then $\DBof{G}\eval\seqgbu{A\imp
      B,\Psi}{C}$.
    
  \item\label{lemma:gbuInv:5} If $A\imp B\in\Sfr{G}$ 
and $\DBof{G}\eval\seqgbu{\Psi}{B}$, 
 with $A\in\Clo{\Psi}$,
then $\DBof{G}\eval\seqgbu{\Psi}{A\imp B}$.
    
  \item\label{lemma:gbuInv:6} If $A\imp B\in\Sfr{G}$ and
    $\DBof{G}\eval\seqgbu{A,\Psi}{B}$ with $A\not\in\Clo{\Psi}$, then
    $\DBof{G}\eval\seqgbu{\Psi}{A\imp B}$.
    
  \item\label{lemma:gbuInv:7} If $C_1\land C_2\in\Sfr{G}$ and
    $\DBof{G}\eval\seqgbui{\Omega}{C_k}$ with $k\in\{1,2\}$, then
    $\DBof{G}\eval\seqgbui{\Omega}{C_1\land C_2}$.

  \item\label{lemma:gbuInv:8} If $A\imp B\in\Sfr{G}$ and
    $\DBof{G}\eval\seqgbui{\Omega}{B}$ with $A\in\Clo{\Omega}$, then
    $\DBof{G}\eval\seqgbui{\Omega}{A\imp B}$.
    
  \item\label{lemma:gbuInv:9} If $\DBof{G}\eval\seqgbu{A,\Omega}{B}$ with
    $A\not\in\Clo{\Omega}$, then $\DBof{G}\eval\seqgbui{\Omega}{A\imp
      B}$.
  \end{enumerate}
\end{lemma}

\begin{proof}
  We only detail some representative cases.

\paragraph{Proof of~\ref{lemma:gbuInv:1}}
Let $\DBof{G}\eval\seqgbu{A,B,\Psi}{C}$. Then, there exists $\G$  such that:
\begin{itemize}
\item  $\seqfrj{\G}{C}\in\DBof{G}$;
  
\item $\Psi\cup \{A,B\}\, \subseteq\, \Clo{\G}$.

\end{itemize}
Since $A\land B\in\Clo{\G}$, we have
$\Psi\cup \{A\land B\} \subseteq \Clo{\G}$, hence
$\DBof{G}\eval\seqgbu{A\land B,\Psi}{C}$.

\paragraph{Proof of~\ref{lemma:gbuInv:2}}
Let $C_1\land C_2\in\Sfr{G}$ and let us assume  $\DBof{G}\eval\seqgbu{\Psi}{C_k}$,
with $k\in\{1,2\}$. Then, there exists $\G$  such that:
\begin{itemize}
\item  $\seqfrj{\G}{C_k}\in\DBof{G}$;

\item $\Psi\, \subseteq\, \Clo{\G}$.

\end{itemize}
By~\ref{DB1}, $\proves{\FRJof{G}}\seqfrj{\G}{C_k}$, hence
$\proves{\FRJof{G}}\seqfrj{\G}{C_1\land C_2}$.  By~\ref{DB2},
$\DBof{G}$ contains a sequent $\seqfrj{\G'}{C_1\land C_2}$ such that
$\G\subseteq \G'$.  By~\ref{propClo:4}, $\Clo{\G}\subseteq\Clo{\G'}$,
which implies $\Psi \subseteq \Clo{\G'}$, hence
$\DBof{G}\eval\seqgbu{\Psi}{C_1\land C_2}$.

\paragraph{Proof of~\ref{lemma:gbuInv:5}}
Let  $A\imp B\in\Sfr{G}$  and let us assume
$\DBof{G}\eval\seqgbu{\Psi}{B}$, with $A\in\Clo{\Psi}$.
Then,  there exists $\G$  such that:
\begin{itemize}
\item  $\seqfrj{\G}{B}\in\DBof{G}$;

\item $\Psi\, \subseteq\, \Clo{\G}$.
\end{itemize}
By~\ref{DB1},  $\proves{\FRJof{G}}{\seqfrj{\G}{B}}$.
By~\ref{propClo:6} we get $\Clo{\Psi}\subseteq\Clo{\G}$, hence
$A\in\Clo{\G}$; this implies  
$\proves{\FRJof{G}}{\seqfrj{\G}{A\imp B}}$.  
Reasoning as in the proof of case~\ref{lemma:gbuInv:2},
we get $\DBof{G}\eval\seqgbu{\Psi}{A\imp B}$.


\paragraph{Proof of~\ref{lemma:gbuInv:6}}
Let $A\imp B\in\Sfr{G}$ and let us assume $\DBof{G}\eval\seqgbu{A,\Psi}{B}$, with
$A\not\in\Clo{\Psi}$. Then, there exists $\G$  such that:
\begin{itemize}
\item  $\seqfrj{\G}{B}\in\DBof{G}$;

\item $\Psi\cup\{A\}\, \subseteq\, \Clo{\G}$.
\end{itemize}
By~\ref{DB1}, $\proves{\FRJof{G}}\seqfrj{\G}{B}$, hence
$\proves{\FRJof{G}}\seqfrj{\G}{A\imp B}$.  
Reasoning as in the proof of case~\ref{lemma:gbuInv:2},
we get
$\DBof{G}\eval\seqgbu{\Psi}{A\imp   B}$.

\paragraph{Proof of~\ref{lemma:gbuInv:8}}
Let us assume $\DBof{G}\eval\seqgbui{\Omega}{B}$ and $A\in\Clo{\Omega}$.
Then, there exist 
$\Sigma$ and $\Theta$  such that:

\begin{itemize}
\item  $\seqfrji{\Sigma}{\Theta}{B}\in\DBof{G}$;
\item $\Sigma\,\subseteq\, \Omega \,\subseteq\,\Sigma\cup\Theta$.
\end{itemize}

\noindent
Since $\Omega\subseteq\Sigma\cup\Theta$, by~\ref{propClo:4} we get
$\Clo{\Omega}\subseteq\Clo{\Sigma\cup\Theta}$, hence
$A\in\Clo{\Sigma\cup\Theta}$.  Let $\Lambda$ be a minimum (possibly
empty) subset of $\Theta$ such that $A\in\Clo{\Sigma\cup \Lambda}$.
Since $A\in\Clo{\Omega}$ and $\Sigma\subseteq \Omega\subseteq
\Sigma\cup\Theta$, we can choose $\Lambda$ so that
$\Sigma\cup\Lambda\subseteq \Omega$.  By~\ref{DB1} $\proves{\FRJof{G}}
\seqfrji{\Sigma}{\Theta}{B}$, hence we can build the
$\FRJof{G}$-derivation:
\[
\AXC{$\vdots$}
\noLine
\UIC{$\seqfrji{\Sigma}{\;\overbrace{\Lambda,\,\Theta_1}^{\Theta}\;}{B}$}
\RightLabel{$\ruleIMPi$}
\UIC{$\seqfrji{\Sigma,\Lambda}{\Theta_1}{A\imp B}$}
\DP
\qquad
\begin{minipage}{15em}
$\Theta_1\,=\, \Theta\setminus\Lambda$
\\[.5ex]
$\Lambda\subseteq \Theta$ and $A\in\Clo{\Sigma\cup\Lambda}$
\\[.5ex]
$\Lambda'\subsetneq \Lambda$ implies  $A\not\in\Clo{\Sigma\cup\Lambda'}$
\end{minipage}
\]
Note that the above derivation matches~\ref{PS1}.
 By~\ref{DB2},
$\DBof{G}$ contains a sequent 
$\seqfrji{\Sigma,\Lambda}{\Theta'}{A\imp B}$ such that
$\Theta_1\subseteq \Theta'$.
Since $\Omega\subseteq \Sigma\cup\Theta$ and
$\Sigma\cup\Theta=\Sigma\cup\Lambda\cup\Theta_1$,
we get  $\Omega\subseteq \Sigma\cup\Lambda\cup \Theta'$.
Thus
$\Sigma\cup\Lambda\,\subseteq\, \Omega\,\subseteq\, \Sigma\cup\Lambda\cup \Theta'$,
and this proves that
$\DBof{G}\eval\seqgbui{\Omega}{A\imp B}$.

\paragraph{Proof of~\ref{lemma:gbuInv:9}}
Let us assume $\DBof{G}\eval\seqgbu{A,\Omega}{B}$  and $A\not\in\Clo{\Omega}$.
Then, there exists
$\Gamma$  such that: 

\begin{itemize}
\item  $\seqfrj{\Gamma}{B}\in\DBof{G}$;

\item $\Omega\cup\{A\}\,\subseteq\,\Clo{\Gamma}$.
  
  \end{itemize}

  \noindent
  Let $\Theta$ be a maximal subset of $\Clo{\G}\cap\bG$ such that
  $A\not\in\Clo{\Theta}$; since $\Omega\subseteq \Clo{\G}\cap\bG$ and
  $A\not\in\Clo{\Omega}$, we can choose $\Theta$ so that
  $\Omega\subseteq \Theta$.  By~\ref{DB1},
  $\proves{\FRJof{G}} \seqfrj{\Gamma}{B}$, hence we can build the
  following $\FRJof{G}$-derivation:
\[
\AXC{$\vdots$}
\noLine
\UIC{$\seqfrj{\G}{B}$}
\RightLabel{$\ruleIMPni$}
\UIC{$\seqfrji{}{\Theta}{A\imp B}$}
\DP
\qquad
\begin{minipage}{20em}
$\Omega\,\subseteq\,\Theta\,\subseteq\, \Clo{\Gamma}\cap\bG$
\\[.5ex]
$A\in\,\Clo{\G}\setminus\Clo{\Theta}$
\\[.5ex]
$\Theta\subsetneq\Theta'\subseteq \Clo{\Gamma}\cap\bG$ implies
  $A\in\Clo{\Theta'}$
\end{minipage}
\]
Note that the above derivation matches~\ref{PS2}.
By~\ref{DB2},
$\DBof{G}$ contains a sequent 
$\seqfrji{}{\Theta'}{A\imp B}$, where $\Theta\subseteq \Theta'$;
since $\Omega\subseteq \Theta'$,
we get $\DBof{G}\eval\seqgbui{\Omega}{A\imp B}$. 
\end{proof}


\begin{lemma}\label{lemma:gbuiOr}
  Let $G$ be a formula, let $\DBof{G}$ be a saturated database for $G$
  and $\Omega \subseteq \bG$.  If $\DBof{G}\eval\seqgbui{\Omega}{C_1}$
  and $\DBof{G}\eval\seqgbui{\Omega}{C_2}$, then
  $\DBof{G}\eval\seqgbui{\Omega}{C_1\lor C_2}$.
\end{lemma}  

\begin{proof}
Let us assume  $\DBof{G}\eval\seqgbui{\Omega}{C_1}$ and
$\DBof{G}\eval\seqgbui{\Omega}{C_2}$. 
For $k\in\{1,2\}$,  there exist
$\Sigma_k$ and $\Theta_k$ be such that:

\begin{itemize}
\item  $\seqfrji{\Sigma_k}{\Theta_k}{C_k}\in\DBof{G}$;

\item $\Sigma_k\,\subseteq\,\Omega\, \subseteq\, \Sigma_k\cup\Theta_k$.
\end{itemize}

\noindent
By~\ref{DB1}, $\proves{\FRJof{G}} \seqfrji{\Sigma_k}{\Theta_k}{C_k}$.
Note that  $\Sigma_1\subseteq \Sigma_2\cup\Theta_2$ and
$\Sigma_2\subseteq \Sigma_1\cup\Theta_1$, hence 
we can build the following $\FRJof{G}$-derivation:
\[
\begin{array}{c}
 \AXC{$\vdots$}
\noLine
\UIC{$\seqfrji{\Sigma_1}{\Theta_1}{C_1}$}
\AXC{$\vdots$}
\noLine
\UIC{$\seqfrji{\Sigma_2}{\Theta_2}{C_2}$}
\RightLabel{$\lor$}
\BIC{$\seqfrji{\underbrace{\Sigma_1,\Sigma_2}_{\Sigma}\;}
{\;\underbrace{\Theta_1\cap\Theta_2}_{\Theta}}{C_1\lor C_2}$}
\DP
\end{array}
\]
By~\ref{DB2}, $\DBof{G}$ contains a sequent
$\seqfrji{\Sigma}{\Theta'}{C_1\lor C_2}$ such that
$\Theta\subseteq\Theta'$.  One can easily check that $\Sigma\subseteq
\Omega$ and $\Omega\subseteq\Sigma\cup\Theta$, which implies
$\Omega\subseteq\Sigma\cup\Theta'$. We conclude
$\DBof{G}\eval\seqgbui{\Omega}{C_1\lor C_2}$.  
\end{proof}

\begin{lemma}\label{lemma:gbuSuccAt}
  Let $G$ be a formula and let $\DBof{G}$ be a saturated database for
  $G$.  Let $\tau=\seqgbu{\Omega}{F}$ be such that $\Omega\subseteq \bG$
  and $F\in\Prime$, and let us assume that:

  \begin{enumerate}[label=(\roman*),ref=(\roman*)]
  \item $F\not\in\Omega$;

  \item for every $A\imp B\in\Omega$,\
    $\DBof{G}\eval\seqgbui{\Omega}{A}$.
  \end{enumerate}
  Then,  $\DBof{G}\eval \tau$.
\end{lemma}

\begin{proof}
  Let $\Omega=\Omat \cup \Omimp$.  Let us assume that $\Omimp$ is
  empty, namely $\Omega = \Omat$, and let $\Gamma' =
  \bGat\setminus\{F\}$.  Then, $\s'=\seqfrj{\G'}{F}$ is an axiom
  sequent of $\FRJof{G}$ hence, by~\ref{DB2}, $\DBof{G}$ contains a
  sequent $\s''=\seqfrj{\G}{F}$ with $\G'\subseteq\G$.  Since
  $\Omat\subseteq \G$, by~\ref{propClo:3} we get $\Omat\subseteq
  \Clo{\G}$, hence $\DBof{G}\eval \tau$.

Let  $\Omimp$ be nonempty and 
let $\Omimp=\{A_1\imp B_1,\dots,A_n\imp B_n\}$ ($n\geq 1$).
For every $k\in\{1,\dots,n\}$, since  $\DBof{G}\eval\seqgbui{\Omega}{A_k}$,
there are $\Sigma_k$ and $\Theta_k$ such that:

\begin{itemize}
\item  $\seqfrji{\Sigma_k}{\Theta_k}{A_k}\in\DBof{G}$;
\item $\Sigma_k\,\subseteq\,\Omega\, \subseteq\, \Sigma_k\cup\Theta_k$.

\end{itemize}
By~\ref{DB1}, $\proves{\FRJof{G}} \seqfrji{\Sigma_k}{\Theta_k}{A_k}$.
One can easily check that, for $i\neq j$, it holds that
 $\Sigma_i\subseteq \Sigma_j\cup\Theta_j$.
For $k\in\{1,\dots,n\}$, let $\Sigma_k\,=\, \Sigat_k\cup\Sigimp_k$ 
and $\Theta_k\,=\, \That_k\cup\Thimp_k$. 
We can build the following $\FRJof{G}$-derivation
($\Sigat$,  $\Sigimp$,  $\That$,  $\Thimp$ are defined as in Fig.~\ref{fig:FRJ}):
\[
\begin{array}{c}
  \AXC{$\vdots$}
  \noLine
  \UIC{$\seqfrji{\Sigma_1}{\Theta_1}{A_1}$}
  \AXC{\dots}
  \AXC{$\vdots$}
  \noLine
  \UIC{$\seqfrji{\Sigma_n}{\Theta_n}{A_n}$}
  \insertBetweenHyps{\hskip -2pt}
  \RightLabel{$\ruleJOINA$}
  \TIC{$\seqfrj{\G}{F}$}
  \DP
  \\[8ex]    
  \G\;=\;\Sigat\,\cup\,\Sigimp\,\cup\,(\That\setminus\{F\})\,\cup\,\Thimp
\end{array}
\]
Note that the above derivation matches the restriction~\ref{PS3} on
the application of $\ruleJOINA$ stated in Sect.~\ref{sec:FRJ}.
By~\ref{DB2}, $\DBof{G}$ contains a sequent $\seqfrj{\G'}{F}$, with
$\G\subseteq\G'$; by~\ref{propClo:4}, $\Clo{\G}\subseteq\Clo{\G'}$.
One can easily check that $\Omega\subseteq\Gamma$ hence,
by~\ref{propClo:3}, $\Omega \subseteq\Clo{\Gamma}$.  Thus
$\Omega \subseteq\Clo{\Gamma'}$, which implies
$\DBof{G}\eval\seqgbui{\Omega}{F}$.  
\end{proof}

In a similar way we can prove that:

\begin{lemma}\label{lemma:gbuSuccOr}
Let $G$ be a formula and
let $\DBof{G}$ be a saturated database for $G$.
  Let $\tau=\seqgbu{\Omega}{C_1\lor C_2}$ be
such that   $\Omega\subseteq \bG$,
and let us assume that:

  \begin{enumerate}[label=(\roman*),ref=(\roman*)]

\item for every $A\imp B\in\Omega$,\
    $\DBof{G}\eval\seqgbui{\Omega}{A}$.

\item
    $\DBof{G}\eval\seqgbui{\Omega}{C_1}$ and $\DBof{G}\eval\seqgbui{\Omega}{C_2}$.
  \end{enumerate}
  Then,  $\DBof{G}\eval \tau$.  \qed
\end{lemma}

We prove the correctness of  $\Search$.

\begin{theorem}[Correctness of $\Search$]\label{theo:search}
  Let $G$ be a formula, let $\tau$ be a $\GBUof{G}$-sequent and let
  $\DBof{G}$ be a saturated database for $G$
  satisfying~\ref{search:ass1} and~\ref{search:ass2}.  Then,
  $\Search(\tau,\DBof{G})$ computes a $\GBUof{G}$-derivation of $\tau$.
\end{theorem}

\begin{proof}
  We prove the assertion by induction on $\wggbu{\tau}$.  Note that,
  whenever we perform a recursive call $\Search(\tau',\DBof{G})$, it
  holds that $\wggbu{\tau'}\prec \wggbu{\tau}$.
Thus, by the induction hypothesis, we can assume that:

\begin{enumerate}[label=(\dag),ref=(\dag)]
\item\label{theo:search:IH}
 every recursive call  $\Search(\tau',\DBof{G})$ yields a
  $\GBUof{G}$-derivation of $\tau'$, provided that $\tau'$ and $\DBof{G}$
  satisfy assumptions~\ref{search:ass1} and~\ref{search:ass2}.
\end{enumerate}

\noindent
We have to show that, in each of the
cases~\ref{search:1}--\ref{search:5}, a $\GBUof{G}$-derivation $\Tcal$
of $\tau$ is built.  Case~\ref{search:1} is immediate.  Let us
consider Case~\ref{search:2}, namely $\tau$ is non-critical.  One can
easily check that at least an application of a rule $\Rcal$, as
displayed in~\ref{search:2}, is possible.  We have to show that each
premise $\tau_j$ of the selected rule $\Rcal$ satisfies
assumptions~\ref{search:ass1} and~\ref{search:ass2}. Le us consider
assumption~\ref{search:ass2}; if $\tau_j$ is regular,
then~\ref{search:ass2} trivially holds.  If $\tau_j$ is irregular,
then, by inspecting the rules of $\GBUof{G}$, one can check that
$\Lhs{\tau_j}=\Lhs{\tau}$; thus, since $\tau$
satisfies~\ref{search:ass2}, $\tau_j$ satisfies~\ref{search:ass2} as
well.  The validity of~\ref{search:ass1} follows by
Lemma~\ref{lemma:gbuInv}.  For instance, let us assume that
$\Rcal=L\lor$ and that the selected application is
  \[
\AXC{$\tau_1\;=\;\seqgbu{A,\Psi}{C}$ }
\AXC{$\tau_2\;=\;\seqgbu{B,\Psi}{C}$ }
\RightLabel{$L\lor$}
    \BIC{$\tau\;=\;\seqgbu{A\lor B,\Psi}{C}$}
\DP
  \]
  Since $\tau$ satisfies~\ref{search:ass1}, we have $\DBof{G}\neval\tau$;
  by Lemma~\ref{lemma:gbuInv}\ref{lemma:gbuInv:3}, it follows that both
  $\DBof{G}\neval\tau_1$ and $\DBof{G}\neval\tau_2$, hence both $\tau_1$ and
  $\tau_2$ satisfy~\ref{search:ass1}.  By~\ref{theo:search:IH}, for
  every $1\leq j \leq n$, $\Tcal_j$ is a $\GBUof{G}$-derivation of
  $\tau_j$, hence $\Tcal$ is a $\GBUof{G}$-derivation of $\tau$.

  Let $\tau$ match Case~\ref{search:3}.  Then,
  $\tau=\seqgbui{\Omega}{C_1\lor C_2}$ and $\Omega\subseteq\bG$.  Let
  $\tau_k= \seqgbui{\Omega}{C_k}$, with $k\in\{1,2\}$.  We have to
  guarantee that there exists $k\in\{1,2\}$ such that
  $\DBof{G}\neval \tau_k$.  If both $\DBof{G}\eval \tau_1$ and
  $\DBof{G}\eval \tau_2$ hold, by Lemma~\ref{lemma:gbuiOr} we would
  conclude $\DBof{G}\eval\tau$, in contradiction with the fact that
  $\tau$ and $\DBof{G}$ satisfy~\ref{search:ass1}.  Thus, we can
  choose $k$ such that $\DBof{G}\neval \tau_k$.  Since $\tau_k$ and
  $\DBof{G}$ satisfy assumptions~\ref{search:ass1}
  and~\ref{search:ass2}, by~\ref{theo:search:IH} $\Tcal_k$ is a
  $\GBUof{G}$-derivation of $\tau_k$, hence $\Tcal$ is a
  $\GBUof{G}$-derivation of $\tau$.

  Let $\tau$ match Case~\ref{search:4}.  Since Case~\ref{search:1}
  does not hold, $\tau$ is not a $\GBUof{G}$-axiom, hence
  $F\not\in\Omega$.  Let us assume that for every $A\imp B\in\Omega$,
  we have $\DBof{G}\eval\seqgbui{\Omega}{A}$.  By
  Lemma~\ref{lemma:gbuSuccAt}, it would follow $\DBof{G}\eval\tau$,
  against the assumption~\ref{search:ass1} of the lemma.  Thus, we can
  pick $A\imp B\in\Omega$ such that
  $\DBof{G}\neval \seqgbui{\Omega}{A}$.  By~\ref{theo:search:IH},
  $\Tcal_1$ is a $\GBUof{G}$-derivation of
  $\tau_1=\seqgbui{\Omega}{A}$.  By~\ref{theo:search:IH} and
  Lemma~\ref{lemma:gbuInv}\ref{lemma:gbuInv:4}, $\Tcal_2$ is a
  $\GBUof{G}$-derivation of $\tau_2=\seqgbu{\Omega_B}{F}$.  We
  conclude that $\Tcal$ is a $\GBUof{G}$-derivation of
  $\tau=\seqgbu{\Omega}{F}$.

  We observe that, if $\tau$ does not match any of the conditions in
  cases~\ref{search:1}--\ref{search:4}, then $\tau$ has the form
  stated in Case~\ref{search:5}.  Case~\ref{search:5} follows along
  the lines of Case~\ref{search:4}, exploiting
  Lemma~\ref{lemma:gbuSuccOr}.  
\end{proof}

By the correctness of  $\Search$,
it follows that $\GBUof{G}$ and
 $\FRJof{G}$ are dual calculi,
in the sense that 
provability in  $\GBUof{G}$
is the complement of provability in $\FRJof{G}$:

\begin{theorem}\label{theo:GBU-FRJ}
$\proves{\GBUof{G}} G$ iff
$\nproves{\FRJof{G}} G$. 
\end{theorem}  

\begin{proof}
If $\proves{\GBUof{G}} G$,
 by the Soundness of $\GBUof{G}$  (Theorem~\ref{theo:GBUsound})
we get $G\in\IPL$;
 by the  Soundness of $\FRJof{G}$
(Theorem~\ref{theo:FRJsound})
we conclude $\nproves{\FRJof{G}} G$.

Conversely, let $\nproves{\FRJof{G}} G$, let $\DBof{G}$ be a saturated
database for $G$ and let $\tau$ be the $\GBUof{G}$-sequent
$\seqgbu{}{G}$.  Note that $\DBof{G}$ does not contain any regular
sequent of the kind $\s=\seqfrj{\G}{G}$; otherwise, by~\ref{DB1}, we
would get $\proves{\FRJof{G}}\s$, contradicting the assumption that
$G$ is not provable in $\FRJof{G}$.  This implies
$\DBof{G}\neval \tau$, hence $\tau$ and $\DBof{G}$ satisfy the
assumptions~\ref{search:ass1} and~\ref{search:ass2} of \Search.  By
the correctness of \Search (Theorem~\ref{theo:search}),
$\Search(\tau,\DBof{G})$ computes a $\GBUof{G}$-derivation of $\tau$;
we conclude $\proves{\GBUof{G}} G$.  
\end{proof}

As a corollary, we get the completeness of $\FRJof{G}$ and $\GBUof{G}$:

\begin{theorem}[Completeness of $\FRJof{G}$ and $\GBUof{G}$]

  \begin{enumerate}[label=(\roman*),ref=(\roman*)]
\item\label{theo:compl1}
$G\not\in\IPL$ implies $\proves{\FRJof{G}} G$.

  \item\label{theo:compl2}
$G\in\IPL$ implies $\proves{\GBUof{G}} G$.
  \end{enumerate}
 \end{theorem}

 \begin{proof}
Let $G\not\in\IPL$.
 By  the Soundness of $\GBUof{G}$
 (Theorem~\ref{theo:GBUsound}),
  $\nproves{\GBUof{G}} G$;
 by Theorem~\ref{theo:GBU-FRJ},  $\proves{\FRJof{G}} G$.
Let $G\in\IPL$.
 By  the Soundness of $\FRJof{G}$
 (Theorem~\ref{theo:FRJsound}),
  $\nproves{\FRJof{G}} G$;
 by Theorem~\ref{theo:GBU-FRJ},  $\proves{\GBUof{G}} G$.
 \end{proof}

\section{Minimality}
\label{sec:minimality}

In this section we prove that, given a non-valid formula $G$, one can
build an $\FRJof{G}$-derivation $\Dcal$ of $G$ such that $\Mod{\Dcal}$
is a countermodel of $G$ having minimal height. Such a derivation can
be constructed by tweaking the proof-search procedure defined in
Sec.~\ref{sec:proofsearch}.

Let $G\not\in\IPL$; the \emph{height} of $G$, denoted by $\height{G}$,
is the minimum height of a countermodel for $G$; formally:
\[
\height{G}\;=\;\min \{\,
\height{\K}~|~\mbox{$\K=\stru{P,\leq,\rho,V}$ and $\rho\nforcing G$}
\,\}  
\]
Note that $\height{G}=0$ iff $G$ is not classically valid.  For
instance, the formulas $S$ and $T$ of Ex.~\ref{ex:frjNishimura} have
height 2, the formula $K$ of Ex.~\ref{ex:frjKP} has height 1.  Let
$\Dcal$ be an $\FRJof{G}$-derivation of $G$.  Since $\Mod{\Dcal}$ is a
countermodel for $G$ (see Theorem~\ref{theo:soundFRJ}), we have
$\height{\Mod{\Dcal}}\geq \height{G}$.  We show that we can build an
$\FRJof{G}$-derivation $\tilde\Dcal$ of $G$ such that
$\height{\Mod{\tilde\Dcal}} = \height{G}$.  To this aim, we prove
that:

\begin{enumerate}[label=(K\arabic*), ref=(K\arabic*)]
\item\label{minimality:1} given a countermodel $\K$ for $G$, 
  there exists an $\FRJof{G}$-derivation $\tilde\Dcal$ of $G$
  such that $\height{\Mod{\tilde\Dcal}}\leq \height{\K}$.
\end{enumerate}

\noindent
By~\ref{minimality:1}, choosing as $\K$ a countermodel for $G$ having
the minimal height $\height{G}$, we get an $\FRJof{G}$-derivation
$\tilde\Dcal$ of $G$ such that $\height{\Mod{\tilde\Dcal}}
=\height{G}$.

Let  $\Dcal$ be an $\FRJof{G}$-derivation of $G$.
The height of the countermodel $\Mod{\Dcal}$
is determined by the maximum number of applications
of join rules along a branch of $\Dcal$.
To account for this, 
we introduce the notion of rank.
Let $\s$ be a sequent occurring in  $\Dcal$;
the \emph{rank} of  $\s$, denoted by $\Rn{\s}$,
is inductively defined  as follows:

\begin{itemize}
\item If $\s$ is an irregular axiom, then $\Rn{\s}=-1$.

\item If $\s$ is a regular axiom, then $\Rn{\s}=0$.

\item If $\s$ is not an axiom, let
  \[
\AXC{$\s_1 \;\cdots\;\s_n$}
\RightLabel{$\Rcal$}
\UIC{$\s $}
\DP    
\]
be the rule applied to get $\s$ ($n\geq 1$). Then:
\[
\Rn{\s}\;=\;\max\{\,\Rn{\s_1},\dots,\Rn{\s_n}\,\}\,+\,c\quad\textrm{where}\; c\;=\;
\begin{cases}
\;1\quad\mbox{if $\Rcal\in\{\ruleJOINA,\ruleJOINO\}$}  
\\
\;0\quad\mbox{otherwise}
\end{cases}
\] 
\end{itemize}

\noindent
The rank of $\Dcal$, denoted by  $\Rn{\Dcal}$,
is the rank  of the root sequent of $\Dcal$.
One can easily prove that, for every regular sequent
$\s$  in $\Dcal$, $\Rn{\s}$ coincides with
the height of the world $\phi(\s)$ in $\Mod{\s}$
(where $\phi$ is the map associated with $\Dcal$).
As an immediate consequence we get:

\begin{lemma}\label{lemma:rnk}
  Let $\Dcal$ be an $\FRJof{G}$-derivation of $G$.
  Then,
  $\Rn{\Dcal}=\height{\Mod{\Dcal}}$.  
\qed
\end{lemma}

Let $\K=\stru{P,\leq,\rho, V}$ be a countermodel for $G$ and $\a\in
P$. We set:
\begin{itemize}
\item $\a\forcings H$ iff $\a\forcing H$ and either $H\in\PV$ or
  $H=A\imp B$ and $\a\nforcing A$.


\item 
$\Lambda_\a\;=\; \{\,\mbox{$A\in\Sfl{G}~|~\a\forcing A$}\,\}$.

\item
  $\Lambdas_\a\;=\; \{\,\mbox{$A\in\Sfl{G}~|~\a\forcings A$}\,\}$.

\item 
$\Omega_\a\;=\; \{\,\mbox{$C\in\Sfr{G}~|~\a\nforcing  C$}\,\}$.
\end{itemize}

\noindent
To prove~\ref{minimality:1}, we exploit the following lemma (the proof
is deferred to the end of this section):

\begin{lemma}\label{lemma:minMod}
  Let $\K=\stru{P,\leq,\rho, V}$ be a countermodel for $G$, let $\a\in
  P$ and $C\in\Omega_\a$.  
  \begin{itemize}
  \item There exists an $\FRJof{G}$-derivation
  $\derfrji{\a}{C}$ of $\sigmairr{\a}{C} =\seqfrji{\Sigma}{\Theta}{C}$ such that:
  \begin{enumerate} [label=(\roman*), ref=(\roman*)]
  \item \label{lemma:minMod:1} $\Rn{ \derfrji{\a}{C} } \,<\,
    \height{\a}$;
  \item \label{lemma:minMod:2} $\Sigma\,\subseteq\, \Lambdas_\a \,
    \subseteq \,\Sigma\cup\Theta$.
  \end{enumerate}
  
\item There exists an $\FRJof{G}$-derivation $\derfrjr{\a}{C}$ of
  $\sigmareg{\a}{C}=\seqfrj{\G}{C}$ such that:
  \begin{enumerate} [start=3, label=(\roman*), ref=(\roman*)]
  \item   \label{lemma:minMod:3}
    $\Rn{ \derfrjr{\a}{C}  }\,\leq\,\height{\a}$;
        
  \item   \label{lemma:minMod:4}
  there is $\b\in P$ such that
  $\a\leq \b$ and $\Lambdas_\b\,\subseteq\, \G$.
\qed
\end{enumerate}
\end{itemize}
\end{lemma}

\noindent
Point~\ref{minimality:1} follows by the above lemma, 
 taking as  $\tilde\Dcal$ the $\FRJof{G}$-derivation
$\derfrjr{\rho}{G}$ associated with the root $\rho$ of $\K$.
Indeed,  the root sequent of $\tilde\Dcal$ is
$\sigmareg{\rho}{G}=\seqfrj{\G}{G}$,
hence $\tilde\Dcal$ is an $\FRJof{G}$-derivation of $G$.  
By Theorem~\ref{theo:soundFRJ}, $\Mod{\tilde\Dcal}$ is a countermodel for
$G$ and, by point~\ref{lemma:minMod:3} of the lemma,
$\Rn{\tilde\Dcal}\leq \height{\rho}$, namely $\Rn{\tilde\Dcal}\leq
\height{\K}$.  By Lemma~\ref{lemma:rnk}, we get
$\height{\Mod{\tilde\Dcal}}\leq \height{\K}$, which 
proves~\ref{minimality:1}. As discussed above,
from~\ref{minimality:1} we conclude:

\begin{theorem}\label{theo:minMod}
  Let $G\not\in \IPL$.
  Then, there exists an $\FRJof{G}$-derivation $\tilde\Dcal$ of $G$ such that
  $\height{\Mod{\tilde\Dcal}}=\height{G}$. \qed
\end{theorem}

We remark that Theorem~\ref{theo:minMod} provides an alternative proof of
the completeness of $\FRJof{G}$.~\footnote{Actually,
  in~\cite{FerFio:2017} the completeness of $\FRJof{G}$ has been proved
  along these lines.}

We have now to prove Lemma~\ref{lemma:minMod}.  We give a constructive
proof to define the sequents $\sigmairr{\a}{C}$ and $\sigmareg{\a}{C}$
and the related derivations, which relies on the following strategy:
\begin{itemize}
\item we visit the model $\K$ top-down, considering the worlds $\a$ of
  $\K$ in increasing order of height;

\item for each world $\a$, we firstly define all the irregular
  sequents $\sigmairr{\a}{C}$ and then all the regular sequents
  $\sigmareg{\a}{C}$;

\item we pick the formulas $C$ of $\Omega_\a$ in increasing order of
  size.
\end{itemize}

\noindent
Since $\rho$ is the bottom world of $\K$ and $G$ is the formula of
maximum size in $\Omega_G$, the $\FRJof{G}$-derivation
$\derfrjr{\rho}{G}$ is obtained at the end of the process.  To
highlight the main insights of the above construction, we present some
examples.


\begin{example}\label{ex:STMinH}
  Let $S$ be the instance of Scott Principle in
  Ex.~\ref{ex:frjNishimura}:
  \[
    \begin{array}{l}
      S \;= \; H\,\imp\, \neg\neg p \lor \neg p
      \qquad
      H\;=\; (\neg\neg p\imp p)\, \imp\, \neg p\lor  p
      \\[1ex]
      \Sfl{S} \:=\: 
      \{\: H,\:\neg p\lor p,\:\neg\neg p,\:\neg p,\:p\:\} 
      \\[.5ex]
      \Sfr{S} \:=\: 
      \{\: S,\:\neg\neg p\lor \neg p,\:\neg\neg p \imp p,\:
      \neg\neg p,\:\neg p,\: p,\: \bot   \:\}    
    \end{array}
    \]
 We have $\height{S}=2$.
 Let us consider the countermodel $\K_S$ for $S$ having height 2 depicted below,
 consisting of the worlds $\a$ and $\b$ of height 0, the world   $\g$ of height 1
 and the root $\rho$ of height 2:
 \begin{center}
   \begin{tikzpicture}
     \path[scale=.42,grow'=up,every node/.style={fill=gray!30,rounded corners},
     level 1/.style = {sibling distance = 13em}, 
     edge from parent/.style={black,draw}]
     node{$\rho$:  }
     [level distance=20mm] 
     child{
       child {node{$\a$:}}
     }
     child{
       node{$\g$: }
       child{
         node{$\b$: $p$}
       }
     }
     ;
   \end{tikzpicture}
 \end{center}
 
 \noindent
 We define the sequents $\sigmairr{\a}{C}$ and $\sigmareg{\a}{C}$
 matching Lemma~\ref{lemma:minMod}.  For each sequent $\s$, we display the
 rule applied to obtain $\s$ and the rank of $\s$.  
Using these annotations, one can immediately build the derivations
 $\derfrji{\a}{C}$ and  $\derfrjr{\a}{C}$.
We traverse the
 model $\K$ downwards, starting from the world $\a$ of height 0. We
 have:
 \[
 \Lambdas_\a\,=\,\{\:\neg p\:\}   
 \qquad
 \Lambda_\a\;=\;\Lambdas_\a\:\cup\:\{\:\neg p\lor p,\:H\:\}
 \qquad
 \Omega_\a \;=\;\{\:\bot,\:p,\:\neg\neg p\:\}
 \]
Note that $H\not\in\Lambdas_\a$ since  $\a\forcing \neg\neg p \imp p$.
 Firstly, we   define all the sequents  $\sigmairr{\a}{C}$,
 where $C\in \Omega_\a$; formulas $C$ are considered
 in  increasing order of size.
 \[
 \begin{array}{lrl|l|l}
   \mbox{Sequent}&&&\mbox{Applied rule}&\mbox{Rank}\\[.5ex]
   \hline
                 &&&&\\[-1.5ex]
   \sigmairr{\a}{\bot} & \tseqfrji{}{p,\:H,\:\neg\neg p,\:\neg p}{\bot} \qquad &  \ruleAXI &-1
   \\
   \sigmairr{\a}{p}  &   
   \tseqfrji{}{H,\:\neg\neg p,\:\neg p}{p}  & \ruleAXI &-1
   \\
   \sigmairr{\a}{\neg \neg p} \qquad &
   \tseqfrji{\neg p}{p,\:H,\:\neg\neg p}{\neg \neg p}\qquad  
   & \ruleIMPi\;  \sigmairr{\a}{\bot} \qquad &-1
 \end{array}
 \]     
 Secondly, we define the sequents $\sigmareg{\a}{C}$,
 where  $C\in \Omega_\a$.
 \[
 \begin{array}{lrlll}
 \sigmareg{\a}{\bot}  &
    \tseqfrj{\neg p}{\bot}  & \ruleJOINA\;  \sigmairr{\a}{p}\qquad & 0
    \\
 \sigmareg{\a}{p} &
    \tseqfrj{\neg p}{p}  & \ruleJOINA\;  \sigmairr{\a}{p}  & 0
    \\
 \sigmareg{\a}{\neg\neg p} \qquad &
    \tseqfrj{\neg p}{\neg\neg p} \qquad & \ruleIMPi\; \sigmareg{\a}{\bot} & 0
   \end{array}
\] 
Let us consider the world $\b$ of height 0. We have:
\[
 \Lambdas_\b\,=\,\{\:p,\:\neg\neg p\:\}   
    \qquad
    \Lambda_\b\,=\,\Lambdas_\b\:\cup\:\{\:\neg p\lor p,\:H\:\}
    \qquad
    \Omega_\b \,=\,\{\:\bot,\:\neg p\:\}
\]  
For $C\in\Omega_\b$,  we   define the sequents  $\sigmairr{\b}{C}$:
\[
   \begin{array}{lrlll}
 \sigmairr{\b}{\bot}\qquad &
    \tseqfrji{}{p,\:H,\:\neg\neg p,\:\neg p}{\bot} \qquad &  \ruleAXI\qquad & -1
      \\
 \sigmairr{\b}{\neg p} &
     \tseqfrji{p}{H,\:\neg\neg p,\:\neg p}{\neg p} &
    \ruleIMPi\;  \sigmairr{\b}{\bot}  &-1
   \end{array}
 \]
For $C\in\Omega_\b$,  sequents  $\sigmareg{\b}{C}$ are:
\[
  \begin{array}{lrlll}
\sigmareg{\b}{\bot}\qquad &
 \tseqfrj{p,\:\neg\neg p}{\bot}\qquad& \ruleJOINA\;\sigmairr{\b}{\neg p}\qquad & 0
     \\
  \sigmareg{\b}{\neg p} &
  \tseqfrj{p,\:\neg\neg p}{\neg p} &
 \ruleIMPi\; \sigmareg{\b}{\bot} & 0
  \end{array}
\]     

\noindent
Let us consider $\g$, the only world of height 1. We have:
\[
\Lambdas_\g\;=\;\Lambda_\g\;=\;\{\,H,\,\neg \neg p\,\}   
    \qquad
    \Omega_\g \;=\;\{\,\bot,\,p,\,\neg p,\,\neg\neg p\imp p\,\}
\]  
For $C\in\Omega_\g$,  sequents  $\sigmairr{\g}{C}$ are:
\[
   \begin{array}{lrlll}
 \sigmairr{\g}{\bot} &\tseqfrji{}{p,\:H,\:\neg\neg p,\:\neg p}{\bot} \qquad &  \ruleAXI &-1
\\
\sigmairr{\g}{p}& \tseqfrji{}{H,\:\neg\neg p,\:\neg p}{p}  & \ruleAXI &-1
\\
 \sigmairr{\g}{\neg p} &
    \tseqfrji{}{H,\:\neg\neg p}{\neg p}  & \ruleIMPni\;\sigmareg{\b}{\bot}\qquad & 0
 \\
\sigmairr{\g}{\neg \neg p\imp p} \qquad&
 \tseqfrji{\neg\neg p}{H,\:\neg p}{\neg \neg p\imp p} \qquad & \ruleIMPi\; \sigmairr{\g}{p} & -1
   \end{array}
 \]
For $C\in \Omega_\g$,
  sequents $\sigmareg{\g}{C}$ are:
\[
   \begin{array}{lrlll}
\sigmareg{\g}{\bot} &
\tseqfrj{p,\:\neg\neg p}{\bot}\qquad& \ruleJOINA\;\sigmairr{\b}{\neg p}\qquad & 0
\\
   \sigmareg{\g}{p}\qquad &
    \tseqfrj{H,\:\neg \neg p}{p}\qquad & \ruleJOINA\;\sigmairr{\g}{\neg p}\;\sigmairr{\g}{\neg \neg p\imp p} \quad& 1
     \\
\sigmareg{\g}{\neg p}\qquad &\tseqfrj{p,\:\neg\neg p}{\neg p} &
 \ruleIMPi\; \sigmareg{\b}{\bot} & 0
\\
 \sigmareg{\g}{\neg\neg p \imp p}\qquad &
    \tseqfrj{H,\:\neg \neg p}{\neg\neg p\imp p} \qquad& \ruleIMPi\;  \sigmareg{\g}{p} & 1
   \end{array}
\]
Finally, let us consider the root $\rho$ having height 2. We have:
\[
\Lambdas_\rho=\Lambda_\rho=\{H\}  
    \qquad
    \Omega_\rho =\Sfr{S}
\]
We can inherit the following definitions from the worlds $\a$, $\b$, $\g$:
\[
\begin{array}{l}
    \sigmairr{\rho}{\bot} \;=\;\sigmairr{\a}{\bot}  
    \qquad
    \sigmairr{\rho}{p} \;=\;\sigmairr{\g}{p}
    \qquad
    \sigmairr{\rho}{\neg p} \,=\,\sigmairr{\g}{\neg p}      
  \\[.5ex]
\sigmareg{\rho}{\bot} \;=\;\sigmareg{\b}{\bot}
      \qquad
    \sigmareg{\rho}{p} \,=\,\sigmareg{\a}{p}  
  \qquad
    \sigmareg{\rho}{\neg p} \;=\;\sigmareg{\b}{\neg p}
\\[.5ex]
    \sigmareg{\rho}{\neg \neg p} \;=\;\sigmareg{\a}{\neg \neg p}  
    \qquad
    \sigmareg{\rho}{\neg \neg p \imp p} \,=\,\sigmareg{\g}{\neg \neg p\imp p}
    \end{array}
\]
For $C\in\{\,\neg\neg p,\,\neg\neg p\imp  p,\, \neg\neg p\lor \neg p,\, S \,\}$,
the sequents $\sigmairr{\rho}{C}$ are:
\[
   \begin{array}{lrllll}
\sigmairr{\rho}{\neg\neg p}  &
    \tseqfrji{}{H}{\neg\neg p}  & \ruleIMPni\; \sigmareg{\a}{\bot} & 0
    \\
\sigmairr{\rho}{\neg\neg p\imp  p}  &
    \tseqfrji{}{H}{\neg\neg p \imp p}  & \ruleIMPni\; \sigmareg{\g}{p} & 1
     \\
\sigmairr{\rho}{\neg\neg p\lor \neg p} \qquad &
  \tseqfrji{}{H}{\neg\neg p \lor\neg p}  \qquad &
     \lor\;\sigmairr{\rho}{\neg\neg p}\;\sigmairr{\g}{\neg p} \quad & 0
    \\
\sigmairr{\rho}{S} &
    \tseqfrji{H}{}{S}  & \ruleIMPi\; \sigmairr{\rho}{\neg\neg p\lor \neg p} & 0
   \end{array}
\]  
For $C\in\{\,\neg\neg p\lor \neg p,\, S \,\}$,
the sequents $\sigmareg{\rho}{C}$ are:
\[
   \begin{array}{lrlll}
\sigmareg{\rho}{\neg\neg p\lor \neg p}\qquad  &
\tseqfrj{H}{\neg\neg p \lor \neg p}\qquad  &
 \ruleJOINO \;\sigmairr{\rho}{\neg\neg p\imp p} \;\sigmairr{\rho}{\neg\neg p}\;\sigmairr{\g}{\neg p}
\quad & 2 
   \\
\sigmareg{\rho}{S}  &
    \tseqfrj{H}{S}  & \ruleIMPi\;\sigmareg{\rho}{\neg\neg p\lor \neg p} & 2
   \end{array}
\]
The $\FRJof{S}$-derivation $\derfrjr{\rho}{S}$ of $\sigmareg{\rho}{S}$
obtained in the end is an $\FRJof{S}$-derivation of the goal $S$.
Note that $\derfrjr{\rho}{S}$ essentially coincides with the
derivation in Fig.~\ref{fig:frjST}, hence the model
$\Mod{\derfrjr{\rho}{S}}$ extracted from $\derfrjr{\rho}{S}$ is
isomorphic to the countermodel $\K_S$ displayed at the beginning of
this example.
\EndEs
\end{example}


In the previous example, the model $\K_S$ initially chosen is a
\emph{minimal} countermodel for $S$,  since there exists no
countermodel for $S$ having less than 4 worlds, and the obtained model
$\Mod{\derfrjr{\rho}{S}}$ is a minimal countermodel for $S$ as well.
One may wonder if this is always the case.
The answer is negative, as shown in the next example;
we also point out that  Lemma~\ref{lemma:minMod} only
sets a bound on the height of $\Mod{ \derfrjr{\rho}{G} }$
and not on the number of worlds.

\begin{example}\label{ex:DUMnonMIN}
  Let $C$ be the formula
  \[
  \begin{array}{l}
    C \;= \; A\lor B
    \qquad
    A\;=\; (p_1\imp p_2)\lor (p_2\imp p_1)
    \qquad
    B\;=\; (q_1\imp q_2)\lor (q_2\imp q_1)
  \end{array}
  \]
  We have $\height{C}=1$.  Let $\K_C$ be the following countermodel
  for $C$ of height 1, consisting of the worlds $\a$ and $\b$ of
  height 0 and of the root $\rho$ of height 1:
  
  \begin{center}
    \begin{tikzpicture}
      \path[scale=.42,grow'=up,every node/.style={fill=gray!30,rounded corners},
      level 1/.style = {sibling distance = 60mm}, 
      edge from parent/.style={black,draw}]
      node{$\rho$:  }
      [level distance=20mm] 
      child{
        node{$\a$: $p_1,\,q_1$}
      }
      child{
        node{$\b$: $p_2,\,q_2$  }
      }
      ;
    \end{tikzpicture}
  \end{center}
  
  \noindent
  One can easily check that there is no countermodel for $C$ having
  less than 3 worlds, hence $\K_C$ is a minimal countermodel for $C$.
  We have:
  \[
  \begin{array}{l}
    \Sfl{C} \;=\; 
    \{\:   p_1,\:p_2,\:q_1,\:q_2   \:\} 
    \\[1ex]   
    \Sfr{C} \:=\:  \{\: 
    C,\:A,\:B,\: p_1\imp p_2,\;p_2\imp p_1,\:q_1\imp q_2,\;q_2\imp q_1,\; p_1,\:p_2,\:q_1,\:q_2
    \:\}   
    \\[1ex]
    \Lambdas_\a\;=\;\Lambda_\a\;=\;\{\:p_1,\:q_1\:\}   
    \qquad
    \Omega_\a \;=\;\{\:\:p_2,\: q_2,\: p_1\imp p_2,\:q_1\imp q_2 \:\}
    \\[1ex]
    \Lambdas_\b\;=\;\Lambda_\b\;=\;\{\:p_2,\:q_2\:\}   
    \qquad
    \Omega_\b \;=\;\{\:   \:p_1,\: q_1,\;p_2\imp p_1,\:q_2\imp q_1\:\}
    \\[1ex]
    \Lambdas_\rho\;=\;\Lambda_\rho\;=\;\emptyset
    \qquad
    \Omega_\rho \;=\;\Sfr{C}
  \end{array}
  \]
  We define the sequents $\sigmairr{\d}{C}$ and $\sigmareg{\d}{C}$,
  where $\d\in\{\a,\b,\rho\}$ and $C\in\Omega_\d$,
  only showing the sequents needed to prove the goal.

  \[
  \begin{array}{lrlll}
    \sigmareg{\a}{p_2}  &  
    \tseqfrj{p_1,\:q_1,\:q_2}{p_2}  & \ruleAXR
    \\
    \sigmareg{\a}{q_2}  &  
    \tseqfrj{p_1,\:p_2,\:q_1}{q_2}  & \ruleAXR
    \\
    \hline
    \sigmareg{\b}{p_1}  &  
    \tseqfrj{p_2,\:q_1,\:q_2}{p_1}  & \ruleAXR
    \\
    \sigmareg{\b}{q_1}  &  
    \tseqfrj{p_1,\:p_2,\:q_2}{q_1}  & \ruleAXR
    \\
    \hline
    \sigmairr{\rho}{p_1 \imp p_2}\qquad  &  
    \tseqfrji{}{q_1,\:q_2}{p_1\imp p_2} \qquad &  \ruleIMPni\; \sigmareg{\a}{p_2}
    \\
    \sigmairr{\rho}{p_2 \imp p_1}  &  
    \tseqfrji{}{q_1,\:q_2}{p_2\imp p_1}  &  \ruleIMPni\; \sigmareg{\b}{p_1}
    \\
    \sigmairr{\rho}{q_1 \imp q_2}  &  
    \tseqfrji{}{p_1,\:p_2}{q_1\imp q_2}  &  \ruleIMPni\; \sigmareg{\a}{q_2}
    \\
    \sigmairr{\rho}{q_2 \imp q_1}  &  
    \tseqfrji{}{p_1,\:p_2}{q_2\imp q_1}  &  \ruleIMPni\; \sigmareg{\b}{q_1}
    \\
    \sigmairr{\rho}{A}  &  
    \tseqfrji{}{q_1,\:q_2}{A}  &  \lor\; \sigmairr{\rho}{p_1\imp p_2} \; \sigmairr{\rho}{p_2\imp p_1} 
    \\
    \sigmairr{\rho}{B}    &  
    \tseqfrji{}{p_1,\:p_2}{B}  &  \lor\; \sigmairr{\rho}{q_1\imp q_2} \; \sigmairr{\rho}{q_2\imp q_1} 
    \\
    \hline
    \sigmareg{\rho}{C} &  
    \tseqfrj{}{A\lor B}  & \ruleJOINO\;  \sigmairr{\rho}{A}\; \sigmairr{\rho}{B}
  \end{array}
  \]  
  The goal $C$ is proved by the $\FRJof{C}$-derivation
  $\derfrjr{\rho}{C}$ of $\sigmareg{\rho}{C}$.  The model $\Mod{
    \derfrjr{\rho}{C} }$ extracted from $\derfrjr{\rho}{C}$ is:
  
  \begin{center}
 \begin{tikzpicture}
      \path[scale=.42,grow'=up,every node/.style={fill=gray!30,rounded corners},
      level 1/.style = {sibling distance = 80mm}, 
      edge from parent/.style={black,draw}]
      node{$\sigmareg{\rho}{C}$:  }  
      [level distance=30mm]  
            child{
        node{$\sigmareg{\a}{p_2}$: $p_1,\,q_1,\,q_2$}
      }
         child{
        node{$\sigmareg{\a}{q_2}$: $p_1,\,p_2,\,q_1$}
      }
      child{
        node{$\sigmareg{\b}{p_1}$: $p_2,\,q_1,\,q_2$}
      }
         child{
        node{$\sigmareg{\b}{q_1}$: $p_1,\,p_2,\,q_2$}
      }
      ;
    \end{tikzpicture}
 \end{center}

 \noindent
 Such a model has the minimal height 1, as expected, but it is not
 minimal.  To obtain minimal models, we should redesign $\FRJof{G}$ in
 a multi-succedent style, so that axioms of the kind
 $\seqfrj{p_1,q_1}{p_2,q_2}$ and $\seqfrj{p_2,q_2}{p_1,q_1}$ are
 allowed.  Thus, we should have regular axioms of the kind
 $\seqfrj{\Gat}{\Dat}$, where $\Gat$ and $\Dat$ are disjoint sets of
 propositional variables such that $\Gat\subseteq\Lhs{G}$ and
 $\Dat\subseteq\Rhs{G}$.  However, this would cause an exponential
 blow-up of the number of provable sequents and proof-search would
 become unfeasible. E.g., in this case we would get 64 regular
 axioms instead of 5.
 \EndEs
\end{example}

We remark that we can modify the proof-search procedure of
Fig.~\ref{fun:FSearch} so to generate $\FRJof{G}$-derivations of
minimal height; this is obtained by delaying the application of join
rules as much as possible, mimicking the strategy applied in the above
examples.

The rest of this section is devoted to  the detailed proof of  Lemma~\ref{lemma:minMod}.
Firstly, we prove some closure
properties of sets  $\Lambda_\a$ and $\Lambdas_\a$:

\begin{lemma}\label{lemma:closure}
  Let $\K$ be a countermodel for $G$ and $\a$  a world of $\K$. Then,
  $\Lambda_\a=\Clo{\Lambda_\a}=\Clo{\Lambdas_\a}$.
\end{lemma}

\begin{proof}
To prove the assertion, we show  that:

  \begin{enumerate} [label=(\roman*), ref=(\roman*)]
  \item\label{lemma:closure:1}
    $\Lambda_\a\,\subseteq\, \Clo{\Lambda_\a}$.

  \item\label{lemma:closure:2}
    $\Clo{\Lambda_\a}\,\subseteq\, \Lambda_\a$.

\item\label{lemma:closure:3}
    $\Clo{\Lambdas_\a}\,\subseteq\, \Clo{\Lambda_\a}$.
    
\item\label{lemma:closure:4}
    $\Lambda_\a\,\subseteq\, \Clo{\Lambdas_\a}$.

\item\label{lemma:closure:5}
    $\Clo{\Lambda_\a}\,\subseteq\, \Clo{\Lambdas_\a}$.
  \end{enumerate}
    
  \noindent
  Point~\ref{lemma:closure:1} immediately follows by~\ref{propClo:3}.
  To prove~\ref{lemma:closure:2},
  let  $C\in\Clo{\Lambda_\a}$; 
by  induction on $\size{C}$, we show that 
$C\in \Lambda_\a$.
The case $C\in\Prime$ immediately follows by~\ref{propClo:5}.
Let $C=A\land B$. 
Then, $A\land B\in\Lambda_\a$ or  $\{A,B\}\subseteq \Clo{\Lambda_\a}$.
In the former case, we are done.
In the latter case, 
by the induction hypothesis we have   $\{A,B\}\subseteq\Lambda_\a$.
Thus $\a\forcing A$ and  $\a\forcing B$,
which implies  $\a\forcing A\land B$, 
hence $A\land B\in\Lambda_\a$.
The cases $C= A\lor B$ and  $C= A\imp B$ are similar.

Point~\ref{lemma:closure:3} follows by the fact that
$\Lambdas_\a\subseteq\Lambda_\a$ and by~\ref{propClo:4}.

To prove~\ref{lemma:closure:4}, let 
$C\in\Lambda_\a$, namely  $\a\forcing C$;
by induction on $\size{C}$, we show that   $C\in\Clo{\Lambdas_\a}$.  If $C\in\PV$, then $\a\forcings C$,
  hence $C\in\Lambdas_\a$, which implies $C\in\Clo{\Lambdas_\a}$.  Let
  $C=A\imp B$. If $\a\nforcing A$, then $\a\forcings C$ and, as
  above, $C\in\Clo{\Lambdas_\a}$.  If $\a\forcing A$, then
  $\a\forcing B$; by induction hypothesis, $B\in
  \Clo{\Lambdas_\a}$, hence $A\imp B\in \Clo{\Lambdas_\a}$. The cases
  $C=A\land B$ and $C=A\lor B$ easily follow by the induction
  hypothesis.
 Point~\ref{lemma:closure:5} immediately follows by~\ref{lemma:closure:4} 
and by~\ref{propClo:6}.

To sum up, by~\ref{lemma:closure:1} and~\ref{lemma:closure:2} we get
$\Lambda_\a=\Clo{\Lambda_\a}$; by~\ref{lemma:closure:3} and~\ref{lemma:closure:5},
$\Clo{\Lambda_\a}=\Clo{\Lambdas_\a}$.
\end{proof}

Now, let us come to the proof of Lemma~\ref{lemma:minMod}.
\begin{proof}[(Lemma~\ref{lemma:minMod})]
  We define the derivation $\derfrjg{\genericArrow}{\a}{C}$, where
  $\genericArrow\in\{\seqfrjArrow,\seqfrjiArrow\}$, using the
  following mutual induction hypothesis (mirroring the order  used
  to define the  sequents $\sigmairr{\a}{C}$ and
  $\sigmareg{\a}{C}$ in Ex.~\ref{ex:STMinH} and~\ref{ex:DUMnonMIN}):
  \begin{enumerate}[label=(IH\arabic*),ref=(IH\arabic*)]
  \item \label{lemma:minMod:IH1} a main induction on $\height{\a}$;
    
  \item \label{lemma:minMod:IH2} a secondary induction on
    $\tpm{\genericArrow}$, where $\tpm{\seqfrjiArrow}=0$ and
    $\tpm{\seqfrjArrow}=1$;
    
  \item \label{lemma:minMod:IH3} a third induction on $\size{C}$.
\end{enumerate}  

\noindent
Let us introduce the following notations:
\[
\begin{array}{ll}
  \bGat\;=\;\Sfl{G}\cap\PV
  \qquad
  &\bGimp\;=\;\Sfl{G}\cap \Fmimp
  \qquad
  \bG\;=\;\bGat\cup\bGimp
  \\[1ex]
  \Lambdasa_\a\;=\;\Lambdas_\a\cap\PV
  &\Lambdasi_\a\;=\;\Lambdas_\a\cap\Fmimp
\end{array}
\]
We proceed by a case analysis on $C$ and
$\derfrjg{\genericArrow}{\a}{C}$.  We point out that, since
$\a\nforcing C$, then  $C\not\in\Lambda_\a$ and $C\not\in\Lambdas_\a$.

\begin{itemize}
\item  Case $C\in\Prime$, definition of $\derfrji{\a}{C}$.
\end{itemize}

\noindent
We set:
\[
\derfrji{\a}{C}\quad=\quad
\AXC{}
\RightLabel{$\ruleAXI$}
\UIC{$\sigmairr{\a}{C}\;=\;  \seqfrji{}{\bGat\setminus\{C\},\bGimp}{C}  $}
\DP
\]
Since $\Rn{ \derfrji{\a}{C} } = -1$  and $\height{\a}\geq 0$, 
we have $\Rn{ \derfrji{\a}{C} } < \height{\a}$, hence~\ref{lemma:minMod:1} holds.
Since $C\not\in\Lambdas_\a$,
Point~\ref{lemma:minMod:2} immediately follows.

\begin{itemize}
\item  Case $C\in\Prime$, definition of  $\derfrjr{\a}{C}$.
\end{itemize}

\noindent
If $\Lambdasi_\a$ is empty, namely $\Lambdas_\a=\Lambdasa_\a$, we set:
\[
\derfrjr{\a}{C}\quad=\quad
\AXC{}
\RightLabel{$\ruleAXR$}
\UIC{$\sigmareg{\a}{C}\;=\; \seqfrj{\bGat\setminus\{C\}}{C} $}
\DP
\]
Since $\Rn{ \derfrjr{\a}{C} } = 0$, we have
$\Rn{ \derfrjr{\a}{C} } \leq \height{\a}$, which
proves~\ref{lemma:minMod:3}.  Point~\ref{lemma:minMod:4} holds for
$\b=\a$ since $C\not\in\Lambdas_\a$.  If $\Lambdasi_\a$ is nonempty
let
\[
\Upsilon\;=\;
\{\,Y~|~Y\imp Z\in \Lambdasi_\a\,\}
\;=\;
\{A_1,\dots,A_n\} \quad (n\geq 1)\ .
\]
By definition of $\Lambdasi_\a$, we have  $\a\nforcing A_j$, for every $A_j\in\Upsilon$.
By~\ref{lemma:minMod:IH2}, 
for every $1\leq j\leq n$, there is an
$\FRJof{G}$-derivation  $\derfrji{\a}{A_j}$ of 
$\sigmairr{\a}{A_j}= \seqfrji{\Sigma_j}{\Theta_j}{A_j}$
such that:

\begin{enumerate}[label=(P\arabic*),ref=(P\arabic*)]
\item \label{lemma:minMod:P1}
$\Rn{\derfrji{\a}{A_j}}\,<\, \height{\a}$;

\item \label{lemma:minMod:P2}
$\Sigma_j\,\subseteq\, \Lambdas_\a\,\subseteq \,\Sigma_j\cup\Theta_j$.

\end{enumerate}

\noindent
We stress that the use of~\ref{lemma:minMod:IH2} is sound since $\tpm{\seqfrjiArrow}<\tpm{\seqfrjArrow}$.
Let $\Sigma_j=\Sigat_j\cup\Sigimp_j$ and
$\Theta_j=\That_j\cup\Thimp_j$.
We prove that $\sigmairr{\a}{A_1},\dots,
\sigmairr{\a}{A_n}$ satisfy the following properties,
for every $1\leq j\leq n$:

\begin{enumerate}[label=(\alph*),ref=(\alph*)]
\item\label{lemma:minMod:join1}
$\Sigma_i\,\subseteq\, \Sigma_j\cup\Theta_j$, for every $i\neq j$;

\item \label{lemma:minMod:join2}
$Y\imp Z\in\Sigimp_j$ implies $Y\in\Upsilon$;

\item \label{lemma:minMod:join3}
$C\not\in\Sigat_j$. 
\end{enumerate}

\noindent
Point~\ref{lemma:minMod:join1} follows by~\ref{lemma:minMod:P2}.
Point~\ref{lemma:minMod:join2}  follows
by~\ref{lemma:minMod:P2} and the definition of $\Upsilon$.
Point~\ref{lemma:minMod:join3} follows by the fact that
$C\not\in\Lambdas_\a$ and~\ref{lemma:minMod:P2}.
By~\ref{lemma:minMod:join1}--\ref{lemma:minMod:join3}, we can
apply the rule $\ruleJOINA$ with premises
 $\sigmairr{\a}{A_1},\dots, \sigmairr{\a}{A_n}$
and build the  $\FRJof{G}$-derivation $\derfrjr{\a}{C}$ displayed below
($\Sigat$,  $\Sigimp$,  $\That$,  $\Thimp$ are defined as in Fig.~\ref{fig:FRJ}):
\[
 \begin{array}{c}
\derfrjr{\a}{C}\quad =\quad
\AXC{\dots}
\AXC{$\derfrji{\a}{A_j}$}
\noLine
   \UIC{$\seqfrji{\Sigat_j,\Sigimp_j}{\That_j,\Thimp_j}{A_j}$}
\AXC{$\dots$}
  \insertBetweenHyps{\hskip -4pt}
\RightLabel{$\ruleJOINA$}
\TIC{$\sigmareg{\a}{C}\;=\;\seqfrj{\G}{C}$}
 \DP
\quad 
 \begin{minipage}{20em}
$j=1\dots n$
\end{minipage}
\\[5ex]
\G\;=\;\Sigat\cup(\That\setminus\{C\})\cup\Sigimp\cup\Thimp
\end{array}
\]
Note that, by the definition of $\Upsilon$, the application of
$\ruleJOINA$ satisfies the restriction~\ref{PS3} stated in
Sec.~\ref{sec:FRJ}.  By definition, $\Rn{ \derfrjr{\a}{C} }= m +1$,
where $m$ is the maximum among $\Rn{\derfrji{\a}{A_1}}$, \dots,
$\Rn{\derfrji{\a}{A_n}}$.  By~\ref{lemma:minMod:P1}, $m <\height{\a}$,
hence $ \Rn{ \derfrjr{\a}{C} }\leq \height{\a}$, and this
proves~\ref{lemma:minMod:3}.  We show that
$\Lambdas_\a\subseteq\Gamma$, and this proves~\ref{lemma:minMod:4}
(where $\b=\a$).  If, for some $j\in\{1,\dots,n\}$,
$\Lambdas_\a\subseteq\Sigma_j$, then
$\Lambdas_\a\subseteq \Sigat\cup\Sigimp$, hence
$\Lambdas_\a\subseteq\Gamma$.  Otherwise, by~\ref{lemma:minMod:P2},
$\Lambdas_\a\subseteq\bigcap_{1\leq j\leq n} \Theta_j$, which implies:

\begin{enumerate}[label=(\alph*),ref=(\alph*),start=4]
\item\label{lemma:minMod:d}
  $\Lambdasa_\a\;\subseteq\;\bigcap_{1\leq j\leq n} \That_j$;

\item \label{lemma:minMod:e}
$\Lambdasi_\a\;\subseteq\;\bigcap_{1\leq j\leq n}\; \Thimp_j$.

\end{enumerate}

\noindent
Since $C\not\in\Lambdasa_\a$ and $\bigcap_{1\leq j\leq n}
\That_j=\That$, by~\ref{lemma:minMod:d} we get $\Lambdasa_\a\subseteq
\That\setminus\{C\}$, hence $\Lambdasa_\a\subseteq \Gamma$.  Moreover,
since $Y\imp Z \in\Lambdasi_\a$ implies $Y\in\Upsilon$,
by~\ref{lemma:minMod:e} we get
$\Lambdasi_\a\,\subseteq\,\Restr{(\bigcap_{1\leq j\leq n}\;
  \Thimp_j)}{\Upsilon}$, namely $\Lambdasi_\a\subseteq \Thimp$, hence
$\Lambdasi_\a\subseteq \G$.  We conclude that $\Lambdas_\a=
\Lambdasa_\a\cup\Lambdasi_\a \subseteq\Gamma$.

\begin{itemize}
\item  Case $C= C_1\lor C_2$,   definition of $\derfrji{\a}{C}$.
\end{itemize}

\noindent
Since $\a\nforcing
C_1\lor C_1$, we have $\a\nforcing C_1$ and $\a\nforcing C_2$.  
By~\ref{lemma:minMod:IH3}, for every $k\in\{1,2\}$ there 
is an $\FRJof{G}$-derivation $\derfrji{\a}{C_k}$ of 
$\sigmairr{\a}{C_k}=\seqfrji{\Sigma_k}{\Theta_k}{C_k}$ such that

\begin{enumerate}[label=(Q\arabic*),ref=(Q\arabic*)]
\item \label{lemma:minMod:Q1}
$\Rn{ \derfrji{\a}{C_k}}\,<\,\height{\a}$;

\item \label{lemma:minMod:Q2}
$\Sigma_{k}\,\subseteq\, \Lambdas_\a\,\subseteq \,\Sigma_{k}\cup\Theta_{k}$.
\end{enumerate}

\noindent
By~\ref{lemma:minMod:Q2}, we get $\Sigma_{1}\subseteq \Sigma_{2}\cup\Theta_{2}$ and
 $\Sigma_{2}\subseteq \Sigma_{1}\cup\Theta_{1}$,
hence we can set:
\[
\derfrji{\a}{C}\quad =\quad
\AXC{$\derfrji{\a}{C_1}$}
\noLine
\UIC{$\seqfrji{\Sigma_1}{\Theta_1}{C_1}$}
\AXC{$\derfrji{\a}{C_2}$}
\noLine
\UIC{$\seqfrji{\Sigma_2}{\Theta_2}{C_2}$}
\RightLabel{$\lor$}
\BIC{$\sigmairr{\a}{C}\;=\;\seqfrji
  {\underbrace{\Sigma_1,\Sigma_2}_\Sigma}{\underbrace{\Theta_1\cap\Theta_2}_\Theta}{C_1\lor C_2}$}
 \DP
\]
Since $\Rn{ \derfrji{\a}{C}}$ is the maximum between
$\Rn{ \derfrji{\a}{C_1}}$  and $\Rn{ \derfrji{\a}{C_2}}$,
by~\ref{lemma:minMod:Q1} we get   $\Rn{ \derfrji{\a}{C}}  <\height{\a}$,
and this proves~\ref{lemma:minMod:1}.
By~\ref{lemma:minMod:Q2}, we immediately get  $\Sigma\subseteq \Lambdas_\a$.
Moreover, by~\ref{lemma:minMod:Q2} it follows that
 $\Lambdas_\a\subseteq \Sigma_1$ or  $\Lambdas_\a\subseteq \Sigma_2$
 or  $\Lambdas_\a\subseteq\Theta_1\cap\Theta_2$,
 which implies  $\Lambdas_\a\subseteq\Sigma\cup\Theta$.
Thus, \ref{lemma:minMod:2} holds.

\begin{itemize}
\item  Case $C= C_1\lor C_2$,   definition of  $\derfrjr{\a}{C}$.
\end{itemize}

\noindent
Since $\a\nforcing C_1\lor C_1$,  we have $\a\nforcing C_1$ and $\a\nforcing C_2$.
By~\ref{lemma:minMod:IH2}, for every $k\in\{1,2\}$
there  is an $\FRJof{G}$-derivation $\derfrji{\a}{C_k}$ of 
$\sigmairr{\a}{C_k}=\seqfrji{\Sigma_k}{\Theta_k}{C_k}$ 
satisfying~\ref{lemma:minMod:Q1} and~\ref{lemma:minMod:Q2}.
Let $\Sigma_k= \Sigat_k\cup\Sigimp_k$ and $\Theta_k= \That_k\cup\Thimp_k$ and
let
\[
   \Upsilon\;=\;
   \{\,Y~|~Y\imp Z\in \Lambdasi_\a\,\}\,\cup\,\{\,C_1,\,C_2\,\}
   \;=\;
   \{A_1,\dots,A_n\} \quad (n\geq 1)\ .
\]
We  argue as in the case concerning $\derfrjr{\a}{C}$ with $C\in\Prime$.
For every $1\leq j\leq n$,  since $\a\nforcing A_j$,
by~\ref{lemma:minMod:IH2}
there is an
$\FRJof{G}$-derivation  $\derfrji{\a}{A_j}$ of 
$\sigmairr{\a}{A_j}= \seqfrji{\Sigma_j}{\Theta_j}{A_j}$
such that:

\begin{itemize}
\item  if  $A_j\in\{C_1, C_2\}$, points~\ref{lemma:minMod:Q1} and~\ref{lemma:minMod:Q2} hold;

\item otherwise,    points~\ref{lemma:minMod:P1} and~\ref{lemma:minMod:P2} hold.
\end{itemize}

\noindent
Hence, we  can build 
the  $\FRJof{G}$-derivation  $\derfrjr{\a}{C}$ as follows
($\Sigat$,  $\Sigimp$,  $\That$,  $\Thimp$ are defined as in Fig.~\ref{fig:FRJ}):

\[
 \begin{array}{c}
\derfrjr{\a}{C}\quad =\quad
\AXC{\dots}
\AXC{$\derfrji{\a}{A_j}$}
\noLine
\UIC{$\seqfrji{\Sigat_j,\Sigimp_j}{\That_j,\Thimp_j}{A_j}$}
\AXC{$\dots$}
  \insertBetweenHyps{\hskip -4pt}
\RightLabel{$\ruleJOINO$}
\TIC{$\sigmareg{\a}{C}\;=\;\seqfrj{\G}{C_1\lor C_2}$}
 \DP
\quad 
 \begin{minipage}{20em}
$j=1\dots n$
\end{minipage}
\\[5ex]
\G\;=\;\Sigat\cup\That\cup\Sigimp\cup\Thimp
\end{array}
\]
We point out that the displayed application of $\ruleJOINO$ matches
the restriction~\ref{PS4} stated in Sec.~\ref{sec:FRJ}.  Reasoning as
above, Point~\ref{lemma:minMod:3} follows by~\ref{lemma:minMod:P1}
and~\ref{lemma:minMod:Q1}, Point~\ref{lemma:minMod:4} (with $\b=\a$)
by~\ref{lemma:minMod:P2} and~\ref{lemma:minMod:Q2},

\begin{itemize}
\item  Case $C= C_1\land C_2$, definition of   $\derfrji{\a}{C}$.
\end{itemize}

\noindent
Since $\a\nforcing C_1\land C_2$, 
there exists $k\in\{1,2\}$ such that
$\a\nforcing C_k$.
By~\ref{lemma:minMod:IH3}, there exists an  $\FRJof{G}$-derivation $\derfrji{\a}{C_k}$ of 
$\sigmairr{\a}{C_k}=\seqfrji{\Sigma}{\Theta}{C_k}$ such that:

\begin{enumerate}[label=(R\arabic*),ref=(R\arabic*)]

\item \label{lemma:minMod:R1}
$\Rn{\derfrji{\a}{C_k}  }\,<\, \height{\a}$;

\item \label{lemma:minMod:R2}
$\Sigma\,\subseteq\, \Lambdas_\a\,\subseteq\, \Sigma\cup\Theta$.

\end{enumerate}

\noindent
We can build the  $\FRJof{G}$-derivation:
\[
\derfrji{\a}{C}\;=
\AXC{$\derfrji{\a}{C_k}$}
\noLine
\UIC{$\seqfrji{\Sigma}{\Theta}{C_k}$}
\RightLabel{$\land$}
\UIC{$\sigmairr{\a}{C}\,=\,\seqfrji{\Sigma}{\Theta}{C_1\land C_2}$}
\DP  
\]
Since $\Rn{\derfrji{\a}{C}} = \Rn{\derfrji{\a}{C_k}}$,
by~\ref{lemma:minMod:R1} we get $\Rn{\derfrji{\a}{C}} < \height{\a}$,
and this proves~\ref{lemma:minMod:1}.
Point~\ref{lemma:minMod:2} immediately follows by~\ref{lemma:minMod:R2}.

\begin{itemize}
\item  Case $C= C_1\land C_2$, definition of   $\derfrjr{\a}{C}$.
\end{itemize}

\noindent
Since $\a\nforcing C_1\land C_2$, there exists $k\in\{1,2\}$ such that
$\a\nforcing C_k$.  By~\ref{lemma:minMod:IH2}, there exists an
$\FRJof{G}$-derivation $\derfrjr{\a}{C_k}$ of
$\sigmareg{\a}{C_k}=\seqfrj{\G}{C_k}$ such that:

\begin{enumerate}[label=(R\arabic*),ref=(R\arabic*),start=3]

\item \label{lemma:minMod:R3}
$\Rn{\derfrjr{\a}{C_k}  }\,\leq\, \height{\a}$;

\item \label{lemma:minMod:R4}
There is $\b\in P$ such that $\a\leq \b$ and 
$\Lambdas_\b\,\subseteq\, \G$.

\end{enumerate}

\noindent
We can build the  $\FRJof{G}$-derivation:
\[
\derfrjr{\a}{C}\;=
\AXC{$\derfrjr{\a}{C_k}$}
\noLine
\UIC{$\seqfrj{\G}{C_k}$}
\RightLabel{$\land$}
\UIC{$\sigmareg{\a}{C}\,=\,\seqfrj{\G}{C_1\land C_2}$}
\DP  
\]
Since $\Rn{\derfrjr{\a}{C}} = \Rn{\derfrjr{\a}{C_k}}$,
by~\ref{lemma:minMod:R3} we get $\Rn{\derfrjr{\a}{C}} \leq \height{\a}$,
and this proves~\ref{lemma:minMod:3}.
Point~\ref{lemma:minMod:4} immediately follows by~\ref{lemma:minMod:R4}.

\begin{itemize}
\item  Case $C= A \imp B$,  definition of  $\derfrji{\a}{C}$. 
\end{itemize}

\noindent
Since $\a\nforcing A\imp B$,
there is $\eta\in P$ such that
$\a\leq \eta$ and
$\eta\forcing A$ and $\eta\nforcing B$.
Without loss of generality, we assume that, for every $\d\in P$ such
that $\a\leq \d < \eta$, we have $\d\nforcing A$.
Since  $\a\leq \eta$, it holds that  $\a\nforcing B$.
By~\ref{lemma:minMod:IH3},
there exists an $\FRJof{G}$-derivation
$\derfrji{\a}{B}$ of $\sigmairr{\a}{B}=\seqfrji{\Sigma_1}{\Theta_1}{B}$
such that:

\begin{enumerate}[label=(S\arabic*),ref=(S\arabic*)]

\item \label{lemma:minMod:S1}
$\Rn{\derfrji{\a}{B}  }\,<\, \height{\a}$;

\item \label{lemma:minMod:S2}
$\Sigma_1\,\subseteq\, \Lambdas_\a\,\subseteq\, \Sigma_1\cup\Theta_1$.

\end{enumerate}

\noindent
If $\eta=\a$, then $\a\forcing A$, hence $A\in\Lambda_\a$, which
implies, by Lemma~\ref{lemma:closure}, $A\in\Clo{\Lambdas_\a}$.  Let
$\Lambda$ be a (possibly empty) minimum subset of $\Lambdas_\a\setminus\Sigma_1$ such
that $A\in\Clo{\Sigma_1\cup\Lambda}$ (namely:
$\Lambda'\subsetneq\Lambda$ implies $A\not\in\Clo{\Sigma_1\cup
  \Lambda'}$); note that $\Lambda\subseteq \Theta_1$.
We can build
the  $\FRJof{G}$-derivation  $\derfrji{\a}{C}$ as follows,
where  rule $\ruleIMPi$ shifts
the set $\Lambda$   to the left of semicolon:
\[
  \begin{array}{c}
  \derfrji{\a}{C}\;=
\AXC{$ \derfrji{\a}{B} $}
\noLine
\UIC{$\seqfrji{\Sigma_1\;}
  {\;\overbrace{\Theta_2,\Lambda}^{\Theta_1}}{B}$}
    \RightLabel{$\ruleIMPi$}
    \UIC{$\sigmairr{\a}{C}\;=\;
      \seqfrji{\underbrace{\Sigma_1,\Lambda}_\Sigma\;}{\;\Theta_2}{A\imp B}$}
   \DP
    \quad
    \begin{minipage}{15em}
      $\Theta_2=\Theta_1\setminus\Lambda$
      \\[.5ex]
      $A\in\Clo{\Sigma_1\cup\Lambda}$
\\[.5ex]
$\Lambda'\subsetneq\Lambda$ implies
$A\not\in\Clo{\Sigma_1\cup \Lambda'}$
    \end{minipage}
 \end{array}
\]
We point out that the choice of $\Lambda$ complies with the
restriction~\ref{PS1} stated in Sec.~\ref{sec:FRJ}.  Since
$\Sigma_1\subseteq \Lambdas_\a$ (see~\ref{lemma:minMod:S2}) and
$\Lambda\subseteq\Lambdas_\a$, we get
$\Sigma_1\cup\Lambda\subseteq \Lambdas_\a$, namely
$\Sigma\subseteq\Lambdas_\a$.  Moreover, since
$\Lambdas_\a\subseteq\Sigma_1\cup\Theta_1$ (see~\ref{lemma:minMod:S2})
and $\Sigma_1\cup\Theta_1=\Sigma\cup\Theta_2$, we get
$\Lambdas_\a\subseteq\Sigma\cup\Theta_2$, and this concludes the proof
of~\ref{lemma:minMod:2}.  Since
$\Rn{\derfrji{\a}{C}}= \Rn{\derfrji{\a}{B}}$, by~\ref{lemma:minMod:S1}
we get $\Rn{\derfrji{\a}{C}} < \height{\a}$, which
proves~\ref{lemma:minMod:1}.

If $\a < \eta$, then $\height{\eta} < \height{\a}$.  By the choice of
$\eta$, we have $\a\nforcing A$.  Since $\eta\nforcing B$,
by~\ref{lemma:minMod:IH1} there is an $\FRJof{G}$-derivation
$\derfrjr{\eta}{B}$ of $\sigmareg{\eta}{B}=\seqfrj{\Gamma}{B}$ such that:

\begin{enumerate}[label=(S\arabic*),ref=(S\arabic*),start=3]

\item \label{lemma:minMod:S3}
$\Rn{\derfrjr{\eta}{B}  } \,\leq\, \height{\eta}$;

\item \label{lemma:minMod:S4}
  There exists $\b\in P$ such that 
$\eta\leq  \b$  and $\Lambdas_\b\,\subseteq\,\G$.
\end{enumerate}

\noindent
Since $\eta \forcing A$ and $\eta\leq \b$, we get $\b \forcing A$,
hence $A\in\Lambda_\b$.  By Lemma~\ref{lemma:closure},
$A\in\Clo{\Lambdas_\b}$ hence, by~\ref{lemma:minMod:S4}
and~\ref{propClo:4}, $A\in\Clo{\Gamma}$.  Since $\a\nforcing A$, we
have $A\not\in\Lambda_\a$ hence, by Lemma~\ref{lemma:closure},
$A\not\in\Clo{\Lambdas_\a}$.  Since $\a <\eta\leq \b$, we have
$\Lambdas_\a\subseteq \Lambda_\b$.  By Lemma~\ref{lemma:closure}, we
get $\Lambda_\b=\Clo{\Lambdas_\b}$, thus
$\Lambdas_\a\subseteq \Clo{\Lambdas_\b}$.  By~\ref{lemma:minMod:S4}
and~\ref{propClo:4}, $\Clo{\Lambdas_\b}\subseteq\Clo{\Gamma}$, hence
$\Lambdas_\a\subseteq\Clo{\Gamma}$.  To sum up:
\begin{itemize}
\item  $\Lambdas_\a\;\subseteq\;\Clo{\Gamma}\cap\bG$
and $A\in\, \Clo{\Gamma}\setminus \Clo{\Lambdas_\a}$.
\end{itemize}

\noindent
Let $\Theta$ be  a maximum extension 
of $\Lambdas_\a$ such that $\Lambdas_\a\subseteq\Theta \subseteq\Clo{\Gamma}\cap\bG$
and $A\not\in\Clo{\Theta}$ (namely: $\Theta\subsetneq \Theta'\subseteq\Clo{\Gamma}\cap\bG$
implies $A\in \Clo{\Theta'})$.
We can build the $\FRJof{G}$-derivation:
\[
\derfrji{\a}{C}\;=\;
\AXC{$\derfrjr{\eta}{B}$}
\noLine
\UIC{$\seqfrj{\G}{B}$}
\RightLabel{$\ruleIMPni$}
\UIC{$\sigmairr{\a}{C}\;=\;\seqfrji{}{\Theta}{A\imp B}$}
\DP  
\quad 
\begin{minipage}{20em}
  $\Lambdas_\a\  \;\subseteq\;   \Theta\;\subseteq\; \Clo{\Gamma} \,\cap\,\bG$
\\[.5ex]
$A\,\in\, \Clo{\G}\setminus\Clo{\Theta}$   
\\[.5ex]
 $\Theta\subsetneq \Theta'\subseteq\Clo{\Gamma}\cap\bG$
implies $A\in \Clo{\Theta'}$
\end{minipage}
\]  
We point out that the choice of $\Theta$ matches the
restriction~\ref{PS2} stated in Sec.~\ref{sec:FRJ}.  We have
$\Rn{ \derfrji{\a}{C}} = \Rn{ \derfrjr{\eta}{B}}$;
by~\ref{lemma:minMod:S3} we get
$\Rn{\derfrji{\a}{C} } \leq \height{\eta}$, hence
$\Rn{\derfrji{\a}{C} } < \height{\a}$, which
proves~\ref{lemma:minMod:1}.  Since $\Lambdas_\a\subseteq\Theta$,
\ref{lemma:minMod:2} holds.

\begin{itemize}
\item  Case $C= A \imp B$,  definition of  $\derfrjr{\a}{C}$.
\end{itemize}

\noindent
Since $\a\nforcing A\imp B$, there is $\eta\in P$ such that $\a\leq
\eta$ and $\eta\forcing A$ and $\eta\nforcing B$.  Since
$\eta\nforcing B$, by induction hypothesis~\ref{lemma:minMod:IH1} if
$\a < \eta$ and~\ref{lemma:minMod:IH3} if
$\a=\eta$, there is an $\FRJof{G}$-derivation $\derfrjr{\eta}{B}$ of
$\sigmareg{\eta}{B}=\seqfrj{\Gamma}{B}$ satisfying
points~\ref{lemma:minMod:S3} and~\ref{lemma:minMod:S4}; note that
$\a\leq\eta\leq \b$.  Since $\eta\leq\b$, we have $\b\forcing A$,
namely $A\in\Lambda_\b$.  By Lemma~\ref{lemma:closure},
$A\in\Clo{\Lambdas_\b}$, which implies, by~\ref{lemma:minMod:S4}
and~\ref{propClo:4}, $A\in\Clo{\G}$.  Thus, we can build the
$\FRJof{G}$-derivation:
\[
\derfrjr{\a}{C}\;=\;
\AXC{$\derfrjr{\eta}{B}$}
\noLine
\UIC{$\seqfrj{\G}{B}$}
\RightLabel{$\ruleIMPi$}
\UIC{$\sigmareg{\a}{C}\;=\;\seqfrj{\G}{A\imp B}$}
\DP  
\qquad 
\begin{minipage}{14em}
$A\,\in\, \Clo{\G}$   
\end{minipage}
\]  
Since $\Rn{\derfrjr{\a}{C}}=\Rn{\derfrjr{\eta}{B}}$
and $\height{\eta}\leq \height{\a}$ (indeed, $\a\leq \eta$),
by~\ref{lemma:minMod:S3} we get  $\Rn{\derfrjr{\a}{C}}\leq\height{\a}$,
which proves~\ref{lemma:minMod:3}.
Point~\ref{lemma:minMod:4} immediately follows by~\ref{lemma:minMod:S4},
being $\a\leq\b$.
\end{proof}






\section{Related and future work}\label{sec:rel}
\label{sec:related}

\begin{figure}[t]
  \centering
\begin{tikzpicture}
 \path[scale=.60,grow'=up,every node/.style={fill=gray!30,rounded corners},
level 1/.style = {sibling distance = 40mm},
level 2/.style = {sibling distance = 20mm},
 level 3/.style = {sibling distance = 10mm},
 edge from parent/.style={black,draw}]

node{1: }
child{
node{2: }
child{
node{5: p}
}
}
child{
node{3: p}
}
child{
node{4: }
}
;
\end{tikzpicture}
\hspace{4em}
\begin{tikzpicture}
\path[scale=.60,grow'=up,every node/.style={fill=gray!30,rounded corners},
level 1/.style = {sibling distance = 40mm},
level 2/.style = {sibling distance = 20mm},
 level 3/.style = {sibling distance = 10mm},
 edge from parent/.style={black,draw}]

node{1: }
child{
node{2: }
child{
node{5: }
}
child{
node{6: p}
}
}
child{
node{3: }
child{
node{7: p}
}
}
child{
node{4: }
}
;
\end{tikzpicture}
  \caption{The countermodels for  $S$  and $T$ 
    (see Ex.~\ref{ex:frjNishimura})
    built by the prover $\protect\lsj$~\protect\cite{FerFioFio:2013}.}
  \label{fig:lsj}
\end{figure}
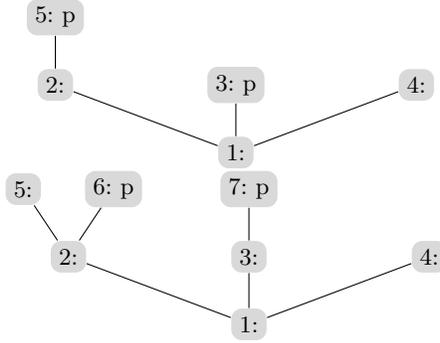

We have introduced a forward calculus $\FRJof{G}$  to  derive the
non-validity of a goal formula $G$ in $\IPL$. If
$G$ is provable in $\FRJof{G}$, from the derivation we can extract a countermodel for $G$.
Otherwise, we eventually get a saturated database  $\DBof{G}$,
which can be  exploited  to build a derivation of $G$ in the
sequent calculus $\GJ$;
accordingly, $\DBof{G}$  can be understood as a
``proof-certificate'' of the validity of $G$ in $\IPL$
(a dual remark for a forward calculus for $\IPL$ has been issued in~\cite{McLPfe:2008}).


To evaluate the potential of our approach, we have developed $\frj$, a Java
implementation of our proof-search procedure based on the
full-fledged  framework JTabWb~\cite{FerFioFio:2017}.  So far we have implemented the plain
forward strategy and the redundancy checks based on forward and
backward subsumption.  At each iteration of the main loop, $\frj$
applies all the possible instances of rules $\land$, $\lor$,
$\ruleIMPi$ and $\ruleIMPni$ involving at least a premise proved in
the last step.  To manage join rules $\ruleJOINA$ and $\ruleJOINO$, \frj maintains a list of
\emph{join-compatible}  sets $\Jcal_k$, namely $\Jcal_k$
is a set of irregular sequents matching 
the  side conditions~\ref{sc:J1} and~\ref{sc:J2}
(see Sec.~\ref{sec:FRJ}).  At each iteration the list is updated resting on the
set $\Ical$ of irregular sequents proved in the last iteration.  In
particular, each join-compatible set $\Jcal_k$ is possibly extended with
elements of $\Ical$ and the new join-compatible sets issued from~$\Ical$
are added.  For every join-compatible set $\Jcal_k$, every possible join
rule having premises $\Jcal_k$ is applied.  We also exploit backward
subsumption to optimize the implementation of join rules: whenever
backward subsumption is detected, every subsumed irregular sequent
occurring in a join-compatible set is replaced by the subsuming one.
Finally \frj, executed with the \texttt{-gbu} option, extracts from
the saturated database generated by a failed proof-search in
$\FRJof{G}$, the $\GBUof{G}$-derivation of $G$.  When executed with
the \texttt{-latex} option, the prover yields the \LaTeX-rendering of
the generated derivations and countermodels.
Most of the examples presented in the papers have been
obtained by running $\frj$.

As discussed in Sec.~\ref{sec:FRJ}, whenever we search for an
 $\FRJof{G}$-derivation of $G$, we are also trying to build a
 countermodel for $G$ in a backward style, starting from the final
 worlds down to the root.
Thus, our countermodel construction technique
is dual    to  standard proof-search procedures such
as~\cite{AvelloneFM:15,FerFioFio:2010lpar,FerFioFio:2013,FerFioFio:2015tocl,GorPos:2010,Negri:14,PinDyc:95},
where proofs and model are searched bottom-up,
starting from the goal and backward applying the rules of the calculus.
One of the advantages of forward
vs.~backward reasoning is that, provided one implements suitable
redundancy checks, derivations are more concise since sequents are
reused and not duplicated. As a consequence, the obtained models are
in general compact and do not contain redundant worlds.  For instance,
the models in Figs.~\ref{fig:countST} and~\ref{fig:countAST} are the
minimal countermodels for the formulas $S$ and $T$ respectively.  The
model in Fig.~\ref{fig:countAST} is particularly significant since it
is not a tree, hence it cannot be obtained by the mentioned standard
proof-search procedures, which only generate tree-shaped models.
For instance,
let us consider the prover
 $\lsj$, implementing   in JTabWb  
a backward proof-search procedure  for
the calculi presented in~\cite{FerFioFio:2013}.
 For unprovable formulas, $\lsj$  yields countermodels of
 minimal height; however, the obtained  models are always trees,
 hence they might contain redundant worlds.
In  Fig.~\ref{fig:lsj} we show  the countermodels 
built by  $\lsj$   for the formulas $S$ and $T$;
 in the left-hand side model, the worlds 3 and 5 can be overlapped;
  in the  right-hand side model,  world 6 can be merged with   7  
 and 4 with 5.
As another significant example, let us consider the one-variable formulas $N_i$
of the Nishimura family~\cite{ChaZak:97}, which are not valid in
$\IPL$:
\[
\begin{array}{lll}
  N_1\: = \:p   &\qquad&N_{2n+3}\:=\: N_{2n+1}\,\lor\, N_{2n+2}
  \\[1ex]
  N_2\:=\:\neg p&&N_{2n+4}\:= \:N_{2n+3}\,\imp\, N_{2n+1}
\qquad n\geq 0
\end{array}
\]
\begin{figure}[t]
  \centering
  \includegraphics[scale=.42,clip=false]{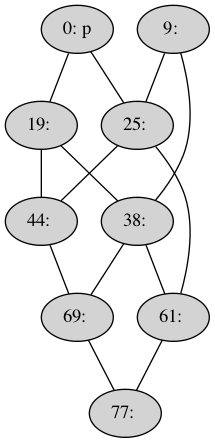}
  \caption{Countermodel for $N_{17}$}
  \label{fig:N17}
\end{figure}

\noindent
We recall that the aforementioned formulas  $S$ and $T$  are equivalent 
to  $N_{10}$ and $N_9$ respectively.
For   formulas $N_j$,  \frj generates the standard ``tower-like'' minimum
countermodel~\cite{ChaZak:97}.
In  Fig.~\ref{fig:N17} we display 
the countermodel  obtained for $N_{17}$;
in contrast, 
$\lsj$ fails to build a countermodel
for such a  formula.

A method to test the validity of intuitionistic formulas,
based on the  construction of small countermodels,
is presented in~\cite{GorTho:2012,Thomson:2014}.
Let  $G$ be the goal formula and 
let us call \emph{atoms} the subformulas
of $G$ which are  propositional variables or $\imp$-formulas
(note that these are the same kind of formulas handled by $\FRJof{G}$).
The construction of a countermodel  for $G$ is based on a fixpoint procedure.
The initial configuration is  the structure  $\K_0=\stru{P_0,\leq,w_0,V}$,
where $P_0$ is the power set of the set of atoms, 
$\leq$ is the subset relation,
$w_0$ is the empty set
(hence, $\stru{P_0,\leq,w_0}$ is a poset with minimum element
$w_0$)
and $V$ is a function mapping  $w\in P_0$ to the set  $w\cap\PV$.
By iterating a refinement procedure, consisting in  removing
elements from $P_0$, we eventually get a Kripke model
$\K_n=\stru{P_n,\leq,w_n,V}$ satisfying the following properties:

\begin{itemize}
\item for every $w\in P_n$ and every  atom $A$,
    $\K_n,w\forcing A$ iff $A\in w$;

\item for every $w\in P_n$ and every subformula $F$ of  $G$,
    $\K_n,w\forcing F$ iff $F\in\Clo{w}$;
  
\item
$G$ is not valid in $\IPL$ iff
$\K_n$ is a countermodel for $G$ (namely,   $\K_n,w_n\nforcing G$).
\end{itemize}
Thus, to decide the validity of $G$, one has to run
the process until the fixpoint $\K_n$ is reached, and then check whether
$\K_n$ is a countermodel for $G$ (e.g., by testing whether $G\in\Clo{w_n}$).
The procedure has been implemented by the system $\bddint$,
using  BDD (Binary Decision Diagram) to efficiently
represent the posets and the  operations on them.
By construction, the generated
countermodels    do not contain redundancies
(indeed, distinct worlds of $\K_n$
are separated by at least one atom).
However, the properties of the obtained  countermodels are not fully
investigated in the mentioned papers and  
$\bddint$ does not explicitly output them;
we guess that they are  close to the ones obtained with $\frj$.
We also point out that~\cite{Thomson:2014}
presents a procedure (not implemented in $\bddint$)
to get a  derivation of $G$
in a standard sequent calculus in case
the final model $\K_n$  is not a countermodel for $G$.


Finally, as a future work we plan to investigate the applicability of
our method to other logics, in particular to modal logics such as
$\mathbf{S4}$ and intermediate logics such as G\"odel-Dummett logic
characterized by linear Kripke models.


\newpage
\appendix


\section{Soundness of $\FRJof{G}$}
\label{sec:soundFRJ}

We prove that, given an $\FRJof{G}$-derivation $\Dcal$ of $G$,
$\Mod{\Dcal}$ is a countermodel for $G$.  As discussed in
Sec.~\ref{sec:FRJ}, we have to prove the soundness
property~\ref{prop:soundr}. We introduce the following relation
$\mapstois$ between irregular sequents:

\begin{itemize}
\item $\s_1\mapstois\s_2$ iff $\s_1\mapstos\s_2$ and every $\s'$ such
  that $\s_1\mapstos\s'\mapstos\s_2$ is irregular.
\end{itemize}

It is easy to check that:
\begin{enumerate}[label=(Ir\arabic*),ref=(Ir\arabic*)]
\item\label{lemma:lhs:4}
  $\seqfrji{\Sigma_1}{\Theta_1}{C_1}\,\mapstois\,
  \seqfrji{\Sigma_2}{\Theta_2}{C_2}$ implies
  $\Sigma_1\subseteq\Sigma_2$ and
  $\Sigma_2\cup\Theta_2\subseteq\Sigma_1\cup\Theta_1$.  
\end{enumerate}

The key lemma is (we recall that $\Sfm{C}$ denotes the set
$\Sf{C}\setminus\{C\}$):

\begin{lemma}[aka Lemma 3.9 of Sec.3]\label{app:lemma:soundFRJ}
  Let $\Dcal$ be an $\FRJof{G}$-derivation of $G$, let $\Mod{\Dcal}$
  be the model extracted from $\Dcal$ and $\phi$ the map associated
  with $\Dcal$.  For every sequent $\s$ occurring in $\Dcal$:

  \begin{enumerate}[label=(\roman*),ref=(\roman*)]
  \item \label{app:lemma:soundFRJ:1}
    if $\s= \seqfrj{\G}{C}$, then $\phi(\s)\forcing\G$ and
    $\phi(\s) \nforcing C$;

  \item \label{app:lemma:soundFRJ:2}
    if $\s= \seqfrji{\Sigma}{\Theta}{C}$, let $\s_p\in\PS{\Dcal}$ such
    that~$\s\mapsto\s_p$ and $\s_p\forcing\Sigma\cap\Sfm{C}$; then
    $\s_p\nforcing C$.
  \end{enumerate}  
\end{lemma}

\begin{proof}
  We prove the assertions by a main induction (IH1) on the height
  $\height{\s}$ of $\s$ in $\Dcal$.  Let
  $\Mod{\Dcal}=\stru{\PS{\Dcal},\leq,\rho,V}$ and let $\Rcal$ be the
  rule applied to get $\s$; we proceed by a case analysis on $\Rcal$.
  In discussing the cases, we use the notation in Fig.~\ref{fig:FRJ}.
  We also recall that
\[
\bGat\;=\;\Sfl{G}\cap\PV
\qquad
\bGimp=\Sfl{G}\cap \Fmimp
\qquad
\bG=\bGat\cup\bGimp
\]


  \begin{itemize}
  \item $\Rcal=\ruleAXR$.
  \end{itemize}

\noindent
We have  $\s=\seqfrj{\bGat\setminus\{C\}}{C}$, with $C\in\Prime$.
Since $\phi(\s)=\s$ and $V(\s)=\bGat\setminus\{C\}$,
\ref{app:lemma:soundFRJ:1} immediately follows.

  \begin{itemize}
  \item $\Rcal=\ruleAXI$.
  \end{itemize}

\noindent
We have
  $\s=\seqfrji{}{\bGat\setminus\{C\},\bGimp}{C}$,  with $C\in\Prime$.
  Let $\Gat=V(\s_p)$. Since $\s\mapsto\s_p$,
  by Lemma~\ref{lemma:lhs}\ref{lemma:lhs:3}
  and~\ref{propClo:5} we get $\Gat\subseteq \bGat\setminus\{C\}$.
 This implies  $C\not\in\Gat$, hence $\s_p\nforcing
  C$,  which proves~\ref{app:lemma:soundFRJ:2}.
 \begin{itemize}
  \item $\Rcal=\land$.
  \end{itemize}

\noindent
Both~\ref{app:lemma:soundFRJ:1} and~\ref{app:lemma:soundFRJ:2}
easily follow by~(IH1). 

 \begin{itemize}
  \item $\Rcal=\lor$.
  \end{itemize}

\noindent
We have
  \[
  \begin{array}{c}
    \AXC{$ \s_{1}\;=\;
      \seqfrji{\Sigma_{1}}{\Theta_{1}}{C_{1}}$}    
    \AXC{$\s_{2}\;=\;
      \seqfrji{\Sigma_{2}}{\Theta_{2}}{C_{2}}$}        
    \RightLabel{$\lor$}
    \BIC{$ \s\;=\;
      \seqfrji{\underbrace{\Sigma_{1},\Sigma_{2}}_{\Sigma}\;}{\;\underbrace{\Theta_{1}\cap\Theta_{2}}_{\Theta}}{C_{1}\lor C_{2}}$} 
    \DP
    \qquad
    \begin{minipage}{10em}
      $\Sigma_1\;\subseteq\;\Sigma_2\cup\Theta_2$\\[.5ex]
      $\Sigma_2\;\subseteq\;\Sigma_1\cup\Theta_1$
    \end{minipage}
  \end{array}
  \]
  By hypothesis, $\s\mapsto \s_p$ and   
$\s_p\forcing \Sigma\cap\Sfm{C_1\lor C_2}$.
Let $k\in\{1,2\}$.  We have both $\s_k\mapsto \s_p$ (indeed,
$\s_k\mapstoz \s$ and $\s\mapsto \s_p$) and
$\s_{p}\forcing \Sigma_k\cap\Sfm{C_k}$ (indeed,
$\Sigma_k\subseteq \Sigma$ and
$\Sfm{C_k}\subseteq \Sfm{C_1\lor C_2}$).  By~(IH1) applied to $\s_k$,
we get $\s_p\nforcing C_k$.  Thus, both $\s_p\nforcing C_1$ and
$\s_p\nforcing C_2$, hence $\s_p\nforcing C_1\lor C_2$ which
proves~\ref{app:lemma:soundFRJ:2}.


\begin{itemize}
\item  $\Rcal =\,\ruleIMPi$.
\end{itemize}

\noindent
If $\s$ is regular, we have
\[
\AXC{$\s_1\;=\; \seqfrj{\G}{B}$}
\RightLabel{$\ruleIMPi$}
\UIC{$\s\;=\;\seqfrj{\Gamma}{A\imp B}$}
\DP
    \qquad 
 \begin{minipage}{8em}
$A\,\in\, \Clo{\G}$    
\end{minipage}
\]
By~(IH1) applied to $\s_1$, it holds that $\phi(\s_1)\forcing \G$ and
$\phi(\s_1)\nforcing B$.  Since $A\in\ \Clo{\G}$, by~\ref{propClo:1}
we get $\phi(\s_1)\forcing A$.  Since $\phi(\s)=\phi(\s_1)$, we get
$\phi(\s)\nforcing A\imp B$, which implies~\ref{app:lemma:soundFRJ:1}.

Let $\s$ be irregular
and let us assume:
\[
  \begin{array}{c}
\AXC{$\seqfrji{\s_1\;=\;\Sigma_1}{\Theta,\Lambda}{B}$}
\RightLabel{$\ruleIMPi$}
    \UIC{$\s\;=\;\seqfrji{\underbrace{\Sigma_1,\Lambda}_{\Sigma}\;}{\Theta}{A\imp B}$}
\DP
    \qquad 
    \begin{minipage}{14em}
      $A\in\Clo{\Sigma}$
\\[.5ex]    
 $\Sigma_A\;=\;\Sigma\,\cap\, \Sf{A}$
  \end{minipage}
\end{array}
\]
By hypothesis 
$\s\mapsto \s_p$ and
$\s_p\forcing \Sigma\cap\Sfm{A\imp B}$.
Thus,
$\s_p\forcing \Sigma_1\cap\Sfm{B}$
and, since  $\Sf{A}\subseteq\Sfm{A \imp B}$,
we get $\s_p\forcing \Sigma_A$.
By hypothesis $\s\mapsto \s_p$, hence
$\s_1\mapsto \s_p$;  we can apply~(IH1) 
to $\s_1$ and infer  $\s_p\nforcing B$.
Since $A\in\Clo{\Sigma}$, 
by~\ref{propClo:2}  we get   $A\in\Clo{\Sigma_A}$.
By the fact that $\s_p\forcing \Sigma_A$ and~\ref{propClo:1},  
we get $\s_p\forcing A$.
We conclude  $\s_p\nforcing A\imp B$, which proves~\ref{app:lemma:soundFRJ:2}.


\begin{itemize}
\item  $\Rcal =\,\ruleIMPni$.
\end{itemize}

\noindent
We have:
\[
\AXC{$\s_1\;=\; \seqfrj{\G}{B}$}
\RightLabel{$\ruleIMPni$}
\UIC{$\s\;=\;\seqfrji{}{\Theta}{A\imp B}$}
\DP
    \qquad 
 \begin{minipage}{8em}
$A\,\in\, \Clo{\G}$    
\end{minipage}
\]
By (IH1) applied to $\s_1$, we have
$\phi(\s_1)\forcing \G$ and $\phi(\s_1)\nforcing B$.
By~\ref{propClo:1}   $\phi(\s_1)\forcing A$.
By hypothesis $\s\mapsto \s_p$, which implies
$\s_1\mapsto \s_p$. Thus
$\s_p \leq \phi(\s_1)$, hence
$\s_p\nforcing A \imp B$,
and this proves~\ref{app:lemma:soundFRJ:2}.


\begin{itemize}
\item  $\Rcal =\,\ruleJOINA$.
\end{itemize}

\noindent
We have:
  \[
  \begin{array}{c}
    \AXC{\dots}
    \AXC{$\s_j\;=\;\seqfrji{\Sigat_j,\Sigimp_j}{\That_j,\Thimp_j}{A_j}$}
    \AXC{$\dots$}
    \insertBetweenHyps{\hskip -4pt}
    \RightLabel{$\ruleJOINA$} 
    \TIC{$\s\;=\;\seqfrj{\Sigat,\,\That\setminus\{C\},\,\Sigimp,\,\Thimp}{C}$}
    \DP
    \qquad 
    \begin{minipage}{10em}
      $j=1\dots n$
      \\[1ex]
      $C\in\Prime\setminus\Sigat$
    \end{minipage}
    \\[4ex]
    \Gat\,=\, \Sigat\,\cup\,(\That\setminus\{C\})
    \qquad
    \Gimp\,=\, \Sigimp\cup\Thimp
    \qquad
    \G\,=\,\Gat\cup\Gimp
  \end{array}
  \]
  Note that $\s\in\PS{\Dcal}$, $\phi(\s)=\s$ 
  and $V(\s)=\Gat$.  Since $C\not\in\Gat$,  we
  get:
  \begin{enumerate}[label=(P\arabic*),ref=(P\arabic*)]
  \item\label{lemma:soundFRJ:glue0} $\s\forcing \Gat$ and $\s\nforcing C$.
  \end{enumerate}
  To complete the proof of~\ref{app:lemma:soundFRJ:1}, it remains to show
  that $\s\forcing \Gimp$.  To this aim,
we show that, for every formula $H$, the following properties hold:
  
  \begin{enumerate}[start=2,label=(P\arabic*),ref=(P\arabic*)]
  \item\label{lemma:soundFRJ:glue1}
    $H\in \Gimp$ implies
    $\s\forcing H$.
    
  \item\label{lemma:soundFRJ:glue2}
$H=A_j$, with $1\leq j\leq n$, implies    $\s\nforcing H$.
  \end{enumerate}
  
  \noindent
To prove~\ref{lemma:soundFRJ:glue1} and~\ref{lemma:soundFRJ:glue2},
we introduce a  secondary induction
  hypothesis (IH2) on $\size{H}$.
  Let $H\in \Gimp$. Then, there is $k\in\{1,\dots,n\}$ such that
  $H=A_k\imp B$.  Let $\s_p\in \PS{\Dcal}$  such that $\s\leq\s_p$
  and $\s_p\forcing A_k$; we show that $\s_p \forcing B$. Let
  $\G_p=\Lhs{\s_p}$; since 
$H\in\Lhs{\s}$ and  
$\s_p\mapstos \s$ (indeed, $\s\leq\s_p$), by
  Lemma~\ref{lemma:lhs}\ref{lemma:lhs:3} we get $H\in\Clo{\G_p}$.  Since
 $\size{A_k} < \size{H}$, we can apply~(IH2)
  on~\ref{lemma:soundFRJ:glue2} and claim that $\s\nforcing A_k$,
  hence $\s_p\neq \s$, which implies $\height{\s_p}<\height{\s}$.  By (IH1) applied
  to $\s_p$, we have $\s_p\forcing \G_p$ and, by~\ref{propClo:1},
  $\s_p\forcing H$, namely $\s_p \forcing A_k\imp B$.  Since
  $ \s_p \forcing A_k$, we get $ \s_p \forcing B$; this concludes the proof 
  of~\ref{lemma:soundFRJ:glue1}.  

Let  $H=A_j$, where
$j\in\{1,\dots,n\}$.  Note
  that $\height{\s_j} < \height{\s}$, $\s_j\mapstoz \s$ and
  $\s\in\PS{\Dcal}$.  We prove $\s \nforcing A_j$   by
  applying~(IH1) on $\s_j$ (Point~\ref{app:lemma:soundFRJ:2}).
To this aim, we have to check that the condition

  \begin{enumerate}
  \item[($\dag$)]
  $\s \;\forcing\; (\Sigat_j\,\cup\,\Sigimp_j)\:\cap\:\Sfm{A_j}$
  \end{enumerate}

\noindent
holds. 
Let $K\in(\Sigat_j\cup\Sigimp_j)\cap\Sfm{A_j}$.
If $K\in\Sigat_j$, since
  $\Sigat_j\subseteq V(\s)$  we immediately get $\s\forcing K$.  
If  $K\in \Sigimp_j$, then 
  $K\in\Gimp$ and    $\size{K} < \size{A_j}$
(indeed,  by definition of $\Sfm{A_j}$, $K$ is a proper subformula of $A_j$).
We can apply~(IH2) on~\ref{lemma:soundFRJ:glue1} and we
  get $\s \forcing K$, hence~($\dag$) holds. 
This  concludes the
  proof of~\ref{lemma:soundFRJ:glue2}.  
By~\ref{lemma:soundFRJ:glue0}
  and~\ref{lemma:soundFRJ:glue1}, Point~\ref{app:lemma:soundFRJ:1} follows.


\begin{itemize}
\item  $\Rcal =\,\ruleJOINO$.
\end{itemize}

\noindent
Similar to the case~$\ruleJOINA$.
\end{proof}

By Lemma~\ref{app:lemma:soundFRJ}, we get:


\begin{lemma}\label{lemma:soundnessProperties}
Soundness properties~\ref{prop:soundr} and~\ref{prop:soundi} hold.  
\end{lemma}

\begin{proof}
  Let  $\s=\seqfrj{\G}{C}$ be provable in $\FRJof{G}$ and
let $\Dcal$ be an $\FRJof{G}$-derivation of $\s$.  By
Lemma~\ref{app:lemma:soundFRJ}\ref{app:lemma:soundFRJ:1},
the world $\phi(\s)$ of the model $\Mod{\Dcal}$ satisfies
$\phi(\s)\forcing \G$ and $\phi(\s)\nforcing C$.
Setting $\K=\Mod{\Dcal}$ and 
$\a=\phi(\s)$,  \ref{prop:soundr} holds.

We show~\ref{prop:soundi}. 
Let  $\s=\seqfrji{\Sigma}{\Theta}{C}$ be  provable in $\FRJof{G}$
and let us assume that  $\s$ can be used to prove
a regular sequent $\s_r$ in $\FRJof{G}$.
Then, there exist an irregular sequent $\s'$ and a p-sequent $\s_p$ such that:
\[
\s\,=\,\seqfrji{\Sigma}{\Theta}{C}
\quad\mapstois\quad
\s'\,=\,\seqfrji{\Sigma'}{\Theta'}{C'}
\quad\mapstorz{\ruleJOIN}\quad
\s_p\,=\,\seqfrj{\Gamma_p}{C_p}
\quad\mapstos\quad
\s_r
\]
where $\ruleJOIN$ denotes  one between the rules $\ruleJOINA$ and  $\ruleJOINO$.
By Point~\ref{lemma:lhs:4} we have
$\Sigma\subseteq\Sigma'$ and $\Sigma'\cup\Theta'\subseteq \Sigma\cup\Theta$.
By definition of  $\ruleJOIN$, we also have
$\Sigma'\subseteq\Gamma_p\subseteq\Sigma'\cup\Theta'$;
hence $\Sigma\subseteq \Gamma_p\subseteq\Sigma\cup\Theta$.
Let $\Dcal$ be the $\FRJof{G}$-derivation having root sequent $\s_p$.
By Lemma~\ref{app:lemma:soundFRJ}\ref{app:lemma:soundFRJ:1} 
$\s_p\forcing \G_p$; since    $\Sigma\cap\Sfm{C}\subseteq \Gamma_p$,
we  get   $\s_p\forcing \Sigma\cap\Sfm{C}$.
Since  $\s\mapsto \s_p$,
we can apply  Lemma~\ref{app:lemma:soundFRJ}\ref{app:lemma:soundFRJ:2} 
to infer  $\s_p\nforcing C$.
Setting $\K=\Mod{\Dcal}$, 
$\a=\s_p$ and  $\G=\G_p$, 
\ref{prop:soundi} holds.
\end{proof}




\end{document}